\documentclass{article}
\usepackage{mystyle}
\usepackage{fullpage}
\title{On the global convergence of gradient flow for wide shallow models beyond homogeneous nonlinearities}

\author{
    Romain Petit \\
    CNRS and ENS, PSL Universit\'e\\
    \texttt{romain.petit@ens.fr}
    \and
    Clarice Poon \\
    Mathematics Institute, University of Warwick\\
    \texttt{clarice.poon@warwick.ac.uk}
    \and
    Gabriel Peyré \\
    CNRS and ENS, PSL Universit\'e \\
    \texttt{gabriel.peyre@ens.fr}
}

\date{}

\begin{document}

\maketitle

\begin{abstract}
A surprising phenomenon in the training of neural networks is the ability of gradient descent to find global minimizers of the training loss despite its non-convexity. Following earlier work, we investigate this behavior for wide shallow models. Existing global convergence results primarily concern models with positively one-homogeneous nonlinearities, such as ReLU activations, and models with scalar output weights and bounded nonlinearities, such as sigmoid activations. We study a broader class of models, including multi-head attention layers and two-layer fully connected networks with bounded or asymptotically positively one-homogeneous activations and vector output weights.

Building upon \cite{chizat2018global}, we prove that, in the limit of infinitely many hidden neurons or attention heads, non-minimizers of the training loss are unstable under the mean-field gradient flow dynamics by constructing ``escape regions'' in the parameter space. Our global convergence statements are conditional in the
following sense: if the mean-field gradient flow converges for the Wasserstein $2$-distance, then its limit must be a global minimizer. We revisit the bounded nonlinearity, scalar-output setting of \cite{chizat2018global}, giving an escape region construction adapted to unbounded nonlinear parameter domains. We also propose new constructions for nonlinearities with at most linear growth under a non-degeneracy assumption and for asymptotically positively one-homogeneous nonlinearities. Finally, we show the well-posedness and stability estimates for the mean-field training dynamics under sub-Gaussian initializations.
\end{abstract}

\section{Introduction}
The training of neural networks leads to the minimization of objective functions that are highly non-convex. As a result, first-order optimization algorithms such as gradient descent may fail to reach an approximate minimizer by converging to non-optimal critical points. In certain settings, it has been observed that this behavior is not typical and that gradient descent finds global minimizers of the training loss despite its non-convexity. Among the various explanations that have been proposed, a line of work has focused on the study of the \emph{over-parameterized regime}. The landmark results of \cite{chizat2018global} (see also \cite{meiMeanFieldView2018,rotskoffParametersInteractingParticles2018}) show that, if the model is sufficiently over-parameterized and the initialization densely covers the parameter space, then (continuous-time) gradient descent can only converge to global minimizers of the training loss. In these works, the over-parameterized models under study are typically fully connected neural networks with a single hidden layer and a large number of hidden neurons. Existing results essentially cover models with positively $1$-homogeneous nonlinearities, such as ReLU activations, and models with scalar output weights and bounded nonlinearities, such as sigmoid activations. Central to these works is the use of  exact homogeneity of the activation function to provide a scaling structure that ensures instability of non-minimizers \cite{chizat2018global,wojtowytschConvergenceGradientDescent2020}.
In this work, we show this mechanism extends beyond the positively homogeneous and bounded scalar-output settings.
We develop geometric constructions that establish this instability, under
suitable hypotheses, for two-layer networks with activations such as GELU
and SiLU/Swish \cite{ramachandranSearchingActivationFunctions2017,elfwingSigmoidweightedLinearUnits2018,hendrycksGaussianErrorLinear2023}
and for a single multi-head attention layer.

\subsection{Main problem}

We consider models of the form $(1/m)\sum_{i=1}^m\Phi(w_i,\theta_i)$, where
$\Phi$ is linear in the output parameter $w$. This includes two-layer fully
connected networks and a single multi-head attention layer. For the latter, we
optimize directly over the score matrix $A=K^TQ$ and the value matrix $V$; this $(A,V)$ parameterization satisfies the differential growth and local Lipschitz bounds of \Cref{ass_wellp_wgf}, whereas the $(Q,K,V)$ parameterization need not satisfy them.

For a parameter space $\Omega=\RR^{d_w}\times\RR^{d_\theta}$ and a loss $R$, the finite-width
training objective is
\begin{equation}\label{eq_def_fm}
    F_m((w_i,\theta_i)_{1\leq i\leq m})=\textstyle R((1/m)\sum_{i=1}^m \Phi(w_i,\theta_i)).
\end{equation}
This objective can be lifted to the space $\mathcal{P}_2(\Omega)$ of probability measures on the parameter space with finite second moment, which yields 
\begin{equation}\label{eq_def_fm_cts}
\textstyle F:\mu\in\mathcal{P}_2(\Omega)\mapsto R(\int_\Omega\Phi\,d\mu).
\end{equation}
Continuous-time gradient descent for $F_m$, with the time scaling specified in \Cref{sec_wellp_wgf}, induces a Wasserstein gradient flow of $F$ through the empirical measure associated to the set of parameters. This viewpoint has been widely used to study the training dynamics of two-layer neural networks \cite{chizat2018global,meiMeanFieldView2018,rotskoffParametersInteractingParticles2018,sirignano2020mean,mei2019mean}.

This paper extends the range of models for which one can establish the implication
\[
    W_2(\mu_t,\mu_\infty)\longrightarrow0
    \quad\Longrightarrow\quad
    \mu_\infty\text{ is a global minimizer of }F,
\]
where $(\mu_t)_{t\geq 0}$ is a Wasserstein gradient flow of $F$ with $\mu_{\infty}$ a candidate limit and $W_2$ the Wasserstein $2$-distance. This implication is proved under the hypotheses stated below, including full support of the initial distribution. The existence of a $W_2$ limit remains an assumption; the new constructions rule out non-optimal limits for the classes of nonlinearities considered here. We study the unregularized objective $F$, whereas \cite{chizat2018global} also allows a possibly non-smooth regularization term.

\subsection{Main contributions}
The results are organized as follows.
 
\begin{enumerate}[label=(\roman*)]
 \item \emph{Escape-region principle.}
   We formulate an abstract criterion connecting these constructions to global optimality (\Cref{thm_glob_conv}): An \textit{escape region} contains a nonempty open set and ensures that the first variation of $F$ stays uniformly negative along characteristics under small perturbations of the limiting measure. If every non-minimizer admits such a region, full support of the initialization supplies positive mass whose second moment must grow unbounded. This excludes convergence in $W_2$ to a non-minimizer.
    \item \emph{Three mechanisms beyond homogeneity.}
    The core contribution of this work is the construction of new escape regions. First, we revisit the
    bounded-nonlinearity, scalar-output argument of \cite{chizat2018global} on an
    unbounded nonlinear-parameter domain (\Cref{prop_esc_reg_bounded_scalar}).
    Second, we develop a new stable-critical-branch mechanism for nonlinearities
    with at most linear growth; it applies to vector output weights under a
    non-degeneracy condition (\Cref{prop_esc_reg_nondegenerate}). Third, we exploit
    the geometry visible only at large parameters for asymptotically positively
    $1$-homogeneous nonlinearities
    (\Cref{prop_esc_reg_asymp_homogeneous}). These constructions use different
    invariant quantities and lead to complementary assumptions.

    \item \emph{Applications.}
    For two-layer networks, the constructions yield conditional global convergence
    results for bounded activations, asymptotically homogeneous activations, and
    activations satisfying the non-degeneracy condition
    (\Cref{prop_glob_conv_sigmoid,prop_glob_conv_gelu,prop_glob_conv_mlp_nondegenerate}).
    This covers common activations including sigmoid, tanh, GELU, and SiLU/Swish.
    We also derive conditional counterparts for a single softmax attention layer in
    the $(A,V)$ parameterization
    (\Cref{prop_glob_conv_attention,prop_glob_conv_attention_bis}) and prove a
    uniform softmax-to-hardmax convergence result (\Cref{prop_conv_attention}).

    \item \emph{Well-posedness beyond compact support.}
    We prove global existence, uniqueness, and stability of the Wasserstein
    gradient flow under sub-Gaussian initialization (\Cref{thm_wgf}). In particular,
    this includes Gaussian initializations, often used in practice.
\end{enumerate}

\subsection{Related work}\label{sec_related_work}
\paragraph{Finite width network landscapes and training dynamics.}
Away from infinite-width limits, rigorous analyses of gradient-based training have often focused on structured models. Deep linear networks already exhibit nonlinear, depth-dependent dynamics: exact gradient-flow solutions in structured settings reveal plateaus, stagewise learning, and a strong dependence on initialization \cite{saxe2014exact}. Their landscapes can nevertheless be benign---under suitable assumptions on the loss, data, and layer widths, every local minimum is global \cite{kawaguchi2016deep,laurent2018deep}---but a benign landscape does not by itself yield a uniform convergence guarantee. Indeed, gradient descent can require time exponential in depth even for a one-dimensional deep linear network with standard random initialization \cite{shamir2019exponential}, whereas, for population squared loss with isotropic inputs, polynomial-time results for symmetric positive-definite target maps use identity initialization \cite{bartlett2018gradient}. Direct analyses for finite nonlinear networks likewise rely on restrictive settings: one example treats a non-overlapping one-hidden-layer convolutional network under Gaussian inputs and weight normalization \cite{du2018gradient}, while a deep-network landscape theorem shows that almost all local minima are global under, among other conditions, squared loss, pairwise-distinct inputs, a strictly monotone analytic activation, one sufficiently wide hidden layer, and a pyramidal tail \cite{nguyen2017loss}. Conversely, spurious local minima already occur in realizable two-layer ReLU population losses with Gaussian inputs \cite{safran2018spurious}. For a two-layer ReLU teacher-student model with an identity mapping and standard Gaussian inputs, Li and Yuan prove convergence of SGD to the teacher by a two-phase analysis \cite{li2017convergence}. Under mean-field scaling, Chen et al.\ prove linear convergence of the empirical squared loss for wide two-layer networks with fixed output weights and linearly independent training inputs \cite{chen2022featurelearning}. Thus, a general global training theory for finite nonlinear networks from arbitrary initialization, without an NTK/lazy rescaling, proximity to a favorable basin, or special structural assumptions, remains elusive. This difficulty motivates our focus on shallow but very wide architectures: lifting neurons or attention heads to a probability measure makes the predictor linear in that measure and turns the training dynamics into a transport partial differential equation (PDE), or equivalently a Wasserstein gradient flow, making global geometric arguments accessible.

\paragraph{Scaling limits and the NTK regime.}
Wide neural networks admit several inequivalent scaling limits. At random initialization, fully connected networks converge to Gaussian processes as their hidden widths tend to infinity \cite{lee2018gaussian,matthews2018gaussian}. Under the standard Neural Tangent Kernel (NTK) scaling, the tangent kernel introduced in \cite{jacot2018neural} becomes deterministic and, in the lazy-training regime, remains close to its initial value; the network is then governed by its first-order linearization and training reduces to kernel gradient descent \cite{lee2019wide,chizat2019lazy,liu2020linearity}. At finite width, related analyses give non-asymptotic optimization guarantees for sufficiently wide deep ReLU networks, including fully connected, convolutional, and residual architectures \cite{allenzhu2019convergence,du2019gdglobal,zou2020gradient,chen2020overparam,nguyen2021global}. It also yields exact or efficiently computable infinite-width kernels in several architectures \cite{arora2019exact}. These results explain one mechanism by which over-parameterization regularizes the training dynamics, but the limiting features are essentially frozen. Different parameterizations, including maximal-update scalings, instead retain order-one feature evolution at infinite width \cite{yang2021tensor}; the distinction between these regimes is therefore a property of the scaling, not merely of the architecture.

\paragraph{Feature-learning and mean-field regimes.}
In the mean-field scaling of a two-layer network, each neuron contributes at order $1/m$ and moves by an order-one amount during training. The empirical distribution of neurons then converges, on finite time horizons, to a nonlinear transport PDE or equivalently a Wasserstein gradient flow \cite{meiMeanFieldView2018,rotskoffParametersInteractingParticles2018,sirignano2020mean,mei2019mean}. Because the parameter distribution itself evolves, this limit captures representation learning rather than a fixed kernel. It connects optimization to convex neural-network and Barron-space descriptions \cite{bach2017breaking,weinan2022representation}, and it can exhibit advantages over fixed kernels on structured data \cite{ghorbani2020neural}. For single-index data, mean-field analyses also describe distinct learning time scales \cite{berthier2025timescales} and give quantitative bounds on finite-width prediction errors relative to the mean-field dynamics over polynomial training horizons \cite{glasgow2025meanfield}. Long-time analyses use the geometry of the first variation and sufficiently rich initial support to rule out convergence to non-optimal measures \cite{chizat2018global,wojtowytschConvergenceGradientDescent2020}, with this extended to  batch normalization in \cite{ma2022riemannian}. An abstract LaSalle principle gives conditions for Wasserstein gradient flows to approach the stationary set in the narrow topology \cite{carrillo2023invariance}, although its application to neural networks remains open. Rigorous multilayer extensions require more elaborate families of interacting random fields \cite{nguyen2023rigorous}. For Elman-type recurrent networks, Agazzi et al.\ show that the population squared loss tends to zero under assumptions on initialization and convergence of the mean-field weights \cite{agazzi2023elman}. Despite these advances, conditional global convergence of mean-field dynamics under vector output weights remains open, where the evolving output direction driving the nonlinear parameters might move away from directions of descent. In this work we present the escape argument to control this coupling  when the features lack an exact scaling law.

\paragraph{Large-head limits and training dynamics of attention.}
The number of attention heads can play the role of width, but the resulting limits depend strongly on the parameterization. At random initialization, Hron et al.\ derive Neural Network Gaussian Process (NNGP) and NTK kernels for deep attention networks; under the standard softmax scaling, their Gaussian limit jointly sends the hidden-layer widths, the number of heads, and the query/key dimension to infinity, rather than taking a head-only limit at fixed internal dimensions \cite{hron2020infinite}. Fu et al.\ study a single layer with many heads in a random-feature regime: the query and key matrices remain frozen, only the value matrices are trained, and the analysis gives expressivity and finite-sample guarantees for normalized ReLU attention, rather than the softmax attention considered here \cite{fu2023single}. In a feature-learning scaling, Bordelon, Chaudhry, and Pehlevan derive a dynamical mean-field theory for several infinite limits, including infinitely many heads at fixed key/query dimension, and characterize finite-time training statistics \cite{bordelon2024infinite}. In \cite{barboni2026transformers} a large-head mean-field limit and an infinite-depth limit are combined to prove well-posedness and small-loss convergence through a conditional Wasserstein geometry. Closest to our single-layer setting, Huan and Yuan derive a Wasserstein gradient flow for the empirical law of heads in a simplified causal self-attention model with cross-entropy loss, prove a finite-time approximation of stochastic gradient descent, and study convergence to stationary sets under additional compactness or gradient-domination assumptions \cite{huan2026meanfield}. Kim and Suzuki also use a mean-field, two-timescale limit to analyze feature learning by an MLP followed by linear attention in an in-context learning model \cite{kim2024transformers}.

The present work isolates a single softmax attention layer with fixed context length and fixed token and key/query dimensions, sends only the number of independently parameterized heads to infinity, and trains both the attention and value parameters. Its emphasis is the long-time geometry of the resulting Wasserstein gradient flow: under regular-value or nondegeneracy assumptions, whenever the flow converges in $W_2$, its limit is globally minimizing, without a small-initial-loss condition. This differs from finite-head, task-specific analyses of gradient descent for text classification, word co-occurrence, in-context multi-task regression, or next-token prediction \cite{lu2021attention,yang2024training,pmlr-v247-siyu24a,huang2024nonasymptotic}. Finally, linear attention, random-feature attention, Performers, and FlashAttention address the complementary computational problem of reformulating, approximating, or evaluating softmax attention efficiently rather than the large-head training limit \cite{katharopoulos2020linear,peng2021random,choromanski2021performer,dao2022flashattention}.

\paragraph{Organization.} \Cref{sec_wellp_wgf} establishes well-posedness and stability of the Wasserstein gradient flow. \Cref{sec_glob_conv} presents the global-optimality criterion and the three escape-region results, with their geometric interpretation. Their detailed constructions are given in \Cref{sec_escape_constructions}, and \Cref{sec_examples} applies the criteria to fully connected networks and multi-head attention layers.

\section{Mean-field training dynamics}\label{sec_wellp_wgf}

In this section, we introduce the evolution one obtains by encoding the parameters of the model as a probability measure on the parameter space. We prove the well-posedness and stability of these mean-field training dynamics for a broad class of initializations, namely, sub-Gaussian distributions.

\subsection{Setting}
Throughout, we rely on the following assumption. As shown in \Cref{sec_examples} below, it covers multi-head attention layers, as well as fully connected networks with a single hidden layer, provided the activation function is differentiable with a bounded and Lipschitz differential.
\begin{ass}
    $\mathcal{F}$ is a separable Hilbert space and $\Omega=\RR^d$ for some $d\geq 1$. The loss $R:\mathcal{F}\to\RR_+$ is differentiable, with differential Lipschitz on bounded sets and bounded on sublevel sets of $R$. The mapping $\Phi:\Omega\to\mathcal{F}$ is differentiable, and its differential has the following properties.
    \begin{enumerate}[label=(\roman*)]
        \item Linear growth: there exists $c_1>0$ such that $\|d\Phi(u)\|_{\mathcal{L}(\RR^d;\mathcal{F})}\leq c_1(1+|u|)$ for every $u\in\RR^d$.
        \item Local Lipschitz continuity: there exists $c_2>0$ such that for every $r>0$ and $u_1,u_2\in B(0,r)$, we have
        \[
            \|d\Phi(u_1)-d\Phi(u_2)\|_{\mathcal{L}(\RR^d;\mathcal{F})}\leq c_2(1+r)|u_1-u_2|.
        \]
    \end{enumerate}
    \label{ass_wellp_wgf}
\end{ass}

\paragraph{Gradient flows.} Gradient flows, which correspond to the vanishing step size limit of gradient descent, are curves of steepest descent. In the case of the objective function $F_m$, a gradient flow is a $C^1$ path $t\in\RR_+\mapsto \bm{u}_t \in \Omega^m$ such that $\dot{\bm{u}}_t=-m\nabla F_m(\bm{u}_t)$ for every $t\geq 0$.\footnote{Following \cite{chizat2018global}, we scale the gradient by $m$, which is convenient when considering the mean-field limit $m\to+\infty$.} Under \Cref{ass_wellp_wgf}, the objective $F_m$ is bounded below and of class $C^1$ with a locally Lipschitz gradient. As a consequence, standard results (see for instance \cite{santambrogioEuclideanMetricWasserstein2017}) imply that, for every initialization $\bm{u}_0\in\Omega^m$, there exists a unique gradient flow of $F_m$ starting from $\bm{u}_0$.

\paragraph{Lifting to the space of probability measures.} Under \Cref{ass_wellp_wgf}, $\Phi$ has at most quadratic growth, so the functional $F$ introduced in \eqref{eq_def_fm_cts} is well defined on $\mathcal P_2(\Omega)$. The identity
\[
    F_m(\bm u)=F\left(\frac{1}{m}\sum_{i=1}^m\delta_{u_i}\right)
\]
connects the finite-dimensional objective to its lifted version. We now describe the common evolution equation for empirical and general parameter distributions.

\paragraph{Wasserstein gradient flow.} In \cite[Proposition B.1]{chizat2018global}, it is shown that, if $(\bm{u}_t)_{t\geq 0}$ is a gradient flow of $F_m$, then $t\mapsto \mu_t:=(1/m) \sum_{i=1}^m \delta_{(\bm{u}_t)_i}$ is a Wasserstein gradient flow of $F$, that is, a curve of steepest descent for the geometry induced by the $W_2$ distance on $\mathcal{P}_2(\Omega)$. More precisely,
\begin{equation}\label{eq_wgf}
    \partial_t\mu_t=-\mathrm{div}(v_t\mu_t)~~\mathrm{with}~~v_t(u)=-\nabla F'(\mu_t)(u)
\end{equation}
for every $t\geq 0$, where $F'(\nu):\Omega\to \RR$ denotes the first variation of $F$ at $\nu\in\mathcal{P}_2(\Omega)$. We refer the reader to \Cref{appendix_wgf} for the precise meaning of \eqref{eq_wgf} and to \cite{santambrogioEuclideanMetricWasserstein2017} for more details on Wasserstein gradient flows. When the evolution \eqref{eq_wgf} is initialized with a more complicated distribution, proving its long-time existence and stability is more involved. This is the purpose of the following subsection.

\subsection{Well-posedness}
Under \Cref{ass_wellp_wgf}, we show that the mean-field training dynamics are well-posed for every sub-Gaussian initialization.
\begin{thm} Under \Cref{ass_wellp_wgf}, if ${\mu_0\in\mathcal{P}_2(\Omega)}$ is sub-Gaussian, then there exists a unique Wasserstein gradient flow $(\mu_t)_{t\geq 0}$ of $F$ starting from $\mu_0$. Moreover, for every sequence ${(\mu^n_0)_{n\geq 0}\in (\mathcal{P}_2(\Omega))^{\NN}}$ of sub-Gaussian measures with $\lim_{n\to+\infty}W_2(\mu^n_0,\mu_0)=0$ and every $T>0$, denoting by $(\mu^n_t)_{t\geq 0}$ the Wasserstein gradient flow of $F$ initialized with $\mu^n_0$ for every $n\in\NN$, it holds $\lim_{n\to+\infty}\sup_{0\leq t\leq T}W_2(\mu_t^n,\mu_t)=0$.
\label{thm_wgf}
\end{thm}

\paragraph{Assumption on the initialization.} The sub-Gaussian condition includes Gaussian initializations and permits unbounded output parameters. Depending on the nonlinearity, \cite[Proposition 2.5]{chizat2018global} instead requires compact support of the initial measure or its $w$-marginal. When $\Phi(w,\theta)=\phi(\theta)w$ with $\phi$ bounded, treating the two groups of parameters separately could remove the sub-Gaussian condition on the $\theta$-marginal. We use a joint condition to keep the well-posedness argument uniform across the models considered here.

\paragraph{Remarks on the proof.} The proof of \Cref{thm_wgf} is given in \Cref{appendix_wgf}. In \cite{chizat2018global}, the analogue of \Cref{thm_wgf} (Proposition 2.5) is proved by relying on the theory of gradient flows in metric spaces developed in \cite{ambrosioGradientFlowsMetric2009}. In contrast, our proof uses only elementary tools and is based on a fixed point argument. We stress that this is made possible because, unlike \cite{chizat2018global}, we do not consider a possibly non-smooth regularization term in the objective. Our proof strategy is borrowed from the literature on interacting particle systems (see also \cite{castinUnifiedPerspectiveDynamics2025} for an application to the analysis of infinitely deep transformers) and uses Dobrushin-type estimates (see, for instance, \cite{dobrushinVlasovEquations1979,canizoWellposednessTheoryMeasures2011}). However, the control we have on the velocity field is weaker than in these works. We circumvent this issue by exploiting the sub-Gaussian assumption on the initialization and by controlling the growth of the tails along the evolution.

\section{Global convergence}\label{sec_glob_conv}
We now specialize to $\Omega=\RR^{d_w}\times\RR^{d_\theta}$ and $\Phi(w,\theta)=\phi(\theta)w$, with $\phi(\theta)\in\mathcal L(\RR^{d_w};\mathcal F)$. We first state the escape-region criterion for global optimality and then present the three constructions and their assumptions. The proofs of the constructions are deferred to \Cref{sec_escape_constructions}, so that the criteria can be compared before entering the detailed estimates.

\subsection{Instability of non-minimizers}
\label{subsec_glob_conv}

\emph{General mechanism.} Because of the special form assumed for $\Phi$ (namely $\Phi(w,\theta)=\phi(\theta)w$ with $\phi(\theta):\mathbb{R}^{d_w}\to\mathcal{F}$ linear), the first variation of $F$ at some measure $\mu\in\mathcal{P}_2(\Omega)$ is \emph{linear} in the $w$ variable. Indeed, for every $(w,\theta)\in\RR^{d_w}\times\RR^{d_{\theta}}$,
\begin{equation*}
    F'(\mu)(w,\theta)=\textstyle\langle R'(\int_{\Omega}\Phi d\mu),\phi(\theta)w\rangle_{\mathcal{F}}=\langle g_{\mu}(\theta),w\rangle ~~\mathrm{where}~~g_{\mu}(\theta):=\phi(\theta)^* R'(\int_{\Omega}\Phi d\mu).
\end{equation*}
As a consequence, the velocity field $(v_t)_{t\geq 0}$ associated with the Wasserstein gradient flow of $F$ (defined in \eqref{eq_wgf}) is the negative gradient of a partially linear function. This observation is at the core of the instability mechanism for non-minimizers of $F$. Indeed, if $(w_t,\theta_t)_{t\geq 0}$ is a solution of the characteristic system of ordinary differential equations (ODE)
\begin{equation}\label{characteristic_ode}
    \begin{pmatrix}\dot w_t\\\dot \theta_t\end{pmatrix}=-\nabla F'(\mu_t)(w_t,\theta_t)=\begin{pmatrix}-g_{\mu_t}(\theta_t)\\-J_{g_{\mu_t}}(\theta_t)^T w_t\end{pmatrix},
\end{equation}
then we have $\frac{d}{dt}(|w_t|^2/2)=\langle \dot w_t,w_t\rangle=-\langle g_{\mu_t}(\theta_t),w_t\rangle=-F'(\mu_t)(w_t,\theta_t)$. This shows that the norm of linear parameters grows as long as $F'(\mu_t)(w_t,\theta_t)$ is negative. To ensure that a positive amount of mass escapes to infinity and to obtain a contradiction to the convergence of $(\mu_t)_{t\geq 0}$ to $\mu$, it is therefore sufficient to find a region $A$ of the parameter space with nonempty interior such that $F'(\mu_t)(w_t,\theta_t)$ is negative for solutions of \eqref{characteristic_ode} initialized in $A$. This motivates \Cref{def_escape_reg} below. In practice, as $\mu$ is not a minimizer of $F$, convexity of $R$ (assumed in \Cref{ass_glob_conv} below) implies that the mapping $F'(\mu)$ is not identically zero and we usually take $A$ as a negative sublevel set of $F'(\mu)$. The main challenge is then to prove that characteristics initialized in $A$ cannot leave this set.
\begin{defn}[Escape region]\label{def_escape_reg}
    Let $\mu\in \mathcal{P}_2(\Omega)$. We say that a Borel set $A\subset \RR^{d_w}\times\RR^{d_\theta}$ containing a nonempty open set is an \emph{escape region} for $\mu$ if there exist $\eta,\epsilon>0$ such that, for every absolutely continuous curve $(\mu_t)_{t\geq 0}$ in $\mathcal{P}_2(\Omega)$ with $\sup_{t\geq 0}W_2(\mu_t,\mu)\leq\epsilon$, every solution of \eqref{characteristic_ode} initialized in $A$ satisfies $\sup_{t\geq 0} F'(\mu_t)(w_t,\theta_t)\leq-\eta$.
\end{defn}

\begin{rem}
    Depending on our assumptions on $\phi$, having $W_2(\mu_t,\mu)\leq \epsilon$ will ensure $\|F'(\mu_t)-F'(\mu)\|_{\mathcal{X}}\lesssim \epsilon$ for some norm $\|\cdot\|_{\mathcal{X}}$. Thus, one can discard $(\mu_t)_{t\geq 0}$ and $\mu$ and only reason at the level of an abstract function $f$ and time-dependent $\|\cdot\|_{\calX}$-perturbations $(f_t)_{t\geq 0}$. Consequently, the construction of an escape region simply amounts to the study of perturbations of a Euclidean gradient flow.
\end{rem}

\begin{rem}
If $\mu$ is a global minimizer of $F$ then $g_{\mu}\equiv0$ and $\mu$ does not admit an escape region. Indeed, taking
$\mu_t=\mu$ for every $t$, every solution of \eqref{characteristic_ode} is constant.
\end{rem}

We can now leverage \Cref{def_escape_reg} to prove our main global convergence result, which relies on the following assumption. As shown in \Cref{lem_ass_ass}, it implies \Cref{ass_wellp_wgf} so that \Cref{thm_wgf} holds.
\begin{ass}
\label{ass_glob_conv}
Let $\mathcal{F}$ be a separable Hilbert space and let
$\Omega=\mathbb R^{d_w}\times\mathbb R^{d_\theta}$.
Assume that the loss
$R:\mathcal{F}\to\mathbb R_+$
is convex and differentiable, with differential Lipschitz on bounded sets and
bounded on sublevel sets of $R$. Assume also that
$\Phi:\Omega\to\mathcal{F}$
has the form $\Phi(w,\theta)=\phi(\theta)w$,  where $\phi(\theta)\in\mathcal{L}(\mathbb R^{d_w};\mathcal{F})$ is differentiable with a bounded Lipschitz differential.
\end{ass}

\begin{thm}
\label{thm_glob_conv}
Assume that \Cref{ass_glob_conv} holds and that every $\mu\in\mathcal{P}_2(\Omega)$ that does not minimize $F$ admits an escape region. Let $(\mu_t)_{t\geq0}$ be a Wasserstein gradient
flow of $F$ and assume that $\mu_0$ has full support. If $W_2(\mu_t,\mu)\to 0$, then $\mu$ is a global minimizer of $F$.
\end{thm}
\begin{proof}
    We argue by contradiction and assume that $\mu$ admits an escape region $A$. The convergence of $(\mu_t)_{t\geq 0}$ to $\mu$ in the $W_2$ distance yields the existence of $t_0\geq 0$ such that $W_2(\mu_t,\mu)\leq \epsilon$ for every $t\geq t_0$, as well as the boundedness of $\int_{\RR^{d_w}\times\RR^{d_\theta}}|w|^2 d\mu_t(w,\theta)$. Since $A$ is an escape region, every solution $(w_t,\theta_t)_{t\geq t_0}$ of the characteristic system \eqref{characteristic_ode} with $(w_{t_0},\theta_{t_0})\in A$ satisfies
    \[
        \frac{d}{dt}\bigg(\frac{1}{2}|w_t|^2\bigg)=-F'(\mu_t)(w_t,\theta_t)\geq \eta
    \]
    for every $t\geq t_0$. As a result, we obtain
    \begin{equation}\label{eq_growth_moment}
            \int_{\RR^{d_w}\times\RR^{d_\theta}}|w|^2d\mu_t(w,\theta)=\int_{\RR^{d_w}\times\RR^{d_\theta}}|w_t|^2d\mu_{t_0}(w_{t_0},\theta_{t_0})\geq\int_{A}|w_t|^2d\mu_{t_0}(w_{t_0},\theta_{t_0})\geq 2\eta (t-t_0) \mu_{t_0}(A).
    \end{equation}
    By \Cref{lem_lagrangian_representation}, $\mu_{t_0}=(X_{0,t_0})_{\#}\mu_0$, where $X_{0,t_0}$ is a homeomorphism. Since $\mu_0$ has full support, the same holds for $\mu_{t_0}$, which implies $\mu_{t_0}(A)>0$. We therefore obtain that the right-hand side of \eqref{eq_growth_moment} is unbounded, which yields a contradiction.
\end{proof}

\begin{rem}[Full-support initializations]
    The full-support assumption in \Cref{thm_glob_conv} is a convenient sufficient condition ensuring that every escape region contains a positive amount of mass at all times. It is stronger than the assumptions used in \cite{chizat2018global}, and it is not expected to be sharp. In particular, in the bounded scalar case one should only need sufficient support in the nonlinear parameter, while in the asymptotically homogeneous case an omnidirectionality condition at large radius should be sufficient.
\end{rem}

Combining \Cref{thm_wgf,thm_glob_conv}, we obtain the following result, which states that, provided a model is sufficiently over-parameterized, then the finite-width training dynamics converge, in terms of objective value, arbitrarily close to the infimum.
\begin{cor}
\label{cor_glob_conv}
Assume that \Cref{ass_glob_conv} holds and that every $\mu\in\mathcal{P}_2(\Omega)$ that does not minimize $F$ admits an escape region. Let $\mu_0\in \mathcal{P}_2(\Omega)$ be sub-Gaussian with full support and assume that the Wasserstein gradient flow of $F$ initialized at $\mu_0$ converges in $W_2$. Let $(\mu_{m,0})_{m\in \NN}$ be such that, for every $m\in \NN$, there exists $\bm{u}_{m,0}\in \Omega^m$ with $\mu_{m,0}=(1/m)\sum_{i=1}^m \delta_{(\bm{u}_{m,0})_i}$. Further assume that $W_2(\mu_{m,0},\mu_0)\to 0$ as $m\to +\infty$. Then for every $\epsilon>0$, there exists $m_0\in \NN$ such that, for every $m\geq m_0$, the gradient flow $(\bm{u}_{m,t})_{t\geq 0}$ of $F_m$ initialized at $\bm{u}_{m,0}$ satisfies
\begin{equation*}
\lim_{t\to+\infty}F_m(\bm{u}_{m,t})\leq \inf_{\mu\in\mathcal{P}_2(\Omega)}F(\mu)+\epsilon\leq \inf_{\bm{u}_m\in\Omega^m}F_m(\bm{u}_m)+\epsilon.
\end{equation*}
\end{cor}

\paragraph{Regular values.} If a $C^1$ function $f$ is defined on $\RR^d$ or the sphere $\mathbb{S}^{d-1}$ for some $d$, we say that a value $\eta$ in the range of $f$ is regular if the differential of $f$ does not vanish on $f^{-1}(\{\eta\})$, and we say it is critical otherwise. Following \cite{chizat2018global}, we often rely on the assumption that some function has at least one regular value. As discussed in \cite[Section 3.2 and Appendix D.3]{chizat2018global}, this assumption is technical and seems quite weak. Indeed, Sard's theorem ensures that, if $f:\RR^d\to\RR$ is of class $C^d$ ($C^{d-1}$ if $f$ is defined on $\SS^{d-1}$ with $d\geq2$), then the set of critical values of $f$ has Lebesgue measure zero in $\RR$ (and hence the set of regular values is nonempty as soon as $f$ is not constant). For $d=1$, the sphere $\SS^0$ is discrete and its tangent spaces are zero-dimensional, so no value in the range is regular under the convention used here; this case must be treated separately. There exist $C^1$ functions for which every real number is a critical value (see \cite{whitneyFunctionNotConstant1935}), but these counterexamples involve artificial constructions. Non-smooth versions of Sard's theorem have been established for semialgebraic or subanalytic functions (see \cite{bolteNonsmoothMorseSard2006}). For every input, the parameter-to-prediction map implemented by a neural network typically enjoys such regularity. However, the functions we consider below are integrals of such maps against the data distribution. It is therefore unclear how to ensure the existence of a regular value, and we leave this property as an assumption.

\subsection{Bounded nonlinearities with scalar output parameters}\label{subsec_glob_conv_scalar}
In this section, we show that, when $d_w=1$ and $\phi$ is bounded, every probability measure that does not minimize $F$ admits an escape region under additional technical assumptions. This region has the form
\begin{equation}\label{def_esc_reg_bounded_scalar}
    A=\{(w,\theta)\in\RR\times\RR^{d_\theta}\,\rvert\, w\geq \bar{r},~g_{\mu}(\theta)\leq -\eta\}
\end{equation} for some $\bar{r},\eta>0$ with $-\eta$ a regular value of $g_{\mu}$. For positive regular values, the region can be defined by reversing the signs, that is by requiring $w\leq -\bar r$ and $g_\mu(\theta)\geq \eta$. The challenge of bounded nonlinearities such as softmax or sigmoids is that the relevant gradients in the nonlinear parameters vanish as those
parameters become large. We provide a corrected proof of \cite[Proposition C.4]{chizat2018global}, involving a new assumption (see \ref{conv_dgmu_dgnu} below) which is satisfied in the relevant examples considered in \Cref{sec_examples}, showing that trajectories inside $A$ can still be controlled in the regime of large nonlinear parameters. The proof of the following proposition is given in \Cref{appendix_bounded_scalar}.

\begin{prop}\label{prop_esc_reg_bounded_scalar}
    Assume that \Cref{ass_glob_conv} holds, $\phi$ is bounded, $d_\theta> 1$ and $d_w=1$. Further assume that $\mu\in\mathcal{P}_2(\Omega)$ is not a minimizer of $F$ and the following conditions hold:
    \begin{enumerate}[label=(\roman*)]
        \item the mapping $\theta\mapsto g_{\mu}(r\theta)$ converges uniformly in $C^1(\SS^{d_\theta-1})$ to a mapping $g_{\mu}^{\infty}$ as $r\to+\infty$, \label{conv_c1_gmu_bounded_scalar}
        \item the quantity $\sup\,\{|\theta||\nabla g_{\nu}(\theta)-\nabla g_{\mu}(\theta)|\,\rvert\,\theta\in\RR^{d_\theta},~\nu\in\mathcal{P}_2(\Omega),~W_2(\mu,\nu)\leq \epsilon\}$ converges to $0$ as $\epsilon\to 0$.\label{conv_dgmu_dgnu}
    \end{enumerate}
    If either of the following conditions holds:
    \begin{enumerate}[label=(\roman*),start=3]
    \item the function $g_{\mu}$ has a nonzero regular value associated with a bounded level set,\label{reg_val_bounded_scalar}
    \item the function $g_{\mu}$ has a nonzero regular value associated with an unbounded level set, which is also a regular value of $g_{\mu}^{\infty}$.\label{reg_val_inf_bounded_scalar}
    \end{enumerate}
    then $\mu$ admits an escape region.
\end{prop}

\begin{rem}\label{rem_bounded_scalar_scalar}
    As we show in \Cref{appendix_bounded_scalar}, if $d_\theta=1$, the same result holds with the only assumption that $g_{\mu}$ has a nonzero regular value. In particular, conditions \ref{conv_c1_gmu_bounded_scalar} and \ref{conv_dgmu_dgnu} are not needed.
\end{rem}

\begin{rem}\label{rem_reg_val_bounded_scalar}
    In \Cref{appendix_bounded_scalar}, we prove that if the following three conditions hold:
    \begin{itemize}
        \item the function $g_{\mu}$ is not constant,
        \item the set of critical values of $g_{\mu}$ has Lebesgue measure zero in $\RR$,
        \item the set of critical values of $g_{\mu}^{\infty}:\theta\in \mathbb{S}^{d_\theta-1}\mapsto \lim_{r\to+\infty} g_{\mu}(r\theta)$ has Lebesgue measure zero in $\RR$,
    \end{itemize}
    then condition \ref{reg_val_bounded_scalar} or condition \ref{reg_val_inf_bounded_scalar} holds.
    If $\mu$ does not minimize $F$ and $g_\mu$ is constant, then it is a nonzero constant; this remaining case is handled directly in the proof by a half-space escape region and does not require a regular value.
\end{rem}

\begin{rem}
    The proof considers an escape region constructed through a sublevel set of the function $g_\mu\approx g_{\mu_t}$ which is approximately driving the nonlinear parameters
    \[
        \dot\theta_t=-w_t\nabla g_{\mu_t}(\theta_t).
    \]
    For $w_t>0$, invariance of $\{g_\mu\leq-\eta\}$ follows from an inward-pointing estimate on its boundary; negativity then ensures growth of $w_t$. On a bounded level set, regularity and compactness give a positive lower bound on $|\nabla g_\mu|$. On an unbounded level set, this gradient may decay. We make use of regularity of the limiting angular profile  to instead give a lower bound on $|\theta||\nabla g_\mu(\theta)|$ at large radius. The weighted perturbation control in condition~\ref{conv_dgmu_dgnu} preserves the inward-pointing estimate in this regime.
\end{rem}

\paragraph{Extension to vector output weights.}
The $d_w>1$ case is more delicate: when $d_w=1$, the sign of $w_t$ fixes the orientation of the $\theta$-dynamics, so the same scalar function $g_\mu$ both defines the escape region and controls its invariance. When $d_w>1$, writing $v_t=w_t/|w_t|$, the dynamics become
\[
    \dot\theta_t \approx -|w_t|J_{g_\mu}(\theta_t)^Tv_t.
\]
The direction $v_t$ can itself evolve, on a different scale from $\theta_t$, so the pair $(v_t,\theta_t)$ is not automatically tied to the level sets of a scalar function related to $F'(\mu)$. This is the role of the construction in the next subsection: the map $\Theta$ couples each direction $v$ to a preferred nonlinear parameter $\Theta(v)$, allowing one to track a reduced scalar function on the sphere.

\subsection{Nonlinearities with at most linear growth under a non-degeneracy condition}\label{subsec_glob_conv_nondegenerate}
Motivated by the above discussion, we introduce a non-degeneracy assumption that allows one to construct an escape region for scalar output parameters ($d_w=1$) and vector output parameters ($d_w> 1$). In fact, this construction is valid beyond the bounded case and applies to any sufficiently smooth nonlinearity such that \Cref{ass_glob_conv} holds. We show that, under our non-degeneracy assumption, every probability measure that does not minimize $F$ admits an escape region of the form
    \begin{equation*}
        A:=\{(w,\theta)\in \RR^{d_w}\times\RR^{d_\theta}\,\rvert\,|w|\geq \bar{r},~w/|w|\in V,~\langle g_{\mu}(\Theta(w/|w|)),w/|w|\rangle\leq -\eta,~|\theta-\Theta(w/|w|)|\leq M\}
    \end{equation*}
for some $\bar r,M,\eta>0$ and a map $\Theta$ defined on an open subset $V$ of $\SS^{d_w-1}$. The proof is given in \Cref{appendix_nondegenerate}.
\begin{prop}\label{prop_esc_reg_nondegenerate}
    Assume that \Cref{ass_glob_conv} holds with $\phi$ locally $C^{2,1}$. Let $\mu\in\mathcal{P}_2(\Omega)$ be a measure that does not minimize $F$. Assume that there exists a nonempty open set $V\subset\SS^{d_w-1}$, a constant $\lambda>0$ and a $C^1$ map $\Theta:V\to\RR^{d_\theta}$ such that, for every $v\in V$,
    \begin{equation*}
        \left\{\begin{aligned}
            \nabla_{\theta}h(v,\Theta(v))&=0,\\
            \nabla^2_{\theta}h(v,\Theta(v))&\succeq \lambda I_{d_\theta},\\
        \end{aligned}\right.
    \end{equation*}
    where $h:(v,\theta)\in\SS^{d_w-1}\times\RR^{d_\theta}\mapsto \langle g_{\mu}(\theta),v\rangle$. Define $h_\Theta(v):=h(v,\Theta(v))$ and suppose that, for some $\eta>0$, the set $K:=\{v\in V\,\rvert\,h_\Theta(v)\leq-\eta\}$ is a nonempty compact subset of $V$. When $d_w>1$, assume also that $-\eta$ is a regular value of $h_\Theta$. Then $\mu$ admits an escape region.
\end{prop}

The hypotheses of \Cref{prop_esc_reg_nondegenerate} follow if $\theta\mapsto (1/2)|g_\mu(\theta)|^2$ has a non-degenerate local maximizer, as shown in \Cref{appendix_nondegenerate}.
\begin{cor}\label{cor_esc_reg_nondegenerate_bis}
    Assume that \Cref{ass_glob_conv} holds with $\phi$ locally $C^{2,1}$. Let $\mu\in\mathcal{P}_2(\Omega)$ be a measure that does not minimize $F$. If there exists a non-degenerate local maximizer $\theta_*\in\RR^{d_\theta}$ of $\theta\mapsto (1/2)|g_{\mu}(\theta)|^2$, then $\mu$ admits an escape region.
\end{cor}

\begin{rem}
    The map $\Theta$ couples each output direction $v=w/|w|$ to a stable critical point of $\theta\mapsto\langle g_\mu(\theta),v\rangle$. This produces the reduced scalar function
    \[
        h_\Theta(v)=F'(\mu)(v,\Theta(v))
        =\langle g_\mu(\Theta(v)),v\rangle
    \]
    and an escape region shaped as a tube around the branch $\theta=\Theta(v)$.  When $d_w>1$, it also points inward on the angular boundary $h_\Theta(v)=-\eta$. Non-degeneracy supplies the restoring force in the nonlinear parameters, allowing their motion to be controlled as the output direction evolves.
\end{rem}

\begin{rem}
For $d_w=1$, the set $K$ consists of admissible signs in $\{-1,1\}$, and the tube has the explicit form
    \[
        A=\bigcup_{v\in K}\Bigl(\{w\in\RR\,\rvert\,vw\geq\bar r\}
            \times\overline{B(\Theta(v),M)}\Bigr).
    \]
    Inside the tube, $|w_t|$ grows, and the vector field points inward on the tube boundary $|\theta-\Theta(v)|=M$.
\end{rem}

\begin{rem}
    From first-order optimality conditions, one can see that, if $\mu$ is a stationary point of the Wasserstein gradient flow of $F$, then $g_{\mu}$ vanishes on the support of the $\theta$-marginal of $\mu$. If $\mu$ does not minimize $F$, then $g_{\mu}$ is not identically zero.
\end{rem}

\subsection{Asymptotically positively one-homogeneous nonlinearities}
This subsection deals with the case where $\phi$ is asymptotically positively $1$-homogeneous, which is relevant, for instance, for two-layer networks with GELU or SiLU activations. The following proposition, whose proof is given in \Cref{appendix_asymp_homogeneous}, states that, under suitable assumptions, every probability measure that does not minimize $F$ admits an escape region of the form
\[
    A=\{(w,\theta)\in\RR^{d_w} \times \RR^{d_\theta}\,\rvert\, \min(|w|,|\theta|)\geq \bar{r},~1/M\leq |w|/|\theta|\leq M,~h_{\mu}^{\infty}(w/|w|,\theta/|\theta|)\leq -\eta\}
\]
for some $\bar{r},M,\eta>0$ (see \eqref{eq_def_h_mu_infty} for the definition of $h_{\mu}^{\infty}$ below).

We recall that a sequence $(L_n)_{n\geq 0}$ of bounded linear operators between two Hilbert spaces $\mathcal{H}_1$ and $\mathcal{H}_2$ is said to converge weakly to $L_{\infty}$ if $(\langle L_n f,g\rangle_{\mathcal{H}_2})_{n\geq 0}$ converges to $\langle L_{\infty}f,g\rangle_{\mathcal{H}_2}$ for every $(f,g)\in\mathcal{H}_1\times\mathcal{H}_2$. Given a family of sequences $(L^{\alpha}_n)_{\alpha\in\mathcal{A},n\geq 0}$, we say that it converges weakly, uniformly on $\mathcal{A}$, to $(L^{\alpha}_{\infty})_{\alpha\in\mathcal{A}}$ if, for every $(f,g)\in\mathcal{H}_1\times\mathcal{H}_2$, $\sup_{\alpha\in\mathcal{A}}|\langle L^\alpha_n f,g\rangle_{\mathcal{H}_2}-\langle L_{\infty}^{\alpha} f,g\rangle_{\mathcal{H}_2}|\to 0$ as $n\to+\infty$.
\begin{prop}\label{prop_esc_reg_asymp_homogeneous}
    Assume that $(d_w,d_\theta)\neq (1,1)$, that \Cref{ass_glob_conv} holds and that
    \begin{enumerate}[label=(\roman*)]
        \item the mapping $\theta\mapsto (1/r)\phi(r\theta)$ converges weakly, uniformly on $\mathbb{S}^{d_\theta-1}$, as $r\to+\infty$,\label{conv_phi_asymp_hom}
        \item the mapping $\theta\mapsto d\phi(r\theta)$ converges weakly, uniformly on $\mathbb{S}^{d_\theta-1}$, as $r\to+\infty$.\label{conv_dphi_asym_hom}
    \end{enumerate}
    Define $\phi_{\infty}:\theta \in \RR^{d_\theta}\mapsto \lim_{r\to+\infty}(1/r)\phi(r\theta)$. Assume that
    \begin{enumerate}[label=(\roman*),start=3]
        \item the linear span of $\{\phi_{\infty}(\theta)w,~(w,\theta)\in\Omega\}$ is dense in $\mathcal{F}$. \label{density_phi_asymp_hom}
    \end{enumerate}
    If $\mu$ is not a minimizer of $F$ then the mapping
    \begin{equation}\label{eq_def_h_mu_infty}
    \textstyle h_{\mu}^{\infty}:(w,\theta)\in \mathbb{S}^{d_w-1}\times\mathbb{S}^{d_\theta-1}\mapsto \lim_{r\to+\infty} (1/r)F'(\mu)(w,r\theta)
    \end{equation} is not identically zero. In addition, if $h_{\mu}^{\infty}$ has a nonzero regular value, then $\mu$ admits an escape region.
\end{prop}

\begin{rem}
    If $d_w=d_\theta=1$, the first part of the statement also holds and there is at least one element of $\{-1,1\}^2$ at which $h_{\mu}^{\infty}$ is nonzero. Assuming this is the case at $(1,1)$ and $h_{\mu}^{\infty}(1,1)<0$ for simplicity, one can prove that
    \begin{equation*}
        A=\{(w,\theta)\in\RR^2\,\rvert\,\min(w,\theta)\geq \bar{r},~1/M\leq w/\theta\leq M\}
    \end{equation*}
    is an escape region for some $\bar{r},M>0$.
\end{rem}

\begin{rem}
    The asymptotically homogeneous construction controls both parameter directions through a common limiting function. At large parameter values,
    \[
        F'(\mu)(w,\theta)\approx |w|\,|\theta|\, h_\mu^\infty(w/|w|,\theta/|\theta|).
    \]
    Writing $v=w/|w|$ and $\varphi=\theta/|\theta|$, the angular dynamics approximate descent of $h_\mu^\infty$ on $\mathbb S^{d_w-1}\times\mathbb S^{d_\theta-1}$, with positive scaling factors depending on $|w|/|\theta|$. The bounds on this ratio keep both angular motions controlled. A negative regular sublevel set then supplies the angular boundary of an invariant region in which both parameter norms grow.
\end{rem}

\paragraph{Comparison of the escape regions.}
    The last two constructions have complementary assumptions. Verifying convergence to the homogeneous profile may require regularity of the data distribution: for example, \Cref{prop_glob_conv_gelu} assumes that $\rho_x$ has a density. Its conclusion therefore does not directly cover finite empirical distributions. The non-degenerate criterion instead requires a non-degenerate local maximizer of $|g_\mu|^2$; the application in \Cref{prop_glob_conv_mlp_nondegenerate} permits discrete data, provided that this condition holds. Geometrically, the asymptotically homogeneous construction places both parameter groups at large radius, whereas the non-degenerate construction keeps the nonlinear parameters near the critical branch and lets the output norm grow.

\section{Construction of the escape regions}\label{sec_escape_constructions}
We prove \Cref{prop_esc_reg_bounded_scalar,prop_esc_reg_nondegenerate,cor_esc_reg_nondegenerate_bis,prop_esc_reg_asymp_homogeneous} in the order in which they were stated. The geometric explanations in \Cref{sec_glob_conv} identify the quantities to control; here we choose the parameters and establish the required invariance and growth estimates.
\subsection{Bounded nonlinearities with scalar output weights}\label{appendix_bounded_scalar}

We prove \Cref{prop_esc_reg_bounded_scalar}. Write $g=g_\mu$ and $g_t=g_{\mu_t}$. For any absolutely continuous curve $(\mu_t)_{t\geq0}$, we have
\begin{equation}\label{bound_diff_g_gt}
    \|g_t-g\|_{C^1}\leq \|\phi\|_{C^1}
    \Big\|R'\Big(\int_\Omega\Phi\,d\mu_t\Big)-R'\Big(\int_\Omega\Phi\,d\mu\Big)\Big\|_{\mathcal F}.
\end{equation}
By Lipschitz continuity of $R'$ on bounded sets and \Cref{lemma_bound_diff_Phi_mu}, we have $\sup_{t\geq0}\|g_t-g\|_{C^1}\leq\epsilon$ whenever $\sup_{t\geq0}W_2(\mu_t,\mu)\leq C\epsilon$, for a suitable $C>0$ and sufficiently small $\epsilon$. We work with such curves and choose $\epsilon$ below.

If $g\equiv c$ is constant, then $c\neq0$ because $\mu$ is not a minimizer. Fix any $\bar r>0$ and set $A=\{(w,\theta)\in\RR\times\RR^{d_\theta}\,\rvert\,-\sign(c)w\geq\bar r\}$. If $\epsilon<|c|/2$, then along every trajectory initialized in $A$, the quantity $-\sign(c)w_t$ increases at speed at least $|c|/2$. Hence $F'(\mu_t)(w_t,\theta_t)\leq-\bar r|c|/2$ for all $t\geq0$, so $A$ is an escape region.

For nonconstant $g$, consider a negative regular value $-\eta$ and set $K=\{g\leq-\eta\}$. For a positive regular value, replace $w$, $g$ and $g_t$ by their negatives in the argument. To verify that the region $A$ in \eqref{def_esc_reg_bounded_scalar} is an escape region, it suffices to check $\langle\nabla g(\theta),\nabla g_t(\theta)\rangle>0$
for every $\theta\in\partial K$. Indeed, on this boundary $$\frac{d}{dt}g(\theta_t)=-w_t\langle\nabla g(\theta_t),\nabla g_t(\theta_t)\rangle<0$$ whenever $w_t\geq\bar r$. Inside $K$, taking $\epsilon\leq\eta/2$ gives $\dot w_t=-g_t(\theta_t)\geq\eta/2$. Thus both constraints are preserved, and $F'(\mu_t)(w_t,\theta_t)=w_tg_t(\theta_t)\leq-\bar r\eta/2$ throughout $A$. Regularity of the level ensures that $K$ has nonempty interior. The same reasoning applies if $K$ is a connected component of the sublevel set with nonempty interior.

\paragraph{Bounded level set.} Under condition~\ref{reg_val_bounded_scalar}, the boundary is compact and regular, so $|\nabla g|\geq\beta$ on $\partial K$ for some $\beta>0$. For $\epsilon\leq\beta/2$,
\[
    \langle\nabla g(\theta),\nabla g_t(\theta)\rangle
    \geq|\nabla g(\theta)|\bigl(|\nabla g(\theta)|-\epsilon\bigr)
    \geq\beta^2/2
    \qquad(\theta\in\partial K).
\]

\paragraph{One nonlinear parameter.} If $d_\theta=1$, each component of $\{g\leq-\eta\}$ is an interval with at most two boundary points. Taking such a component as $K$ reduces the proof of \Cref{rem_bounded_scalar_scalar} to the bounded-boundary estimate.

\paragraph{Unbounded level set.} Under condition~\ref{reg_val_inf_bounded_scalar}, $-\eta$ is also a regular value of $g_\infty:=g_\mu^\infty$. Write
\begin{equation*}
\beta_{\infty}\eqdef\inf\,\{|\nabla_{\mathbb{S}} g_{\infty}(\varphi)|\,\rvert\,\varphi\in\mathbb{S}^{d_\theta-1},~g_{\infty}(\varphi)=-\eta\}>0,
\end{equation*}
where $\nabla_{\SS}$ denotes the spherical gradient. Compactness of $\mathbb{S}^{d_\theta-1}$ gives $\gamma_\infty>0$ such that
\begin{equation}\locallabel{inf_grad_ginf}
    |\nabla_{\SS}g_\infty(\varphi)|\geq\beta_\infty/2
    \qquad\text{if }|g_\infty(\varphi)+\eta|\leq\gamma_\infty.
\end{equation}
By condition~\ref{conv_c1_gmu_bounded_scalar}, choose the threshold $\bar r\geq1$ in \eqref{def_esc_reg_bounded_scalar} large enough that, for every $r\geq\bar r$,
\begin{align}
    &\sup_{\varphi\in\mathbb S^{d_\theta-1}}|g(r\varphi)-g_\infty(\varphi)|
    \leq\gamma_\infty,\locallabel{conv_unif_gr}\\
    &\sup_{\varphi\in\mathbb S^{d_\theta-1}}
    \big|r\,\mathrm{proj}_{\{\varphi\}^\perp}\nabla g(r\varphi)-\nabla_{\SS}g_\infty(\varphi)\big|
    \leq\beta_\infty/8.\locallabel{conv_unif_grad_gr}
\end{align}
On the part of $\partial K$ with $|\theta|\leq \bar r$, set
\[
    \beta_{\mathrm{loc}}:=\inf\{\abs{\nabla g(\theta)}\,\rvert\,\theta\in\partial K,\ \abs{\theta}\leq \bar r\}>0.
\]
This is positive by compactness and regularity, with the convention $\inf\varnothing=+\infty$. Choose
\begin{equation}\locallabel{choice_epsilon}
    \epsilon\leq\tfrac12\min(\eta,\beta_{\mathrm{loc}}).
\end{equation}
For $\theta\in\partial K$ with $|\theta|\leq\bar r$, we have $|\nabla g(\theta)|\geq\beta_{\mathrm{loc}}$ and $|\nabla g_t(\theta)-\nabla g(\theta)|\leq\epsilon$. Hence
\[
    \langle\nabla g(\theta),\nabla g_t(\theta)\rangle
    \geq |\nabla g(\theta)|\bigl(|\nabla g(\theta)|-\epsilon\bigr)
    \geq \beta_{\mathrm{loc}}^2/2>0.
\]
By condition~\ref{conv_dgmu_dgnu}, decrease $\epsilon$ further so that
\begin{equation}\locallabel{choice_epsilon_bis}
    \sup_{t\geq0,\ \theta\in\RR^{d_\theta}}
    |\theta|\,|\nabla g_t(\theta)-\nabla g(\theta)|\leq\beta_\infty/8.
\end{equation}
For $\theta\in\partial K$ with $|\theta|\geq \bar r$, let $\varphi=\theta/|\theta|$. Since $g(\theta)=-\eta$, \localeqref{conv_unif_gr} puts $g_\infty(\varphi)$ within $\gamma_\infty$ of $-\eta$. Combining \localeqref{inf_grad_ginf} and \localeqref{conv_unif_grad_gr} gives $|\theta||\nabla g(\theta)|\geq 3\beta_\infty/8$. Therefore
\[
    \langle\nabla g(\theta),\nabla g_t(\theta)\rangle
    \geq |\nabla g(\theta)|
        \bigl(|\nabla g(\theta)|-|\nabla g_t(\theta)-\nabla g(\theta)|\bigr)
    \geq \frac{3\beta_\infty^2}{32|\theta|^2}>0.
\]
The boundary criterion now applies to all of $\partial K$.

\paragraph{Regular values.} To prove \Cref{rem_reg_val_bounded_scalar}, suppose that the three conditions there hold. Since $g$ is nonconstant and its critical values have measure zero, its nonzero regular values form a set of positive measure. If one of them has a bounded level set, condition~\ref{reg_val_bounded_scalar} holds. Otherwise, each such value $\eta'$ is attained along a sequence $\theta_n=r_n\varphi_n$ with $r_n\to+\infty$. Along a subsequence, $\varphi_n\to\varphi$ on the sphere, so uniform convergence and continuity give
\[
    \eta'=\lim_{n\to\infty}g(r_n\varphi_n)=g_\infty(\varphi).
\]
The nonzero regular values of $g$ thus lie in the range of $g_\infty$. Because the critical values of $g_\infty$ have measure zero, one can choose a common nonzero regular value, giving condition~\ref{reg_val_inf_bounded_scalar}.

\subsection{Nonlinearities with at most linear growth under a non-degeneracy condition}\label{appendix_nondegenerate}
\subsubsection{Proof of Proposition \ref{prop_esc_reg_nondegenerate}}
Write $g=g_\mu$, and let $g_t=g_{\mu_t}$ for an absolutely continuous curve satisfying $\sup_{t\geq0}W_2(\mu_t,\mu)\leq\epsilon$, with $\epsilon$ chosen below. Along a characteristic, write $v_t=w_t/|w_t|$ when $w_t\neq0$. For $d_w=1$, the sphere $\SS^0=\{-1,1\}$ is discrete, so $K$ has empty boundary and $v_t$ stays constant as long as $w_t$ stays nonzero. Set $\beta:=\inf_{v\in\partial K}|\nabla_{\SS}h_\Theta(v)|$, with $\inf\varnothing=+\infty$. When $\partial K$ is nonempty, compactness and regularity give $\beta>0$.

\paragraph{Choice of parameters.} We choose the tube width $M$ to control deviations from the branch, then $\epsilon$ to control perturbations of the vector field. Finally, we choose the lower bound $\bar r$ on $|w|$ so that attraction toward the branch dominates the motion of the branch itself.

    Set $W_1=\{\theta:\mathrm{dist}(\theta,\Theta(K))\leq1\}$, which is compact. Since $g$ is locally $C^{2,1}$ and $\nabla_\theta h(v,\Theta(v))=0$, there is a uniform $C>0$ such that, for $v\in K$ and $|\omega|\leq1$,
    \[
        \big|\nabla_\theta h(v,\Theta(v)+\omega)
             -\nabla_\theta^2h(v,\Theta(v))\omega\big|\leq C|\omega|^2.
    \]
    On any compact $W\subset\RR^{d_\theta}$, the estimate \eqref{bound_diff_g_gt} with $C^1$ norms restricted to $W$ gives $\delta_W(\epsilon)\to0$ as $\epsilon\to0$, where
    \begin{equation}\locallabel{eq_bound_diff_g_gt}
        \delta_W(\epsilon):=\sup\{\|g_\nu-g_\mu\|_{C^1(W)}\,\rvert\,\nu\in\mathcal P_2(\Omega),~W_2(\mu,\nu)\leq\epsilon\}.
    \end{equation}
    Let $L=\|J_g\|_{C^0(W_1)}$ and $G_1=1+\|g\|_{C^0(W_1)}$. Choose $M\in(0,1]$ so that
    \begin{equation}\locallabel{eq_nondeg_choice_M}
        LM\leq \frac{\eta}{4},\qquad
        G_1LM\leq \frac{\beta^2}{8},\qquad
        CM\leq\frac{\lambda}{6}.
    \end{equation}
    Set $W=\{\theta:\mathrm{dist}(\theta,\Theta(K))\leq M\}$ and choose $\epsilon>0$ small enough that $\delta(\epsilon):=\delta_W(\epsilon)$ satisfies
    \begin{equation}\locallabel{eq_nondeg_choice_eps}
        \delta(\epsilon)\leq \min\left(1,\frac{\eta}{4},\frac{\beta^2}{8G_1},\frac{M\lambda}{6}\right).
    \end{equation}
    Finally, choose $\bar r>0$ large enough that
    \begin{equation}\locallabel{eq_nondeg_choice_r}
        \frac{\|J_\Theta\|_{C^0(K)}}{M\bar r^2}\big(1+\|g\|_{C^0(W)}\big)\leq \frac{\lambda}{6}.
    \end{equation}
    For $d_w=1$, $J_\Theta=0$, so any $\bar r>0$ suffices.

    \paragraph{The escape region.} With these parameters fixed, set
    \begin{equation*}
        A:=\{(w,\theta)\in \RR^{d_w}\times\RR^{d_\theta}\,\rvert\,|w|\geq \bar{r},~w/|w|\in K,~|\theta-\Theta(w/|w|)|\leq M\}.
    \end{equation*}

    While $(w_t,\theta_t)\in A$, the segment between $\theta_t$ and $\Theta(v_t)$ lies in $W_1$. The mean value and perturbation bounds therefore give
    \[
        |g_t(\theta_t)-g(\Theta(v_t))|\leq LM+\delta(\epsilon).
    \]

    \paragraph{Outer weight growth.} By this bound, $h_\Theta(v_t)\leq-\eta$, and the choices of $M$ and $\epsilon$,
    \[
        \frac{d}{dt}|w_t|=-\langle g_t(\theta_t),v_t\rangle
        \geq-h_\Theta(v_t)-LM-\delta(\epsilon)\geq\eta/2.
    \]
    Hence the constraint $|w_t|\geq\bar r$ cannot be the first to fail.

    \paragraph{Angular boundary.} Suppose $\partial K$ is nonempty. Since $\nabla_{\theta}h(v,\Theta(v))=0$ for every $v\in V$, we have $\nabla_{\SS} h_{\Theta}(v)=\nabla_{\SS,v} h(v,\Theta(v))$. Writing $P_v:=\Id-vv^{\top}$, we also have $\dot v_t=-(1/|w_t|)P_{v_t}g_t(\theta_t)$ and $\nabla_{\SS} h_{\Theta}(v_t)=P_{v_t}g(\Theta(v_t))$. Therefore, when $(w_t,\theta_t)\in A$ and $h_{\Theta}(v_t)=-\eta$,
    \begin{equation*}
        \begin{aligned}
            \frac{d}{dt}h_{\Theta}(v_t)&=\langle \dot v_t,\nabla_{\SS} h_{\Theta}(v_t)\rangle\\
            &=-\frac{1}{|w_t|}\langle P_{v_t}g_t(\theta_t),P_{v_t}g(\Theta(v_t))\rangle\\
            &\leq -\frac{1}{|w_t|}\left(\beta^2-G_1(LM+\delta(\epsilon))\right)
             \leq-\frac{\beta^2}{2|w_t|}<0.
        \end{aligned}
    \end{equation*}
    Here the last line uses \localeqref{eq_nondeg_choice_M} and \localeqref{eq_nondeg_choice_eps}. Hence the angular component cannot exit $K$ through $\partial K$.

    \paragraph{Tube boundary.}
    We now prove that $\omega_t:=\theta_t-\Theta(v_t)$ satisfies $\frac{d}{dt} |\omega_t|<0$ when $(w_t,\theta_t)\in A$ and $|\omega_t|=M$. Defining $h_t:(v,\theta)\mapsto\langle g_t(\theta),v\rangle$, the tube-error dynamics are
    \[
        \dot \omega_t = -\abs{w_t} \nabla_\theta h_t(v_t,\theta_t) + \frac{1}{\abs{w_t}} J_\Theta(v_t)\nabla_{\SS,v} h_t(v_t,\theta_t).
    \]
    The Taylor estimate $C$ and the perturbation bound give
    \begin{equation*}
        \left|\nabla_{\theta}h_t(v_t,\theta_t)-\nabla_\theta^2 h(v_t,\Theta(v_t))\omega_t\right|
        \leq CM^2+\delta(\epsilon).
    \end{equation*}
    Thus, when $(w_t,\theta_t)\in A$ and $|\omega_t|=M$,
    \begin{equation*}
    \begin{aligned}
        \frac{d}{dt}\bigg(\frac{1}{2}|\omega_t|^2\bigg) &= \dotp{\dot \omega_t}{\omega_t}\\
        &\leq -\abs{w_t} \dotp{\nabla_\theta^2 h(v_t,\Theta(v_t))\omega_t}{\omega_t} + \abs{w_t} CM^3  + \abs{w_t} M \delta(\epsilon) + \frac{M}{\abs{w_t}} \norm{J_\Theta}_{C^0(K)} (\norm{g}_{C^0(W)}+\delta(\epsilon))\\
        &\leq -\abs{w_t}\pa{ \lambda M^2 - CM^3 - M\delta(\epsilon) - \frac{M}{\bar r^2}\norm{J_\Theta}_{C^0(K)}(\norm{g}_{C^0(W)} + \delta(\epsilon)) }\\
        &\leq -\abs{w_t} M^2 \frac{\lambda}{2}.
    \end{aligned}
    \end{equation*}
    The last step follows by our choice of parameters in \localeqref{eq_nondeg_choice_M}, \localeqref{eq_nondeg_choice_eps} and \localeqref{eq_nondeg_choice_r}.

    The boundary estimates rule out a first exit from $A$. The growth estimate then gives $$F'(\mu_t)(w_t,\theta_t)=-|w_t|\frac{d}{dt}|w_t|\leq-\bar r\eta/2$$ for all $t\geq0$. The set $K$ has nonempty interior by regularity when $d_w>1$ and by discreteness when $d_w=1$. Hence $A$ has nonempty interior and is an escape region.

\subsubsection{Proof of Corollary \ref{cor_esc_reg_nondegenerate_bis}}
    We use the notation $g:=g_{\mu}$. Nondegeneracy forces $g(\theta_*)\neq0$: otherwise the nonnegative function $|g|^2/2$ would vanish identically in a neighborhood of a local maximizer, contradicting its nondegenerate Hessian. We set $v_*=-g(\theta_*)/|g(\theta_*)|$. Since $\theta_*$ is a non-degenerate local maximizer of $\theta\mapsto (1/2)|g(\theta)|^2$, we have $J_g(\theta_*)^T g(\theta_*)=0$ and $H_g(\theta_*)[g(\theta_*)]+J_g(\theta_*)^T J_g(\theta_*)\prec 0$, where
    \[
        H_g(\theta)[w]:=\Bigg[\sum_{k=1}^{ d_w}\partial_{ij}g_k(\theta)w_k\Bigg]_{1\leq i,j\leq d_\theta}
    \]
    for every $(w,\theta)\in\RR^{d_w}\times\RR^{d_\theta}$. Defining $h:(v,\theta)\in\SS^{d_w-1}\times\RR^{d_\theta}\mapsto \langle g(\theta),v\rangle$, we hence obtain
    \begin{equation*}\left\{
        \begin{aligned}
            \nabla_{\theta}h(v_*,\theta_*)&=J_g(\theta_*)^Tv_*=0,\\
            \nabla^2_{\theta}h(v_*,\theta_*)&=H_g(\theta_*)[v_*]=-(1/|g(\theta_*)|)H_g(\theta_*)[g(\theta_*)]\succ (1/|g(\theta_*)|)J_g(\theta_*)^T J_g(\theta_*)\succeq 0.
        \end{aligned}\right.
    \end{equation*}

    For $d_w=1$, take $V=\{v_*\}$, which is open in $\SS^0$, and set $\Theta(v_*)=\theta_*$. Since $h_\Theta(v_*)=-|g(\theta_*)|$, any $\eta\in(0,|g(\theta_*)|)$ gives $K=V$. Thus \Cref{prop_esc_reg_nondegenerate} applies.

    For $d_w>1$, define
    \begin{equation*}
    \psi:(v,\theta)\in\SS^{d_w-1}\times\RR^{d_\theta}\mapsto \nabla_\theta h(v,\theta)=J_g(\theta)^Tv.
    \end{equation*} 
    We have $\psi(v_*,\theta_*)=0$ and $D_{\theta}\psi(v_*,\theta_*)=H_g(\theta_*)[v_*]\succ 0$. By the implicit function theorem, we obtain the existence of an open set $V\subset\SS^{d_w-1}$ containing $v_*$ and a $C^1$ map $\Theta:V\to\RR^{d_\theta}$ such that ${\Theta(v_*)=\theta_*}$ and $\nabla_{\theta}h(v,\Theta(v))=0$ for every $v\in V$. By continuity, shrinking $V$ if necessary, we can ensure that ${\nabla^2_{\theta}h(v,\Theta(v))\succeq \lambda I_{d_\theta}}$ for some $\lambda>0$ and every $v\in V$.

    Defining $h_{\Theta}:v\in V\mapsto h(v,\Theta(v))$, for every $v\in V$, we have
    \begin{equation*}
        \nabla_{\SS} h_{\Theta}(v)=\nabla_{\SS,v} h(v,\Theta(v))+J^{\SS}_{\Theta}(v)^T\nabla_{\theta}h(v,\Theta(v))=\nabla_{\SS,v} h(v,\Theta(v))=\mathrm{proj}_{\{v\}^\perp}(g(\Theta(v))).
    \end{equation*}
    This gradient is $C^1$ and vanishes at $v_*$. To show that $v_*$ is a nondegenerate local minimum of $h_\Theta$, write $a=|g(\theta_*)|$, $J=J_g(\theta_*)$ and $B=\nabla_\theta^2 h(v_*,\theta_*)$. Differentiating $\nabla_\theta h(v,\Theta(v))=0$ gives $D\Theta(v_*)[u]=-B^{-1}J^\top u$ for each tangent vector $u\perp v_*$. In the derivative of $\nabla_{\SS}h_\Theta(v)$, the derivative of the projection $\Id-vv^\top$ contributes $au$ because $g(\theta_*)=-av_*$. The resulting spherical Hessian satisfies
    \[
        \nabla^2_{\SS} h_\Theta(v_*)[u,u]
        =a|u|^2-\big|B^{-1/2}J^\top u\big|^2.
    \]
    The inequality $aB\succ J^\top J$ above gives $\|JB^{-1/2}\|_{\mathrm{op}}^2<a$, so this Hessian is positive definite on the tangent space. In particular, $v_*$ is a nondegenerate local minimum with $h_\Theta(v_*)=-a$.

    Finally, shrinking $V$ to a sufficiently small geodesic ball around $v_*$, we may assume that $v_*$ is the only critical point of $h_{\Theta}$ in $V$ and that $h_{\Theta}(v_*)<\inf_{v\in \partial V} h_{\Theta}(v)$. As a result, every $\eta\in (0,-h_{\Theta}(v_*))$ sufficiently close to $-h_{\Theta}(v_*)$ satisfies $\{h_{\Theta}\leq-\eta\}\subset\subset V$ and $-\eta$ is a regular value of $h_{\Theta}$. The result then follows from the previous proposition.

\subsection{Asymptotically positively one-homogeneous nonlinearities}\label{appendix_asymp_homogeneous}
We prove \Cref{prop_esc_reg_asymp_homogeneous}. Write $g=g_\mu$, $g_t=g_{\mu_t}$ and $f=R'(\int_\Omega\Phi\,d\mu)$. We consider absolutely continuous curves $(\mu_t)_{t\geq0}$ with $\sup_{t\geq0}W_2(\mu_t,\mu)\leq\epsilon$, where $\epsilon$ will be chosen below.

Define $g_\infty(\varphi)=\phi_\infty(\varphi)^*f$ on $\SS^{d_\theta-1}$, so that $h_\infty(v,\varphi):=\langle g_\infty(\varphi),v\rangle$ is the function in \eqref{eq_def_h_mu_infty}. Since $\mu$ is not a minimizer, $f\neq0$. If $g_\infty$ were identically zero, positive homogeneity of $\phi_\infty$ would imply that $f$ is orthogonal to every $\phi_\infty(\theta)w$, contradicting the density assumption. Hence $h_\infty$ is not identically zero.

By oddness of $h_\infty$ in $v$, the regular-value hypothesis yields a negative regular value $-\eta$. Define
\[
    K:=\{h_\infty\leq-\eta\},\qquad
    \beta:=\inf_{(v,\varphi)\in\partial K}|\nabla_{\SS}h_\infty(v,\varphi)|,
\]
where $\nabla_{\SS}$ denotes the gradient on the product of spheres. Compactness and regularity of $\partial K=\{h_\infty=-\eta\}$ give $\beta>0$.

\paragraph{Perturbation bounds.} For $\nu\in\mathcal{P}_2(\Omega)$ write $f_\nu=R'(\int_\Omega\Phi\,d\nu)$. Since $g_\nu(\theta)=\phi(\theta)^*f_\nu$, for $|\theta|\geq1$ we have
\[
    \frac{|g_\nu(\theta)-g(\theta)|}{|\theta|}
      +\|J_{g_\nu}(\theta)-J_g(\theta)\|
      \leq\left(\frac{\|\phi(\theta)\|}{|\theta|}+\|d\phi(\theta)\|\right)\|f_\nu-f\|_{\mathcal F}.
\]
Boundedness of $d\phi$ also gives $\|\phi(\theta)\|/|\theta|\leq\|\phi(0)\|+\sup_\xi\|d\phi(\xi)\|$ for $|\theta|\geq1$. By \Cref{lemma_bound_diff_Phi_mu} and Lipschitz continuity of $R'$ on bounded sets, the perturbation error
\begin{equation}\locallabel{eq_bound_diff_g_gt}
    \delta_1(\epsilon):=
    \sup_{\substack{|\theta|\geq1\\ W_2(\mu,\nu)\leq\epsilon}}
    \left(\frac{|g_\nu(\theta)-g(\theta)|}{|\theta|}
      +\|J_{g_\nu}(\theta)-J_g(\theta)\|\right)
\end{equation}
therefore tends to zero as $\epsilon\to0$.

Conditions~\ref{conv_phi_asymp_hom} and~\ref{conv_dphi_asym_hom} control the two errors in \localeqref{conv_unif_gr} below. Testing the weak operator limits against $f$ and the coordinate vectors yields $g(r\varphi)/r\to g_\infty(\varphi)$ and convergence of $J_g(r\varphi)$, both uniformly on the sphere. To identify the Jacobian limit, consider $x\mapsto g(rx)/r$ on $1/2\leq|x|\leq2$. Writing $x=|x|\varphi$ shows that these functions converge uniformly to $|x|g_\infty(\varphi)$ and that their Jacobians converge uniformly there, since $r|x|\geq r/2$. The Jacobian limit is therefore the derivative of this homogeneous extension. On the sphere, its radial part is $g_\infty(\varphi)\varphi^\top$ and its tangential part is the spherical Jacobian $J_{g_\infty}^{\SS}(\varphi)$, which annihilates $\varphi$. It follows that $\delta_2(\bar r)\to0$ as $\bar r\to+\infty$, where
\begin{equation}\locallabel{conv_unif_gr}
    \begin{aligned}
    \delta_2(\bar r):=\sup_{r\geq\bar r,\ \varphi\in\SS^{d_\theta-1}}
    \Big(&\big|g(r\varphi)/r-g_\infty(\varphi)\big|\\
        &+\big\|J_g(r\varphi)-g_\infty(\varphi)\varphi^\top
                         -J^{\SS}_{g_\infty}(\varphi)\big\|\Big).
    \end{aligned}
\end{equation}

\paragraph{The escape region.} Fix $M>1$ and choose $\epsilon$ small enough and $\bar r\geq1$ large enough that
\[
    \delta:=\delta_1(\epsilon)+\delta_2(\bar r)
    \leq\min\left(\frac{\eta}{2}\frac{M^2-1}{M^2+1},\,
                     \frac{\beta}{2\sqrt{2}M^2}\right).
\]
Set
\[
    A=\left\{(w,\theta):\min(|w|,|\theta|)\geq\bar r,\quad
        \frac1M\leq\frac{|w|}{|\theta|}\leq M,\quad
        \left(\frac{w}{|w|},\frac{\theta}{|\theta|}\right)\in K\right\}.
\]
Along a characteristic in $A$, write $v_t=w_t/|w_t|$ and $\varphi_t=\theta_t/|\theta_t|$. Define the errors
\[
    \begin{aligned}
        e_t &= \frac{g_t(\theta_t)}{|\theta_t|}-g_\infty(\varphi_t),\\
        E_t &= J_{g_t}(\theta_t)-g_\infty(\varphi_t)\varphi_t^\top-J_{g_\infty}^{\SS}(\varphi_t).
    \end{aligned}
\]
By \localeqref{eq_bound_diff_g_gt} and \localeqref{conv_unif_gr}, $|e_t|+\|E_t\|\leq\delta$.

\paragraph{Radial growth.} If $(w_t,\theta_t)\in A$, the characteristic equations give
\[
    \begin{aligned}
        \frac{d}{dt}|w_t|
            &=-|\theta_t|\bigl(h_\infty(v_t,\varphi_t)+\langle e_t,v_t\rangle\bigr)
              \geq\frac{\eta}{2}|\theta_t|,\\
        \frac{d}{dt}|\theta_t|
            &=-|w_t|\bigl(h_\infty(v_t,\varphi_t)+\langle E_t\varphi_t,v_t\rangle\bigr)
              \geq\frac{\eta}{2}|w_t|.
    \end{aligned}
\]
Here we used $h_\infty(v_t,\varphi_t)\leq-\eta$ and $\delta\leq\eta/2$. Both inequalities point inward at the corresponding lower radial bounds.

\paragraph{Angular boundary.} Write $P_u=\Id-uu^\top$ for projection onto the tangent space at a unit vector $u$. The radial term of the limiting Jacobian vanishes after projection onto $\{\varphi_t\}^\perp$. With the gradients evaluated at $(v_t,\varphi_t)$, the angular equations become
\[
    \begin{aligned}
        \frac{|w_t|}{|\theta_t|}\dot v_t &= -\nabla_{\SS,v}h_\infty-P_{v_t}e_t,\\
        \frac{|\theta_t|}{|w_t|}\dot\varphi_t &= -\nabla_{\SS,\varphi}h_\infty-P_{\varphi_t}E_t^\top v_t .
    \end{aligned}
\]
When $(w_t,\theta_t)\in A$ and $(v_t,\varphi_t)\in\partial K$, the chain rule and the ratio bounds give
\[
    \frac{d}{dt}h_\infty(v_t,\varphi_t)
    \leq -\frac{|\nabla_{\SS}h_\infty|^2}{M}
            +\sqrt{2}M\delta|\nabla_{\SS}h_\infty|
    \leq -\frac{\beta^2}{2M}<0.
\]
The last inequality uses $|\nabla_{\SS}h_\infty|\geq\beta$ and the choice of $\delta$.

\paragraph{Ratio bounds.} We now show that $(w_t,\theta_t)$ cannot exit $A$ through either ratio boundary, $|w_t|/|\theta_t|=M$ or $|w_t|/|\theta_t|=1/M$. The radial equations yield
\[
    \begin{aligned}
        \frac{d}{dt}\left(\frac{|w_t|}{|\theta_t|}\right)
        &=h_\infty(v_t,\varphi_t)\left(\frac{|w_t|^2}{|\theta_t|^2}-1\right)\\
        &\quad+\frac{|w_t|^2}{|\theta_t|^2}\langle E_t\varphi_t,v_t\rangle
               -\langle e_t,v_t\rangle.
    \end{aligned}
\]
The last two terms have absolute value at most $\delta(1+|w_t|^2/|\theta_t|^2)$. Using $h_\infty(v_t,\varphi_t)\leq-\eta$ and the first bound in the choice of $\delta$ gives
\[
    \begin{aligned}
        \frac{d}{dt}\left(\frac{|w_t|}{|\theta_t|}\right)
            &\leq -\tfrac{\eta}{2}(M^2-1)<0
            &&\text{if } |w_t|/|\theta_t|=M,\\
        \frac{d}{dt}\left(\frac{|w_t|}{|\theta_t|}\right)
            &\geq \tfrac{\eta}{2}(1-M^{-2})>0
            &&\text{if } |w_t|/|\theta_t|=1/M.
    \end{aligned}
\]

\paragraph{Conclusion.} The strict inequalities on the radial, angular and ratio boundaries rule out a first exit from $A$. Along every characteristic initialized there, the radial growth estimate gives
\[
    F'(\mu_t)(w_t,\theta_t)=-|w_t|\frac{d}{dt}|w_t|\leq-\eta|w_t||\theta_t|/2
       \leq-\eta\bar r^2/2.
\]
Since $-\eta$ is an attained regular value, $K$ has nonempty interior. Together with $M>1$, this gives $A$ nonempty interior, so $A$ is an escape region.

The radial estimates and ratio bounds give
\[
    \frac{d}{dt}|w_t|^2\geq\frac{\eta}{M}|w_t|^2,\qquad
    \frac{d}{dt}|\theta_t|^2\geq\frac{\eta}{M}|\theta_t|^2.
\]
Both parameter norms thus grow exponentially along these characteristics.
For a Wasserstein gradient flow, the Lagrangian representation then gives an exponentially growing contribution to the second moment from any positive initial mass in $A$.

\section{Examples}\label{sec_examples}
In this section, we show that the two examples mentioned in the introduction fit into the abstract frameworks of \Cref{sec_wellp_wgf,sec_glob_conv}, so that \Cref{thm_wgf,thm_glob_conv} apply.

\paragraph{Model space and loss function.} We consider two settings, namely regression with the square loss and classification with the cross-entropy loss. Given a data distribution $\rho\in\mathcal{P}(\RR^{\din}\times \RR^{\dout})$ over pairs of features and labels, we minimize the expected risk $R:f\mapsto\int_{\RR^\din\times\RR^\dout} \ell(f(x),y) d\rho(x,y)$, where $\ell$ is either the square loss $\ell:(z,y)\mapsto (1/2)|z-y|^2$ or the cross-entropy loss
\begin{equation*}
\ell:(z,y)\mapsto \left\{
\begin{aligned}
    &\textstyle-\langle z,y\rangle +\log(\sum_{1\leq i\leq \dout}e^{z_i})&\mathrm{if}~\dout>1,\\
    &-yz+\log(1+e^z)&\mathrm{if}~\dout=1.
\end{aligned}\right.
\end{equation*}
We denote by $\rho_x$ and $\rho_y$ the marginals of $\rho$ on $\RR^{\din}$ and $\RR^{\dout}$, and $\rho_{y|x}$ the conditional distribution of labels given a feature $x$. In the classification case, if $\dout>1$, we ask that $\rho_y$ is supported on the set of one-hot labels $\mathcal{B}_{\dout}$, namely, the canonical basis of $\RR^\dout$, where $k=\dout$ is the number of classes and the $i$-th element of $\mathcal{B}_k$ is associated with the $i$-th class. If $\dout=1$, we ask that $\rho_y$ is supported on $\{0,1\}$. \Cref{lem_ass_loss} below, whose proof is given in \Cref{appendix_examples}, shows that if $\mathcal{F} = L^2(\rho_x;\RR^{\dout})$ then in both cases $R$ fits in our abstract framework.
\begin{lem}\label{lem_ass_loss}
    In the following two cases:
    \begin{enumerate}[label=(\roman*)]
        \item $\ell$ is the square loss and $\rho_y$ has a finite moment of order $2$,\label{cond_square}
        \item $\ell$ is the cross-entropy loss and either $\dout>1$ with $\rho_y$ supported on $\mathcal{B}_{\dout}$ or $\dout=1$ with $\rho_y$ supported on $\{0,1\}$,\label{cond_cross_entropy}
    \end{enumerate}
    the functional $R$ is convex and differentiable on $\mathcal{F} = L^2(\rho_x;\RR^{\dout})$. Moreover, its differential is bounded on sublevel sets and Lipschitz.
\end{lem}

\subsection{Two-layer fully connected networks}
The mapping implemented by a two-layer fully connected neural network with $m$ hidden neurons and activation function $\sigma$ is $(1/m)\sum_{i=1}^m \Phi(w_i,\theta_i)$, where $\Phi(w,\theta):x\mapsto \sigma(\langle\theta,(x, 1)\rangle)w$, $d_{\theta}=\din+1$ and $d_w=\dout$. The first part of \Cref{lem_nn_ass_wellp_wgf} below states that the assumptions on $\phi$ in \Cref{ass_glob_conv} hold under suitable assumptions on the activation function and the data distribution. The second part ensures that, under slightly stronger assumptions, $\phi$ is $C^{2,1}$, so that \Cref{prop_esc_reg_nondegenerate} and \Cref{cor_esc_reg_nondegenerate_bis} can be applied.
\begin{lem}\label{lem_nn_ass_wellp_wgf}
    If $\sigma\in C^{1,1}(\RR)$ with $\sigma'$ bounded and $\rho_x$ has a finite moment of order $4$, then $\phi$ is differentiable with a bounded Lipschitz differential. If $\sigma$ is, in addition, twice differentiable with a bounded Lipschitz second derivative and $\rho_x$ has a finite moment of order $6$, then $\phi$ is twice differentiable with a bounded Lipschitz second differential.
\end{lem}
The proof of the first statement follows by arguing as in \cite[Lemma D.2]{chizat2018global}, which deals with sigmoid activations. A closer inspection reveals that the bounds only depend on the supremum norm and the Lipschitz constant of $\sigma'$, so that the result also applies to any differentiable activation with bounded Lipschitz derivative. The second part of the statement can be obtained in the same way.

As a consequence of \Cref{lem_nn_ass_wellp_wgf}, we immediately obtain the following corollary.
\begin{cor}\label{cor_nn_ass}
    Assume $\sigma\in C^{1,1}(\RR)$ with $\sigma'$ bounded and $\rho_x$ has a finite moment of order $4$. If $\ell$ is the square loss and $\rho_y$ has a finite moment of order $2$, or if $\ell$ is the cross-entropy loss and either $\dout>1$ with $\rho_y$ supported on $\mathcal{B}_{\dout}$ or $\dout=1$ with $\rho_y$ supported on $\{0,1\}$, then \Cref{ass_glob_conv} (and hence \Cref{ass_wellp_wgf}) holds.
\end{cor}

We end this section with the following results, whose proofs are postponed to \Cref{appendix_fcn}. They give sufficient assumptions ensuring that, whenever the mean-field gradient flow for a two-layer fully connected network converges in $W_2$, its limit is a global minimizer. A list of common activation functions to which the results apply is given in \Cref{table_activations}.

\begin{table}[h]
    \centering
    \begin{tabular}{|c|c|}
        \hline
        \Cref{prop_glob_conv_sigmoid} & sigmoid, tanh, softsign \\
        \hline
         \Cref{prop_glob_conv_gelu} & GELU, SiLU, Mish, ELU, softplus, logsigmoid\\
         \hline
         \Cref{prop_glob_conv_mlp_nondegenerate} & sigmoid, tanh, GELU, SiLU, Mish, softplus, logsigmoid\\
         \hline
    \end{tabular}
    \caption{List of common activation functions to which the conditional global convergence results of \Cref{prop_glob_conv_sigmoid,prop_glob_conv_gelu,prop_glob_conv_mlp_nondegenerate} apply (see the \href{https://docs.pytorch.org/docs/2.13/nn.html\#non-linear-activations-weighted-sum-nonlinearity}{PyTorch documentation} for definitions of these functions).}
    \label{table_activations}
\end{table}

\begin{prop}[Bounded activations and scalar output weights]\label{prop_glob_conv_sigmoid}

    Let $d_w=\dout=1$ and $\sigma\in C^{1,1}(\RR)$ be such that $\sigma'$ is integrable and $s\sigma'(s)\to 0$ as $|s|\to +\infty$. Assume that $\ell$ is either
    \begin{itemize}
        \item the square loss with $\rho_y$ having a finite moment of order $2$ and $x\mapsto \int_{\RR} y\,d\rho_{y|x}(y)$ having a bounded and continuous representative on the support of $\rho_x$,
        \item the cross-entropy loss with $\rho_y$ supported on $\{0,1\}$ and $x\mapsto \rho_{y|x}(\{i\})$ having a continuous representative on the support of $\rho_x$ for $i=0,1$.
    \end{itemize}
      If $\rho_x\in\mathcal{P}(\RR^\din)$ has a finite moment of order $4$ and has a continuous density $p_x$ with respect to the Lebesgue measure satisfying $p_x(x)\leq C(1+|x|)^{-p}$ for some $C>0$ and $p>{\din}+2$, then conditions \ref{conv_c1_gmu_bounded_scalar} to \ref{conv_dgmu_dgnu} of \Cref{prop_esc_reg_bounded_scalar} hold for every $\mu\in\mathcal{P}_2(\Omega)$. Moreover, suppose that condition \ref{reg_val_bounded_scalar} or \ref{reg_val_inf_bounded_scalar} of \Cref{prop_esc_reg_bounded_scalar} holds for every measure that does not minimize $F$, and that the initial distribution of parameters is sub-Gaussian and has full support. Then, if the Wasserstein gradient flow of $F$ converges in $W_2$, its limit is a global minimizer.
\end{prop}

\begin{prop}[Asymptotically positively one-homogeneous activations]\label{prop_glob_conv_gelu}
    Let $\sigma\in C^{1,1}(\RR)$ be such that $\lim_{s\to+\infty}\sigma'(s)= a$ and $\lim_{s\to-\infty}\sigma'(s)= b$ with $a\neq b$. Assume that $\ell$ is the square loss and $\rho_y$ has a finite moment of order $2$, or that $\ell$ is the cross-entropy loss and either $\dout>1$ with $\rho_y$ supported on $\mathcal{B}_{\dout}$ or $\dout=1$ with $\rho_y$ supported on $\{0,1\}$. If $\rho_x$ is absolutely continuous with respect to the Lebesgue measure and has a finite moment of order $4$, then conditions \ref{conv_phi_asymp_hom} to \ref{density_phi_asymp_hom} hold. Moreover, suppose that $h_{\mu}^{\infty}$ has a nonzero regular value for every $\mu$ that does not minimize $F$, and that the initial distribution of parameters is sub-Gaussian and has full support. Then, if the Wasserstein gradient flow of $F$ converges in $W_2$, its limit is a global minimizer.
\end{prop}

\begin{prop}[Activations with at most linear growth under a non-degeneracy condition]\label{prop_glob_conv_mlp_nondegenerate}
    Let $\sigma$ be twice differentiable with a bounded first derivative and a bounded Lipschitz second derivative. Assume that $\ell$ is the square loss and $\rho_y$ has a finite moment of order $2$, or that $\ell$ is the cross-entropy loss and either $\dout>1$ with $\rho_y$ supported on $\mathcal{B}_{\dout}$ or $\dout=1$ with $\rho_y$ supported on $\{0,1\}$. Further assume that $\rho_x\in\mathcal{P}(\RR^\din)$ has a finite moment of order $6$ and that, for every $\mu\in\mathcal{P}_2(\Omega)$ that does not minimize $F$, there exists a non-degenerate local maximizer of $\theta\mapsto (1/2)|g_{\mu}(\theta)|^2$. Then, provided that the initial distribution of parameters is sub-Gaussian and has full support and that the Wasserstein gradient flow of $F$ converges in $W_2$, the limit of the flow is a global minimizer.
\end{prop}
The proof of this last proposition follows directly by applying \Cref{cor_nn_ass,cor_esc_reg_nondegenerate_bis}.

\subsection{Multi-head attention layers}\label{sec_attention_layers}
\subsubsection{Model and conditional convergence results}
We consider a multi-head attention layer taking input contexts of size $n$ composed of tokens in $\RR^d$. The number of attention heads is $m$, and the query, key, and value matrices associated with the $i$-th head are $(Q_i,K_i,V_i)\in (\RR^{k\times d})^3$. In a next-token prediction setting, the implemented predictor is $(1/m)\sum_{i=1}^m V_i\psi(A_i)$, where $A_i=K_i^TQ_i$ and $\psi(A):\RR^{n\times d}\to \RR^d$ is defined by
\begin{equation*}
    \psi(A):X=\begin{pmatrix}x_1~\cdots~x_n\end{pmatrix}^T\mapsto \sum\limits_{i=1}^n \frac{e^{\langle Ax_n,x_i\rangle}}{\sum\limits_{j=1}^n e^{\langle Ax_n,x_j\rangle}}x_i=\sum\limits_{i=1}^n \sigma_i(XAx_n)x_i=X^T\sigma(XAx_n),
\end{equation*}
where $\sigma:z\in\RR^n\mapsto (e^{z_i}/(\sum_{j=1}^n e^{z_j}))_{i=1}^n\in \RR^n$ is the softmax function.

This setting is covered by our abstract framework by taking $d_w=kd$, $d_{\theta}=d^2$, $\din=nd$ and $\dout=k$, with $w$ the vectorized version of $V$ and $\theta$ the vectorized version of $A=K^TQ$. The mapping $\phi$ is then defined by the relation $\phi(\theta)w=V\psi(A)$. For this reason, properties of $\psi$ (boundedness, regularity, etc.) directly transfer to $\phi$. Since $\psi$ has a simpler expression, we use it rather than $\phi$ in our computations.

\begin{rem}[Parameterization of attention]
Transformers combine attention with other important operations, including layer normalization, residual connections, and feed-forward blocks. Here we deliberately isolate a single softmax self-attention layer and study the limit of a large number of heads. We also use a slightly relaxed parameterization. If $(Q,K,V)$ are the query, key, and value matrices of one head, its output depends on $Q$ and $K$ only through $A=K^TQ$. Directly optimizing the factors $(Q,K)$ introduces an additional non-convex factorization and gradients that no longer satisfy our global Lipschitz assumptions. We therefore optimize over $(A,V)$ and leave the standard factorized parameterization to future work. When $k<d$, allowing arbitrary $A\in\RR^{d\times d}$ also removes the rank constraint inherited from $A=K^TQ$.
\end{rem}

The first part of the following lemma, whose proof is postponed to \Cref{appendix_examples}, shows that, under a finite moment assumption on the distribution of input contexts $\rho_X\in\mathcal{P}(\RR^{n\times d})$, the regularity assumptions on $\phi$ in \Cref{ass_glob_conv} are satisfied. The second part shows that, under slightly stronger assumptions, $\phi$ is $C^{2,1}$, so that \Cref{prop_esc_reg_nondegenerate} and \Cref{cor_esc_reg_nondegenerate_bis} apply.
\begin{lem}\label{lem_reg_transfo}
    If $\rho_X$ has a finite moment of order $10$, then $\phi$ is differentiable with a bounded Lipschitz differential. Moreover, if $\rho_X$ has a finite moment of order $14$, then $\phi$ is twice differentiable with a bounded Lipschitz second differential.
\end{lem}
\begin{cor}\label{cor_ass_glob_conv_attention}
    Assume $\rho_X$ has a finite moment of order $10$. If $\ell$ is the square loss and $\rho_y$ has a finite moment of order $2$, or if $\ell$ is the cross-entropy loss and either $\dout>1$ with $\rho_y$ supported on $\mathcal{B}_{\dout}$ or $\dout=1$ with $\rho_y$ supported on $\{0,1\}$, then \Cref{ass_glob_conv} (and hence \Cref{ass_wellp_wgf}) hold.
\end{cor}

As in the case of fully connected networks, we finally provide two propositions, whose proofs are given in \Cref{appendix_attention}, that summarize our conditional global convergence results. As $d_w=kd$ and $d_\theta=d^2$, we cannot have $d_w=1$ and $d_\theta>1$. Hence, \Cref{prop_esc_reg_bounded_scalar} does not apply and \Cref{prop_glob_conv_attention} below follows from the $d_\theta=1$ statement proved in \Cref{appendix_bounded_scalar} (see \Cref{rem_bounded_scalar_scalar}). As a result, \Cref{prop_glob_conv_attention} only covers scalar input and output tokens.

\begin{prop}\label{prop_glob_conv_attention} Let $k=d=1$. Assume that $\ell$ is the square loss and $\rho_y$ has a finite moment of order $2$, or that $\ell$ is the cross-entropy loss and $\rho_y$ is supported on $\{0,1\}$.
Further assume that $\rho_X\in\mathcal{P}(\RR^{n\times d})$ has a finite moment of order $10$. Moreover, suppose that $g_{\mu}$ has a nonzero regular value for every $\mu$ that does not minimize $F$, and that the initial distribution of parameters is sub-Gaussian and has full support. Then, if the Wasserstein gradient flow of $F$ converges in $W_2$, its limit is a global minimizer.
\end{prop}

\begin{prop}\label{prop_glob_conv_attention_bis}
Assume that $\ell$ is the square loss and $\rho_y$ has a finite moment of order $2$, or that $\ell$ is the cross-entropy loss and either $\dout>1$ with $\rho_y$ supported on $\mathcal{B}_{\dout}$ or $\dout=1$ with $\rho_y$ supported on $\{0,1\}$. Further assume that $\rho_X\in\mathcal{P}(\RR^{n\times d})$ has a finite moment of order $14$ and that, for every $\mu\in\mathcal{P}_2(\Omega)$ that does not minimize $F$, there exists a non-degenerate local maximizer of $\theta\mapsto (1/2)|g_{\mu}(\theta)|^2$. Then, provided that the initial distribution of parameters is sub-Gaussian and has full support and that the Wasserstein gradient flow of $F$ converges in $W_2$, its limit is a global minimizer.
\end{prop}
The proof of this last proposition follows directly by applying \Cref{cor_ass_glob_conv_attention,cor_esc_reg_nondegenerate_bis}.

\subsubsection{Uniform convergence towards hardmax attention}
In this subsection, motivated by condition \ref{conv_c1_gmu_bounded_scalar} of \Cref{prop_esc_reg_bounded_scalar}, we discuss the convergence of softmax attention towards hardmax attention. The scalar-output requirement $d_w=kd=1$ forces $k=d=1$ and hence $d_\theta=1$. We therefore use the one-dimensional extension in \Cref{rem_bounded_scalar_scalar}, which does not require condition \ref{conv_c1_gmu_bounded_scalar} of \Cref{prop_esc_reg_bounded_scalar}, so we could avoid the discussion below. However, we believe this convergence could be of independent interest and could prove useful if a result similar to \Cref{prop_esc_reg_bounded_scalar} were to be proved for the $d_w>1$ case in future work.

Below, we study the convergence, as $A$ becomes large, of the parameter-to-predictor mapping associated with softmax attention to the corresponding mapping for hardmax attention. We prove a uniform convergence result for the mapping itself. To our knowledge, this result is new and could be of independent interest. The convergence of the differential is more subtle and is also discussed.

Given $z\in\RR^n$, we define $z_{\max}\eqdef\max_{1\leq i\leq n}z_i$ and $\Sigma(z)\eqdef\{i\in\{1,\ldots,n\}\,\rvert\,z_i=z_{\max}\}$. We also define $\sigma_r:z\mapsto \sigma(rz)$ and $\sigma_{\infty}:z\mapsto \lim_{r\to+\infty}\sigma(rz)=(1/\#\Sigma(z))\mathbf{1}_{\Sigma(z)}$, where $\mathbf{1}_{\Sigma(z)}$ is the $n$-dimensional vector with $i$-th entry $1$ if $i\in\Sigma(z)$ (that is, if $z_i=z_{\max}$) and $0$ otherwise. Abusing notation, we denote by $\mathbb{S}^{d^2-1}$ the set of $d\times d$ matrices with $\|A\|=1$, where $\|\cdot\|$ is the Frobenius norm. Finally, we define $\psi_r$ and $\psi_{\infty}$ on $\mathbb{S}^{d^2-1}$ by $\psi_r(A)=\psi(rA)$ and
\[
    \psi_{\infty}(A)(X)=\lim_{r\to+\infty}\psi(rA)(X)=\sum\limits_{i=1}^n [\sigma_{\infty}(XAx_n)]_i x_i=\frac{1}{\#\Sigma(XAx_n)}\sum\limits_{i\in \Sigma(XAx_n)} x_i.
\]
We stress that, when $\Sigma(XAx_n)=\{i\}$ for some $i\in\{1,\ldots,n\}$ (that is, when $j=i$ is the unique maximizer of $j\mapsto\langle Ax_n,x_j\rangle$), then $\psi_{\infty}(A)(X)=x_i$.

We can now state the following proposition, whose proof is postponed to \Cref{appendix_attention}.
\begin{prop}\label{prop_conv_attention}
    Assume that $\rho_X$ has a bounded continuous density with respect to the Lebesgue measure on $\RR^{n\times d}$ and has a finite moment of order $2$. Then
    \begin{equation*}
        \underset{r\to+\infty}{\lim}\sup_{A\in\mathbb{S}^{d^2-1}} \|\psi(rA)-\psi_\infty(A)\|_{L^2(\rho_X;\RR^d)}=0.
    \end{equation*}
\end{prop}

If $\rho_X$ has a finite moment of order $3$, one can show that for every bounded and continuous function ${f:\RR^{n\times d}\to \RR^d}$ the mapping $g_f:A\mapsto \int_{\RR^{n\times d}}\langle f(X),\psi(A)(X)\rangle d\rho_X(X)$ is of class $C^1$ and
\begin{equation*}
    \nabla g_f(A)=\int_{\RR^{n\times d}}\sum\limits_{i=1}^n\sum\limits_{\substack{j=1\\j\neq i}}^n \sigma_i(XAx_n)\sigma_j(XAx_n)\langle f(X),x_i\rangle (x_i-x_j)x_n^T d\rho_X(X).
\end{equation*}

Assume that $\rho_X$ has density $p_X$ with respect to the Lebesgue measure. Provided that ties between three or more maximizers are $\mathcal{H}^{nd-1}$-negligible and that passage to the limit in the coarea integrals is justified, any limit of the mapping $A\mapsto r\nabla g_f(rA)$ as $r\to+\infty$ must be given by
    \begin{equation}\label{limit_grad_attention}
        A\mapsto\sum\limits_{i=1}^n\sum\limits_{\substack{j=1\\j\neq i}}^n \int_{\Gamma_{ij}(A)}\frac{1}{\alpha_{ij}(A,X)}\langle f(X),x_i\rangle p_X(X)(x_i-x_j)x_n^T d\mathcal{H}^{nd-1}(X),
    \end{equation}
    where
    \begin{equation*}
        \Gamma_{ij}(A)=\{X\in\RR^{n\times d}\,\rvert\,\langle Ax_n,x_i-x_j\rangle=0~\mathrm{and}~\langle Ax_n,x_i-x_k\rangle\geq 0~\mathrm{for~all}~k\in\{1,...,n\}\}
    \end{equation*}
    and
    \begin{equation}\label{eq_jac_coarea}
        \alpha_{ij}(A,X)=|\nabla_X \langle Ax_n,x_i-x_j\rangle|=\begin{cases}\sqrt{2|Ax_n|^2+|A^T(x_i-x_j)|^2}&\mathrm{if}~i\neq n~\mathrm{and}~j\neq n,\\\sqrt{|Ax_n|^2+|A^T(x_n-x_j)+Ax_n|^2}&\mathrm{if}~i=n,\\\sqrt{|Ax_n|^2+|A^T(x_i-x_n)-Ax_n|^2}&\mathrm{if}~j=n.\end{cases}
    \end{equation}
Existence of a limit alone does not justify this identification: multiple ties and contributions from the tails must also be controlled. We leave as an open question the determination of sufficient conditions on $\rho_X$ and on a set of matrices $A$ that justify \eqref{limit_grad_attention} and ensure uniform convergence. This is nontrivial even for smooth compactly supported densities, because of rank-one obstructions. For example, take $n=d=2$, $A=e_1e_1^T$, $f\equiv e_2$, and write $u=x_2$, $v=x_1-x_2$. Choose a smooth compactly supported density $p(u,v)$, positive near $(u_1,v_1,u_2,v_2)=(0,0,1,1)$ and supported where $u_2,v_2>0$. With $\varsigma(s)=(1+e^{-s})^{-1}$, one obtains
\[
    [r\nabla g_f(rA)]_{22}
    =\int_{\RR^4}r\varsigma'(ru_1v_1)u_2v_2^2p(u,v)\,du\,dv
    \gtrsim\log r\qquad(r\to+\infty).
\]
Indeed, for a sufficiently small fixed $\delta>0$, integrating over $(r\delta)^{-1}\leq u_1\leq\delta$, $|v_1|\leq(ru_1)^{-1}$, and a fixed small neighborhood of $u_2=v_2=1$ gives a lower bound proportional to $\int_{(r\delta)^{-1}}^\delta du_1/u_1$. Thus no finite derivative limit exists at this rank-one matrix.

\section{Conclusion}
In this work, we studied conditional global convergence of gradient flow for wide shallow models. We provided a blueprint for ruling out convergence in $W_2$ to measures that are not global minimizers, based on the construction of an escape region in the parameter space. We completed the escape-region construction for bounded nonlinearities and scalar output weights proposed in \cite{chizat2018global}. We also proposed a new construction for nonlinearities with at most linear growth under a non-degeneracy assumption, and proved the existence of an escape region for asymptotically positively $1$-homogeneous nonlinearities.

A better understanding of the instability mechanism for bounded nonlinearities with vector output weights is a major open question. Even if one could ensure that the non-degeneracy assumption holds under mild assumptions on the learning task, the escape region in this case is, in some sense, much smaller than in the other examples. Proving the existence of an escape region similar to the scalar case would be desirable.

There are several other directions for future work. Concerning the training of attention layers, other variants of attention, such as masked attention, could be considered. An interesting but challenging question is to study the traditional $(Q,K,V)$ parameterization, instead of optimizing over $(A,V)$ where $A=K^TQ$. One could also investigate whether the abstract framework in which our results are proved can handle contexts with variable, and possibly arbitrarily large, size. For the training of transformers, the Adam algorithm seems to significantly outperform gradient descent. Studying its global convergence properties is a highly interesting challenge.

\section*{Acknowledgments}

The work of G. Peyr\'e was supported by the French government under the management of Agence Nationale de la Recherche as part of the ``Investissements d’avenir'' program, reference ANR-19-P3IA-0001 (PRAIRIE 3IA Institute). The work of G. Peyr\'e and R. Petit was supported by the European Research Council (ERC project WOLF).

\section*{AI use statement}
AI tools were used for proofreading and to explore proof strategies for certain technical lemmas. The authors verified and refined all AI-assisted outputs and take full responsibility for the content of this work.

\bibliography{ref}
\bibliographystyle{abbrv}


\appendix

\section{Proof of Theorem \ref{thm_wgf}}\label{appendix_wgf}
To prove \Cref{thm_wgf}, we show the existence and uniqueness of a weak solution to the continuity equation
\begin{equation*}
    \partial_t \mu_t=-\mathrm{div}(\mu_t v_t)~~\mathrm{with}~~v_t(u)=-\nabla F'(\mu_t)(u)
\end{equation*}
by a fixed point argument. More precisely, a curve $\mu\in AC^2_{\mathrm{loc}}([0,+\infty);\mathcal{P}_2(\Omega))$ is understood as a weak solution starting from $\mu_0$ if $\mu_{t=0}=\mu_0$, $\int_0^T\int_\Omega |v_t(u)|^2d\mu_t(u)dt<+\infty$ for every $T>0$, and, for every test function $\varphi\in C_c^1((0,+\infty)\times\Omega;\RR)$, we have
\begin{equation*}
    \int_{0}^{+\infty}\int_{\Omega}[\partial_t\varphi_t(u)+\langle \nabla\varphi_t(u),v_t(u)\rangle]d\mu_t(u)dt=0.
\end{equation*}
In \Cref{appendix_wgf_prelim}, we introduce some useful notions that are used in the proof. The core of the proof is given in \Cref{appendix_wgf_core}, while the technical lemmas on which it relies are proved in \Cref{subsec_technical_lemmas}.

\subsection{Preliminaries}\label{appendix_wgf_prelim}
If $\mu\in\mathcal{P}(\Omega)$, we denote by $m_2(\mu)=\int_{\Omega}|u|^2 d\mu(u)$ its second-order moment. We denote by $\mathcal{P}_2(\Omega)$ the set of probability measures on $\Omega$ that have finite second-order moments. We recall that $(\mathcal{P}_2(\Omega),W_2)$ is a complete metric space.

\paragraph{Sub-Gaussian norm.} We define the sub-Gaussian norm of a probability measure $\mu\in\mathcal{P}(\Omega)$ by
\[
    \|\mu\|_{\psi_2}=\inf\{c>0\,\rvert\, \int_{\Omega}e^{|u|^2/c^2}d\mu(u)\leq 2\}
\]
and say that it is sub-Gaussian if $\|\mu\|_{\psi_2}<+\infty$. We stress that, when $d\geq 2$, this definition differs from the usual definition of sub-Gaussianity for random vectors, which can be found for instance in \cite[Section 3.4]{vershyninHighDimensionalProbabilityIntroduction2026}. Namely, if $X$ is a random vector with law $\mu$, then with our definition
\[
    \|\mu\|_{\psi_2}=\inf\{c>0\,\rvert\, \mathbb{E}[\exp(|X|^2/c^2)]\leq 2\},
\]

whereas the sub-Gaussian norm of $X$ is usually defined as $\sup_{|v|\leq 1}\inf\{c>0\,\rvert\, \mathbb{E}[\exp((v^TX)^2/c^2)]\leq 2\}$. Our choice is motivated by the following lemma, which shows that with our definition, the pushforward operation by a map with linear growth preserves sub-Gaussianity.
\begin{lem}
    \label{lem_subg_pushf}
    If $\mu\in\mathcal{P}(\Omega)$ is sub-Gaussian and $T:\Omega\to\Omega$ satisfies $|T(u)|\leq C(1+|u|)$ for some $C>0$, we have
    \begin{equation}
        \|T_{\#}\mu\|_{\psi_2}\leq 2C\max\Big(\|\mu\|_{\psi_2},1/\sqrt{\ln(2)}\Big).
    \end{equation}
\end{lem}
\begin{proof}
    Let $X$ be a random vector with law $\mu$. Fix $c>\|\mu\|_{\psi_2}$ such that $\mathbb E[\exp(|X|^2/c^2)]\leq2$, and set $M:=\max(c,1/\sqrt{\ln(2)})$ and $K:=2CM$. Using $|T(X)|\leq C(1+|X|)$ and $(1+a)^2\leq2(1+a^2)$, we obtain
    \begin{equation*}
        \begin{aligned}
            \mathbb E\left[\exp\left(\frac{|T(X)|^2}{K^2}\right)\right]
            &\leq
            \exp\left(\frac{1}{2M^2}\right)
            \mathbb E\left[\exp\left(\frac{|X|^2}{2M^2}\right)\right]\\
            &\leq
            \exp\left(\frac{\ln(2)}{2}\right)
            \mathbb E\left[\exp\left(\frac{|X|^2}{c^2}\right)\right]^{1/2}
            \leq 2.
        \end{aligned}
    \end{equation*}
    Here we used $M\geq1/\sqrt{\ln(2)}$ for the first factor, and Jensen's inequality applied to the concave function $s\mapsto\sqrt{s}$ for the second factor. Hence, $\|T_{\#}\mu\|_{\psi_2}\leq 2C\max(c,1/\sqrt{\ln(2)})$. Letting $c\downarrow\|\mu\|_{\psi_2}$ concludes the proof.
\end{proof}

\paragraph{Continuity of the predictor for $W_2$.} The following lemma shows that the measure-to-predictor mapping $\mu\mapsto\int_{\Omega}\Phi d\mu$ is continuous with respect to the $W_2$ distance.
\begin{lem}\label{lemma_bound_diff_Phi_mu}
    Under \Cref{ass_wellp_wgf}, for every $\mu_1,\mu_2\in\mathcal{P}_2(\Omega)$, we have
    \begin{equation*}
        \textstyle \|\int_{\Omega}\Phi d\mu_1-\int_{\Omega}\Phi d\mu_2\|_{\mathcal{F}}\lesssim (1+m_2(\mu_1)+m_2(\mu_2))W_2(\mu_1,\mu_2).
    \end{equation*}
\end{lem}
\begin{proof}
    Considering $\Pi$ an optimal transport plan (for the quadratic cost) between $\mu_1$ and $\mu_2$, we obtain
    \begin{equation*}
        \begin{aligned}
            \textstyle\|\int_{\Omega}\Phi d\mu_1-\int_{\Omega}\Phi d\mu_2\|&\textstyle=\|\int_{\Omega\times\Omega} [\Phi(u_1)-\Phi(u_2)]d\Pi(u_1,u_2)\|\\
            &\textstyle\leq \int_{\Omega\times\Omega} \|\Phi(u_1)-\Phi(u_2)\|d\Pi(u_1,u_2)\\
            &\textstyle\lesssim \int_{\Omega\times\Omega} (1+\mathrm{max}(|u_1|,|u_2|))|u_1-u_2|d\Pi(u_1,u_2)\\
            &\textstyle\leq \sqrt{\int_{\Omega\times\Omega}(1+\mathrm{max}(|u_1|,|u_2|))^2d\Pi(u_1,u_2)}\sqrt{\int_{\Omega\times\Omega}|u_1-u_2|^2d\Pi(u_1,u_2)}\\
            &\textstyle =\sqrt{\int_{\Omega\times\Omega}(1+\mathrm{max}(|u_1|,|u_2|))^2d\Pi(u_1,u_2)}W_2(\mu_1,\mu_2)\\
            &\textstyle\lesssim (1+m_2(\mu_1)+m_2(\mu_2))W_2(\mu_1,\mu_2).
        \end{aligned}
    \end{equation*}
\end{proof}

\subsection{Core of the proof}\label{appendix_wgf_core}
Throughout this subsection, we work under \Cref{ass_wellp_wgf}. We set an initial condition $\bar{\mu}_0\in\mathcal{P}_2(\Omega)$ and $\alpha>0$ such that $F(\bar{\mu}_0)\leq \alpha$.

\paragraph{Truncation of the objective.} We claim that there exists a $C^{1,1}$ function $\xi$ such that $\xi(x)=x$ if $x\in[0, \alpha]$, $\xi$ is (strictly) increasing on $[\alpha,2\alpha]$, $\xi(x)=2\alpha$ if $x\geq 2\alpha$, ${\|\xi'\|_{\infty}\leq 3/2}$ and $\mathrm{Lip}(\xi')\leq 4/\alpha$. One can check that a possible choice for $\xi_{\rvert[\alpha,2\alpha]}$ is
\begin{equation*}
    \xi(x)=x+\frac{(x-\alpha)(2\alpha-x)(1-\cos(\pi(x-\alpha)/\alpha))}{2\alpha}.
\end{equation*}
We define $\tilde{R}=\xi\circ R$ and $\tilde{F}=\xi\circ F$. The purpose of this truncation is to obtain global bounds during the construction of the flow. Indeed, the assumptions on $R$ give boundedness of $R'$ on sublevel sets, and such bounds would be available along a gradient flow by energy dissipation. However, at this stage the gradient flow has not yet been constructed, so these a priori bounds cannot be used in the fixed point argument. The truncated objective $\tilde{F}$ avoids this circularity by making the corresponding differential globally bounded while leaving the dynamics unchanged on the relevant sublevel set. More precisely, $\tilde{F}=F$ on ${\{F\leq \alpha\}=\{\tilde{F}\leq \alpha\}}$. We will construct a gradient flow for $\tilde F$ starting from $\bar\mu_0$; energy dissipation then gives $\tilde F(\mu_t)\leq \tilde F(\bar\mu_0)\leq\alpha$, so the curve remains in the region where $\tilde F=F$ and is therefore also a Wasserstein gradient flow for $F$.

\paragraph{Construction of the flow map.} Let $T>0$ and $\mu\in C([0,T];\mathcal{P}_2(\Omega))$. The continuity of $\mu$ implies that $t\mapsto m_2(\mu_t)$ is continuous, so that $m_2(\mu)=\sup_{0\leq t\leq T}m_2(\mu_t)$ is finite.

There exists a unique solution to
    \begin{equation*}\left\{
        \begin{aligned}
            &\dot r(t)=-\nabla \tilde{F}'(\mu_t)(r(t)),\qquad 0\leq t\leq T,\\
            &r(0)=u.
        \end{aligned}\right.
    \end{equation*}

Indeed, \Cref{lem_cauchy_prop_velocity} below shows that the velocity $v(t,u)=-\nabla \tilde{F}'(\mu_t)(u)$ is continuous, locally Lipschitz in space, and has linear growth in space (uniformly in time), so standard ODE theory applies. We define the associated flow map $X_t[\mu]:u\mapsto r(t)$.
The relation between these characteristic maps and weak solutions of the continuity equation is provided by \Cref{lem_lagrangian_representation}.

\paragraph{Fixed point argument on $[0,T]$.} We set an initial condition $\bar{\mu}_0\in\mathcal{P}_2(\Omega)$ and define the map
\begin{equation*}
    \begin{aligned}
        \mathcal{G}\colon C([0,T];\mathcal{P}_2(\Omega)) &\to C([0,T];\mathcal{P}_2(\Omega))\\
        \mu &\mapsto (X_t[\mu]_{\#}\bar{\mu}_0)_{0\leq t\leq T}.
    \end{aligned}
\end{equation*}
Our aim is to find a complete metric space $\calX\subset C([0,T];\mathcal{P}_2(\Omega))$ such that $\mathcal{G}(\calX)\subset \calX$ and $\mathcal{G}$ is a contraction on $\calX$ when $T$ is sufficiently small. This will allow us to apply a fixed point argument which will yield the existence and uniqueness of the solution on $[0,T]$. In \Cref{subsec_technical_lemmas} below, we show that there exists a continuous function $f:(\RR_+)^2\to\RR_+$ that is non-decreasing in each variable and such that, if $\bar{\mu}_0$ is sub-Gaussian and
\begin{equation}
    \calX=\{\mu\in C([0,T];\mathcal{P}_2(\Omega))\,\rvert\, \forall t\in [0,T],~\|\mu_t\|_{\psi_2}\leq f(t, \|\bar{\mu}_0\|_{\psi_2})\}
    \label{def_X}
\end{equation}
then $\mathcal{G}(\calX)\subset \calX$. We also show that, equipped with the distance $W_{2,T}:(\mu,\nu)\mapsto \sup_{0\leq t\leq T}W_2(\mu_t,\nu_t)$, the space $\calX$ is complete. Finally, we prove the existence of a positive continuous function $g:\RR_+\times \RR_+\to\RR_+$ satisfying $g(T,K)\to 0$ as $T\to 0$ for every $K>0$ such that
\[
    W_{2,T}(\mathcal{G}(\mu^1),\mathcal{G}(\mu^2))\leq g(T,K_0)W_{2,T}(\mu^1,\mu^2)
\]
for every $\mu^1,\mu^2\in \calX$ and every $K_0$ such that $\|\bar{\mu}_0\|_{\psi_2}\leq K_0$. Thus, for every sub-Gaussian initial condition $\bar{\mu}_0$, there exists $T>0$ such that $\mathcal{G}$ is a contraction and hence has a unique fixed point. This shows the existence and uniqueness of the solution on $[0,T]$.

\paragraph{Energy identity and return to the original objective.}
Let $\mu$ be the fixed point constructed above and set $q_t:=\int_\Omega\Phi\,d\mu_t$. Since $\mu_t=X_t[\mu]_{\#}\bar\mu_0$, we have
\[
    q_t=\int_\Omega \Phi(X_t[\mu](u))\,d\bar\mu_0(u).
\]
The linear-growth bounds for $d\Phi$, the velocity field, and $X_t[\mu]$, together with the finite second moment of $\bar\mu_0$, justify differentiation under the integral. Consequently, for almost every $t\in[0,T]$,
\[
    \dot q_t=\int_\Omega d\Phi(u)v_t(u)\,d\mu_t(u),
    \qquad
    v_t(u)=-d\Phi(u)^*\tilde R'(q_t)=-\nabla\tilde F'(\mu_t)(u).
\]
The Hilbert-space chain rule therefore gives
\begin{equation}\label{eq_energy_dissipation_truncated}
    \frac{d}{dt}\tilde F(\mu_t)
    =\langle\tilde R'(q_t),\dot q_t\rangle_{\mathcal{F}}
    =-\int_\Omega|\nabla\tilde F'(\mu_t)(u)|^2\,d\mu_t(u).
\end{equation}
In particular, the velocity belongs to $L^2(\mu_t\,dt)$ on every compact time interval, so the curve is $2$-absolutely continuous in $W_2$ and satisfies the energy-dissipation identity. Moreover,
\[
    \tilde F(\mu_t)\leq\tilde F(\bar\mu_0)=F(\bar\mu_0)\leq\alpha.
\]
Thus, $F(\mu_t)\leq\alpha$, $\xi'(F(\mu_t))=1$, and the truncated and original vector fields coincide along the whole curve. Hence, the constructed curve is a Wasserstein gradient flow of $F$. For any other weak solution starting from $\bar\mu_0$, \Cref{lem_lagrangian_representation} first gives its characteristic representation. The same chain-rule calculation then keeps it in the sublevel set, so it solves the truncated equation and is a fixed point of $\mathcal{G}$. Contraction uniqueness therefore gives uniqueness for the original equation.

\paragraph{Construction of a global solution.} We can iterate the previous argument by applying the result again with the updated initial condition $X_T[\mu]_{\#}\bar{\mu}_0$, which satisfies ${F(X_T[\mu]_{\#}\bar{\mu}_0)\leq F(\bar{\mu}_0)\leq \alpha}$, so that $\alpha$ does not need to be updated. Since the sub-Gaussian norm of the initial condition increases at each step, however, the time interval that we can add at each step decreases. Suppose for contradiction that this procedure fails to guarantee existence beyond some time $T_{\lim}>0$. Reasoning as in \Cref{lem_bound_psi2}, we have $\|\mu_t\|_{\psi_2}\leq f(t,\|\bar{\mu}_0\|_{\psi_2})$ for every $t\in [0,T_{\lim})$. Let $T'\in(0,2T_{\lim})$ be such that $g(T', f(T_{\lim},\|\bar{\mu}_0\|_{\psi_2}))<1$. Since $f$ is non-decreasing in its first variable, we have $\|\mu_{T_{\lim}-T'/2}\|_{\psi_2}\leq f(T_{\lim}-T'/2,\|\bar{\mu}_0\|_{\psi_2})\leq f(T_{\lim},\|\bar{\mu}_0\|_{\psi_2})$, and we can therefore apply the previous argument to extend $\mu_{\rvert[0,T_{\lim}-T'/2]}$ to $[0,T_{\lim}+T'/2]$, which gives a contradiction.

\subsection{Technical lemmas}\label{subsec_technical_lemmas}
By \Cref{ass_wellp_wgf}, $R'$ is bounded on sublevel sets of $R$. Hence, there exists a constant $C>0$ depending only on $\alpha$ such that $\|\tilde{R}'(h)\|_{\mathcal{F}}\leq (3/2)\|R'(h)\|_{\mathcal{F}}\leq C$ for every $h\in\mathcal{F}$ such that $R(h)\leq 2\alpha$. Since $\tilde{R}'(h)=0$ if $R(h)>2\alpha$, the previous bound in fact holds for every $h\in\mathcal{F}$.
\begin{lem}\label{lem_cauchy_prop_velocity}
    Let $\alpha>0$, $T>0$ and $\mu\in C([0,T];\mathcal{P}_2(\Omega))$. Define ${v:[0,T]\times \Omega\to \RR^d}$ by
    \[
        v(t,u)=-\nabla \tilde{F}'(\mu_t)(u).
    \]
    Then $v$ is continuous, locally Lipschitz in space, and has linear growth in space, uniformly in time.
\end{lem}
\begin{proof}
    By \Cref{lemma_bound_diff_Phi_mu}, \Cref{ass_wellp_wgf} and the fact that $\xi'$ is continuous, we have that $t\mapsto \tilde{R}'(\int_{\Omega}\Phi d\mu_t)$ and $d\Phi$ are continuous. Using the continuity of $h_1,h_2\mapsto \langle h_1,h_2\rangle_{\mathcal{F}}$, we obtain the continuity of $v$.

    Moreover, $|v(t,u_1)-v(t,u_2)|\leq \|\tilde{R}'(\int_{\Omega}\Phi d\mu_t)\|_{\mathcal{F}} \|d\Phi(u_1)-d\Phi(u_2)\|$. Using \Cref{ass_wellp_wgf}, which guarantees that $d\Phi$ is locally Lipschitz, we obtain that $v$ is locally Lipschitz in space.

    Finally, using $\|\tilde{R}'(\int_{\Omega}\Phi d\mu_t)\|_{\mathcal{F}}\leq C$ and the fact that \Cref{ass_wellp_wgf} guarantees that $d\Phi$ has linear growth, we obtain the result.
\end{proof}

\begin{lem}\label{lem_bound_flow}
    If $T>0$ and $\bar{\mu}_0\in\mathcal{P}(\Omega)$ is sub-Gaussian, there exists a constant $C(T)>0$ depending on $\alpha$ and $T$ such that, for every $\mu\in\calX$, $t\in[0,T]$ and $u,u_1,u_2\in\Omega$,
    \begin{equation*}\left\{
        \begin{aligned}
            &|X_t[\mu](u)|\leq C(T)(1+|u|),\\
            &|X_t[\mu](u_1)-X_t[\mu](u_2)|\leq e^{C(T)(1+\max(|u_1|,|u_2|))t}|u_1-u_2|.
        \end{aligned}\right.
    \end{equation*}
\end{lem}
\begin{proof}
    Using \Cref{ass_wellp_wgf} and $X_t[\mu](u)=u+\int_{0}^t v[\mu_s](X_s[\mu](u))ds$, we obtain:
    \begin{equation*}
        \begin{aligned}
            |X_t[\mu](u)|&\textstyle\leq |u|+\int_{0}^t \|\tilde{R}'(\int_{\Omega}\Phi d\mu_s)\|_{\mathcal{F}} \|d\Phi(X_s[\mu](u))\|_{\mathcal{L}(\RR^d;\mathcal{F})}ds\\
            &\leq |u|+C'\int_{0}^t (1+|X_s[\mu](u)|)ds=|u|+C't + C'\int_{0}^t |X_s[\mu](u)|ds,
        \end{aligned}
    \end{equation*}
    for some constant $C'>0$.
    Applying Gr\"onwall's inequality yields the first inequality.

    For the second inequality, we use $|v_t(u_1)-v_t(u_2)|\leq C(1+\max(|u_1|,|u_2|))|u_1-u_2|$ together with the first inequality to obtain
    \begin{equation*}
        |X_t[\mu](u_1)-X_t[\mu](u_2)|\leq |u_1-u_2|+C(T)(1+\max(|u_1|,|u_2|))\int_{0}^t |X_s[\mu](u_1)-X_s[\mu](u_2)|ds.
    \end{equation*}
    Another application of Gr\"onwall's inequality yields the second inequality.
\end{proof}

Using \Cref{lem_subg_pushf,lem_bound_flow}, we immediately obtain the following result, which shows that $\mathcal{G}(\calX)\subset \calX$.
\begin{lem}\label{lem_bound_psi2}
    There exists a continuous function $f:(\RR_+)^2\to \RR_+$ that is non-decreasing in each variable and such that, for every $\mu\in\calX$, defining $\nu=\mathcal{G}(\mu)$, we have $\|\nu_t\|_{\psi_2}\leq f(t,\|\bar{\mu}_0\|_{\psi_2})$.
\end{lem}

\begin{lem}
 The space $(\calX,W_{2,T})$ with $\calX$ defined by \eqref{def_X} and $f$ as in \Cref{lem_bound_psi2} is a complete metric space.
\end{lem}
\begin{proof}
    Let $(\mu^n)_{n\geq 0}$ be a Cauchy sequence in $(\calX,W_{2,T})$. Since $\calX$ is a subset of the complete metric space $C([0,T];\mathcal{P}_2(\Omega))$, the sequence $(\mu^n)_{n\geq 0}$ converges for $W_{2,T}$ to some $\mu^*\in C([0,T];\mathcal{P}_2(\Omega))$. For every $c>0$, the mapping $u\mapsto \exp(|u|^2/c^2)$ is continuous and positive. Since $(\mu^n_t)_{n\geq 0}$ converges weakly to $\mu^*_t$ for every $t\in[0,T]$, the Portmanteau theorem gives
    \begin{equation*}
        \textstyle \int_{\Omega}\exp(|u|^2/c^2)d\mu^*_t(u)\leq \liminf_{n\to+\infty} \int_{\Omega}\exp(|u|^2/c^2)d\mu_t^n(u)
    \end{equation*}
    for every $c>0$. Taking $c=f(t,\|\bar{\mu}_0\|_{\psi_2})$ yields
    \[
        \int_{\Omega}\exp(|u|^2/c^2)d\mu_t^*(u)\leq 2
        \qquad\text{and hence}\qquad
        \|\mu_t^*\|_{\psi_2}\leq f(t,\|\bar{\mu}_0\|_{\psi_2}).
    \]
    This shows that $\mu^*\in \calX$ and concludes the proof.
\end{proof}

\begin{lem}\label{lem_lip_calF}
    There exists a positive continuous function $g:\RR_+\times \RR_+\to\RR_+$ satisfying $g(T,K)\to 0$ as $T\to 0$ for every $K>0$ such that, for every $\mu^1,\mu^2\in\calX$, we have
    \begin{equation}\label{eq_lip_calF}
    W_{2,T}(\mathcal{G}(\mu^1),\mathcal{G}(\mu^2))\leq g(T,K_0)W_{2,T}(\mu^1,\mu^2)
    \end{equation}
    for every $K_0$ such that $\|\bar{\mu}_0\|_{\psi_2}\leq K_0$.
\end{lem}
\begin{proof}
    Let $X_t^1$ and $X_t^2$ be the flow maps associated with $\mu^1$ and $\mu^2$ and $K_0$ such that $\|\bar{\mu}_0\|_{\psi_2}\leq K_0$. Using the fact that ${X_t^i(u)=u+\int_{0}^t v[\mu^i_s](X^i_s(u))ds}$, we obtain that
    \begin{equation*}
        \begin{aligned}
            |X_t^1(u)-X_t^2(u)|&\leq\int_{0}^t |v[\mu^1_s](X^1_s(u))-v[\mu^2_s](X^2_s(u))|ds\\
            &\leq \int_{0}^t [|v[\mu^1_s](X^1_s(u))-v[\mu^1_s](X^2_s(u))|+|v[\mu^1_s](X^2_s(u))-v[\mu^2_s](X^2_s(u))|]ds.
        \end{aligned}
    \end{equation*}
    To bound the first term, we use the boundedness of $\tilde{R}'$, the local Lipschitz bound for $d\Phi$ from \Cref{ass_wellp_wgf} and the bound from \Cref{lem_bound_flow} to obtain
    \begin{equation*}
        \int_{0}^t |v[\mu^1_s](X^1_s(u))-v[\mu^1_s](X^2_s(u))|ds \leq C(T)(1+|u|)\int_{0}^t |X^1_s(u)-X^2_s(u)|ds.
    \end{equation*}
    To bound the second term, we first notice that, for every $h_1,h_2\in\mathcal{F}$, we have:
    \begin{equation*}
        \|\tilde{R}'(h_1)-\tilde{R}'(h_2)\|_{\mathcal{F}}\leq \|\xi'\|_{\infty}\|R'(h_1)-R'(h_2)\|_{\mathcal{F}} + \|R'(h_2)\|_{\mathcal{F}}|\xi'(R(h_1))-\xi'(R(h_2))|.
    \end{equation*}
    Moreover,
    \begin{equation}\label{bound_predict}
        \textstyle\|\int_{\Omega}\Phi d\mu^i_s\|_{\mathcal{F}}\lesssim 1+m_2(\mu^i_s)\lesssim 1+\|\mu^i_s\|_{\psi_2}^2\leq 1+f(T,K_0)^2,
    \end{equation}
    so that $\{\int_{\Omega}\Phi d\mu^i_s,~(i,s)\in\{1,2\}\times[0,T]\}$ is included in some bounded set $B_{T,K_0}$. Since $R'$ is Lipschitz and hence bounded on $B_{T,K_0}$, using the fact that $\mathrm{Lip}(\xi')\leq 4/\alpha$ and \Cref{lemma_bound_diff_Phi_mu}, we obtain
    \begin{equation*}
        \begin{aligned}
        \textstyle\|\tilde{R}'(\int_{\Omega}\Phi d\mu^1_s)-\tilde{R}'(\int_{\Omega}\Phi d\mu^2_s)\|_{\mathcal{F}}&\textstyle\leq C(T,K_0)\|\int_{\Omega}\Phi d\mu^1_s-\int_{\Omega}\Phi d\mu^2_s\|_{\mathcal{F}}\\&\leq C(T,K_0) W_2(\mu^1_s,\mu^2_s).
        \end{aligned}
    \end{equation*}
    Finally, using $\|d\Phi(X_s^2(u))\|\lesssim 1+|X_s^2(u)|\leq C(T)(1+|u|)$, where the last inequality follows from \Cref{lem_bound_flow}, we obtain
    \begin{equation*}
        \int_{0}^t |v[\mu^1_s](X^2_s(u))-v[\mu^2_s](X^2_s(u))|ds\leq C(T,K_0)(1+|u|)\int_0^t W_2(\mu^1_s,\mu^2_s)ds.
    \end{equation*}
    Combining the two bounds and applying Gr\"onwall's inequality, we obtain
    \begin{equation*}
        |X^1_t(u)-X^2_t(u)|\leq C(T,K_0)(1+|u|)\int_{0}^t \exp(C(T,K_0)(1+|u|)(t-s))W_2(\mu_s^1,\mu_s^2)ds.
    \end{equation*}
    Using the transport plan $(X_t^1\times X_t^2)_{\#}\bar{\mu}_0$, we obtain
    \begin{equation*}
        \begin{aligned}
            W_2^2((X^1_t)_{\#}\bar{\mu}_0,(X^2_t)_{\#}\bar{\mu}_0)&\leq \int_{\Omega}|X^1_t(u)-X^2_t(u)|^2d\bar{\mu}_0(u)\\
            &\leq \int_{\Omega}C(T,K_0)^2 (1+|u|)^2 \bigg[\int_{0}^te^{C(T,K_0)(1+|u|)(t-s)}W_2(\mu^1_s,\mu^2_s)ds\bigg]^2 d\bar{\mu}_0(u)\\
            & \leq C(T,K_0)^2 W_{2,T}^2(\mu^1,\mu^2)\int_{\Omega}(1+|u|)^2 \bigg[\int_{0}^t e^{C(T,K_0)(1+|u|)(t-s)}ds\bigg]^2 d\bar{\mu}_0(u)\\
            & =C(T,K_0)^2 W_{2,T}^2(\mu^1,\mu^2)\int_{\Omega}(1+|u|)^2 \bigg[\frac{e^{C(T,K_0)(1+|u|)t}-1}{C(T,K_0)(1+|u|)}\bigg]^2 d\bar{\mu}_0(u)\\
            &=W_{2,T}^2(\mu^1,\mu^2)\int_{\Omega}\Big[e^{C(T,K_0)(1+|u|)t}-1\Big]^2d\bar{\mu}_0(u).
        \end{aligned}
    \end{equation*}
    Inserting the exponential weight and using $\int_{\Omega} \exp(|u|^2/K_0^2)d\bar{\mu}_0(u)\leq 2$, we obtain \eqref{eq_lip_calF} with
    \begin{equation*}
        g(T,K)= \sqrt{2~\underset{u\in\Omega}{\mathrm{sup}}~(e^{C(T,K)(1+|u|)T}-1)^2 e^{-|u|^2/K^2}},
    \end{equation*}
    which satisfies $g(T,K)\to 0$ as $T\to 0$ for every $K>0$.
\end{proof}

We end this section with a stability property for the flow with respect to its initialization.
\begin{lem}\label{lem_wgf_stability}
   Let $\mu_0\in\mathcal{P}_2(\Omega)$ be sub-Gaussian and $T>0$. Let $(\mu^n_0)_{n\geq 0}\in (\mathcal{P}_2(\Omega))^{\NN}$ be a sequence of sub-Gaussian measures with $\lim_{n\to+\infty}W_2(\mu^n_0,\mu_0)=0$. Then, denoting by $(\mu_t)_{t\geq 0}$ and $(\mu^n_t)_{t\geq 0}$ the Wasserstein gradient flows of $F$ initialized with $\mu_0$ and $\mu^n_0$ for every $n\in\NN$, it holds $\lim_{n\to+\infty}\sup_{0\leq t\leq T}W_2(\mu_t^n,\mu_t)=0$.
\end{lem}
\begin{proof}
    Let $(X_t)_{t\geq 0}$ and $(X^n_t)_{t\geq 0}$ denote the flow maps associated with $(\mu_t)_{t\geq 0}$ and $(\mu^n_t)_{t\geq 0}$ for every $n\in\NN$. We have
    \begin{equation*}
            W_2(\mu^n_t,\mu_t)=W_2((X^n_t)_{\#}\mu^n_0,(X_t)_{\#}\mu_0)\leq W_2((X^n_t)_{\#}\mu_0,(X_t)_{\#}\mu_0)+W_2((X^n_t)_{\#}\mu^n_0,(X^n_t)_{\#}\mu_0).
    \end{equation*}
    
    By \Cref{lemma_bound_diff_Phi_mu} and the continuity of $R$, the functional $F$ is continuous with respect to $W_2$. Hence, we have that $(F(\mu_0^n))_{n\in\NN}$ is bounded. By energy dissipation, we obtain an upper bound for $(F(\mu_t))_{t\geq 0}$ and $(F(\mu_t^n))_{t\geq 0,n\in\NN}$. As $R'$ is bounded on sublevel sets of $R$, this turns into a bound for $(\|R'(\int_{\Omega}\Phi d\mu_t)\|_{\calF})_{t\geq 0}$ and $(\|R'(\int_{\Omega}\Phi d\mu^n_t)\|_{\calF})_{t\geq 0,n\in\NN}$. Arguing as in \Cref{lem_cauchy_prop_velocity,lem_bound_flow}, we obtain the linear growth of the velocity fields and flow maps on every bounded time interval $[0,T]$ with constants depending only on $T$. This yields a bound on $(m_2(\mu_t))_{0\leq t\leq T}$ and $(m_2(\mu^n_t))_{0\leq t\leq T,n\in\NN}$ depending only on $T$, $m_2(\mu_0)$ and $(m_2(\mu^n_0))_{n\in \NN}$. By $W_2$ convergence, the initial second moments are uniformly bounded, and we can bound the first term by arguing as in the proof of \Cref{lem_lip_calF} to obtain
    \begin{equation*}
        \begin{aligned}
            W_2((X^n_t)_{\#}\mu_0,(X_t)_{\#}\mu_0)&\leq \sqrt{\int_{\Omega}\bigg[\int_{0}^t C(T) (1+|u|) e^{C(T)(1+|u|)(t-s)}W_2(\mu^n_s,\mu_s)ds\bigg]^2 d\mu_0(u)}\\
            &\leq \int_{0}^t W_2(\mu^n_s,\mu_s)\sqrt{\int_{\Omega}C(T)^2 (1+|u|)^2 e^{2C(T)(1+|u|)(t-s)} d\mu_0(u)}ds\\
            &\leq h_T\int_{0}^t W_2(\mu^n_s,\mu_s)ds,
        \end{aligned}
    \end{equation*}
    where $h_{T}^2=2\sup_{u\in\Omega} C(T)^2(1+|u|)^2 e^{2C(T)(1+|u|)T}e^{-|u|^2/\max(1,\|\mu_0\|_{\psi_2}^2)}$.

    For the second term, let $\Pi^n_0$ be an optimal transport plan for the quadratic cost between $\mu_0$ and $\mu^n_0$. For every  $R>0$, it holds
    \begin{equation*}
        \begin{aligned}
            \int_{\max(|u_1|,|u_2|)\leq R}|X^n_t(u_1)-X^n_t(u_2)|^2d\Pi^n_0(u_1,u_2)&\leq C(T, R)\int_{\Omega\times\Omega}|u_1-u_2|^2d\Pi^n_0(u_1,u_2)\\&=C(T,R)W_2^2(\mu^n_0,\mu_0).
        \end{aligned}
    \end{equation*}
    On the other hand, provided $R\geq 1$, it holds 
    \begin{equation*}\left\{
        \begin{aligned}
            &(1+|u_1|^2+|u_2|^2)\mathbf{1}_{\{\max(|u_1|,|u_2|)>R\}}\leq 3(|u_1|^2\mathbf{1}_{\{|u_1|>R\}}+|u_2|^2\mathbf{1}_{\{|u_2|>R\}}),\\
            &|u_2|^2\mathbf{1}_{|u_2|>R}\leq 4|u_1-u_2|^2+2|u_1|^2\mathbf{1}_{|u_1|>R/2},
        \end{aligned}\right.
    \end{equation*}
    which yields
    \begin{equation*}
        \begin{aligned}
            \int_{\max(|u_1|,|u_2|)>R}|X^n_t(u_1)-X^n_t(u_2)|^2d\Pi^n_0(u_1,u_2)&\leq C(T)\int_{\max(|u_1|,|u_2|)>R}(1+|u_1|^2+|u_2|^2)d\Pi^n_0(u_1,u_2)\\
            &\leq 3C(T)\bigg[\int_{\{|u_2|>R\}}|u_2|^2d\mu^n_0(u_2)+\int_{\{|u_1|>R\}}|u_1|^2d\mu_0(u_1)\bigg]\\
            &\leq 3C(T)\bigg[4 W_2^2(\mu^n_0,\mu_0)+3\int_{\{|u|>R/2\}}|u|^2d\mu_0(u)\bigg].
        \end{aligned}
    \end{equation*}
    As a result, using the coupling $(X^n_t\times X^n_t)_{\#}\Pi^n_0$, we obtain
    \begin{equation*}
      W_2^2((X^n_t)_{\#}\mu^n_0,(X^n_t)_{\#}\mu_0)\leq C'(T,R)W_2^2(\mu^n_0,\mu_0)+C'(T)\int_{\{|u|>R/2\}}|u|^2d\mu_0(u).  
    \end{equation*}
    Combining the two terms and using Gr\"onwall's lemma, we obtain
    \begin{equation*}
       \limsup_{n\to+\infty} \sup_{0\leq t\leq T} W_2(\mu^n_t,\mu_t)\leq e^{h_T T}\sqrt{C'(T)\int_{\{|u|>R/2\}}|u|^2d\mu_0(u)}.
    \end{equation*}

    As this holds for every $R\geq 1$ and $\int_{\{|u|>R/2\}}|u|^2 d\mu_0(u)$ goes to $0$ as $R\to+\infty$ (since $\mu_0$ has a finite second moment), we obtain the result.
\end{proof}

\section{Auxiliary results}\label{appendix_examples}
\subsection{Regularity of the lifted feature map}
We first verify that the product parametrization used throughout \Cref{sec_glob_conv} satisfies \Cref{ass_wellp_wgf} under a bounded-Lipschitz regularity assumption on $\phi$.
\begin{lem}\label{lem_ass_ass}
    Assume that $\Omega=\RR^{d_w}\times\RR^{d_{\theta}}$ and that $\Phi:\Omega\to\mathcal{F}$ has the form $\Phi(w,\theta)=\phi(\theta)w$ where $\phi(\theta)\in \mathcal{L}(\RR^{d_w};\mathcal{F})$. If $\phi$ is differentiable with a bounded Lipschitz differential, then the part of \Cref{ass_wellp_wgf} concerning $\Phi$ holds.
\end{lem}
\begin{proof}
    We have
    \[
        d\Phi(w,\theta)\cdot (w',\theta')=\phi(\theta)w'+[d\phi(\theta)\cdot\theta']w.
    \]
    As a result, $\|d\Phi(w,\theta)\|\leq (\|\phi(\theta)\|^2+|w|^2\|d\phi(\theta)\|^2)^{1/2}$. Since $d\phi$ is bounded and Lipschitz, $\|d\Phi(u)\|\lesssim 1+|u|$. Defining $\Delta=[d\Phi(w_1,\theta_1)-d\Phi(w_2,\theta_2)]\cdot (w',\theta')$, we also have
    \[
        \Delta=[\phi(\theta_1)-\phi(\theta_2)]w'
        +[(d\phi(\theta_1)-d\phi(\theta_2))\cdot\theta']w_1
        +[d\phi(\theta_2)\cdot\theta'](w_1-w_2).
    \]
    As a consequence, we obtain
    \begin{equation*}
        \|d\Phi(w_1,\theta_1)-d\Phi(w_2,\theta_2)\|\leq \sqrt{\|d\phi\|_{\infty}^2 |\theta_1-\theta_2|^2+(\mathrm{Lip}(d\phi)|\theta_1-\theta_2||w_1|+\|d\phi\|_{\infty}|w_1-w_2|)^2}
    \end{equation*}
    and a direct computation concludes the proof.
\end{proof}

\subsection{Standard loss functions}
We now prove \Cref{lem_ass_loss}, which states that the square loss and the cross-entropy loss are both compatible with \Cref{ass_wellp_wgf,ass_glob_conv}.
\begin{proof}[Proof of \Cref{lem_ass_loss}]\label{proof_lem_ass_loss}
    We argue as in the proof of \cite[Lemma D.1]{chizat2018global}. First, a direct computation shows that, in both cases, the integral in the expression of $R(f)$ is finite for every $f\in\mathcal{F}$. The convexity of $R$ directly follows from the fact that, for every $y\in\RR^{\dout}$, the mapping $z\mapsto \ell(z,y)$ is convex.

    In both cases, $z\mapsto \ell(z,y)$ is differentiable for every $y$ and there exists $L>0$ such that, for every $(z_1,z_2,y)\in(\RR^{\dout})^3$, ${|\nabla_z \ell(z_1,y)-\nabla_z\ell(z_2,y)|\leq L|z_1-z_2|}$. To see this, notice that ${\nabla_z \ell(z,y)=z-y}$ for the square loss, so that we can take $L=1$. For the cross-entropy loss with $\dout>1$, ${\nabla_z \ell(z,y)=-y+\sigma(z)}$, where $\sigma$ denotes the softmax function. For the binary case $\dout=1$, the derivative is $\nabla_z\ell(z,y)=\sigma_{\log}(z)-y$, where $\sigma_{\log}(z)=e^z/(1+e^z)$. In both cross-entropy cases, one can take $L=2$.

    We now prove that the differential of $R$ is given by
    \[
        dR(f)\cdot h=\int_{\RR^{\din}\times\RR^{\dout}}\langle \nabla_z\ell(f(x),y),h(x)\rangle d\rho(x,y).
    \]
    We have
    \begin{equation*}
        \begin{aligned}
            \Delta(h)&\eqdef|R(f+h)-R(f)-dR(f)\cdot h|\\
            &=\textstyle|\int_{\RR^{\din}\times\RR^{\dout}}[\ell(f(x)+h(x),y)-\ell(f(x),y)-\langle \nabla_z \ell(f(x),y),h(x)\rangle]d\rho(x,y)|.
        \end{aligned}
    \end{equation*}
    Using that $|\ell(z+z',y)-\ell(z,y)- \nabla_z\ell(z,y)\cdot z'|\leq L |z'|^2/2$, we obtain $\Delta(h)\leq (L/2)\|h\|^2_{\mathcal{F}}$. A similar computation also shows that $dR$ is $L$-Lipschitz.

    Finally, we claim that $|\nabla_z\ell(z,y)|^2\leq 2\ell(z,y)$. This is immediate for the square loss. For the multiclass cross-entropy loss, since $\rho_y$ is supported on $\mathcal{B}_k$ with $k=\dout$, it is enough to show that property when $y\in\mathcal{B}_k$, say $y=e_q$. We use $|a-1|^2\leq -\log(a)$ if $a\in(0,1)$ to obtain
    \begin{equation*}
        \begin{aligned}
            |\sigma(z)-y|^2&=\Bigg|\frac{e^{z_q}}{\sum_{j=1}^k e^{z_j}}-1\Bigg|^2 + \sum\limits_{i\neq q}\Bigg|\frac{e^{z_i}}{\sum_{j=1}^k e^{z_j}}\Bigg|^2\\
            &\leq \Bigg|\frac{e^{z_q}}{\sum_{j=1}^k e^{z_j}}-1\Bigg|^2 + \Bigg|\sum\limits_{i\neq q}\frac{e^{z_i}}{\sum_{j=1}^k e^{z_j}}\Bigg|^2\\
            &= \Bigg|\frac{e^{z_q}}{\sum_{j=1}^k e^{z_j}}-1\Bigg|^2 + \Bigg|\frac{e^{z_q}}{\sum_{j=1}^k e^{z_j}}-1\Bigg|^2\\
            &\leq -2\log(\sigma_q(z))=2\ell(z,y).
        \end{aligned}
    \end{equation*}
    In the binary cross-entropy case, writing $\sigma_{\log}(z)=e^z/(1+e^z)$, the same conclusion follows from $\sigma_{\log}(z)^2\leq -\log(1-\sigma_{\log}(z))=\ell(z,0)$ and $(1-\sigma_{\log}(z))^2\leq -\log(\sigma_{\log}(z))=\ell(z,1)$.
    To conclude, we notice that
    \begin{equation*}
        \begin{aligned}
            |dR(f)\cdot h|&\textstyle\leq \int_{\RR^{\din}\times\RR^{\dout}}|\nabla_z\ell(f(x),y)||h(x)| d\rho(x,y)\\
            &\textstyle \leq \|h\|_{\mathcal{F}} \sqrt{\int_{\RR^{\din}\times\RR^{\dout}}|\nabla_z\ell(f(x),y)|^2 d\rho(x,y)}\\
            &\leq \|h\|_{\mathcal{F}}\sqrt{2R(f)},
        \end{aligned}
    \end{equation*}
    so that $dR$ is bounded on sublevel sets of $R$.
\end{proof}

\subsection{Lagrangian representation of continuity equations}
The next lemma connects the Eulerian continuity equation with its characteristic flow. It is used in the proofs of both \Cref{thm_wgf,thm_glob_conv}: for the former, it reduces uniqueness to uniqueness of the characteristic fixed point; for the latter, it propagates full support along the flow.
\begin{lem}[Lagrangian representation]\label{lem_lagrangian_representation}
Let $T>0$ and $\mu\in AC^2([0,T];\mathcal{P}_2(\Omega))$ be a solution of the continuity equation $\partial_t\mu_t+\operatorname{div}(v_t\mu_t)=0$, where $v$ is continuous, locally Lipschitz in space uniformly on compact subsets of $[0,T]\times\Omega$, has linear growth in space uniformly in time, and belongs to $L^2(\mu_tdt)$. If $X_{s,t}$ is the two-parameter flow of $v$, then
\[
    \mu_t=(X_{s,t})_{\#}\mu_s\qquad\text{for every }s,t\in[0,T].
\]
Moreover, $X_{s,t}$ is a homeomorphism of $\Omega$ and $X_{s,t}^{-1}=X_{t,s}$.
\end{lem}
\begin{proof}
The superposition principle for continuity equations (see \cite[Theorem 8.2.1]{ambrosioGradientFlowsMetric2009}) represents $\mu$ by a probability measure on absolutely continuous paths solving $\dot\gamma_t=v_t(\gamma_t)$, with time marginal $\mu_t$. The spatial local-Lipschitz bound and linear growth give existence and uniqueness of each characteristic on $[0,T]$, so every such path satisfies $\gamma_t=X_{s,t}(\gamma_s)$. Taking marginals gives the pushforward identity. Continuous dependence on the initial condition makes $X_{s,t}$ continuous, and solving the ODE backward in time gives the continuous inverse $X_{t,s}$.
\end{proof}

\subsection{Fully connected networks}\label{appendix_fcn}
\subsubsection{Bounded activations}
The two properties below are stated for the sigmoid function (without proof) in \cite[Appendix D.3]{chizat2018global}. We provide a proof of both results for the sake of completeness and extend them to a larger class of bounded activation functions. We then proceed with the proof of \Cref{prop_glob_conv_sigmoid}.
\begin{lem}\label{lem_aux_bounded_scalar}
    Assume that $\rho\in\mathcal{P}(\RR^d)$ is absolutely continuous with respect to the Lebesgue measure. If $\sigma$ is continuous and $\sigma(r)\to \alpha$, $\sigma(-r)\to \beta$ for some $\alpha,\beta\in\RR$, then for every $f\in L^{1}(\rho)$ the mapping
    \[
        (\theta,b)\mapsto \int_{\RR^d} f(x)\sigma(r[\langle \theta,x\rangle+b])d\rho(x)
    \]
    converges uniformly on $\mathbb{S}^{d}$, as $r\to+\infty$, to the mapping
    \[
        (\theta,b)\mapsto \int_{\RR^d}f(x) \sigma_{\infty}(\langle\theta,x\rangle+b)d\rho(x),
    \]
    where $\sigma_{\infty}(s)=\alpha\mathbf{1}_{(0,+\infty)}(s)+\beta\mathbf{1}_{(-\infty,0)}(s)$ for $s\neq0$ and $\sigma_\infty(0)=0$.
\end{lem}
\begin{proof}
    Our assumptions on $\sigma$ ensure that it is bounded and $\epsilon_{\eta}(r):=\sup_{|s|\geq \eta}|\sigma(rs)-\sigma_{\infty}(s)|\to 0$ when $r\to+\infty$ for every $\eta>0$. We fix $\eta>0$ and $R\geq 1$ and split $\RR^{d}$ into two sets $E^1,E^2$ defined by
    \begin{equation*}
        E^1=\{x\in\RR^d\,\rvert\, |\langle \theta,x\rangle+b|\geq \eta\}~\mathrm{and}~ E^2=\RR^{d}\setminus E^1.
    \end{equation*}
    If $x\in E^1$ then $|\sigma(r [\langle\theta, x\rangle+b])-\sigma_{\infty}(\langle\theta, x\rangle + b)|\leq \epsilon_{\eta}(r)$. Since $\|\sigma-\sigma_{\infty}\|\leq \|\sigma\|_{\infty}+\max(|\alpha|,|\beta|):=c$, we obtain
    \begin{equation*}
        \begin{aligned}
            \Delta_r(\theta,b)&:=\textstyle\int_{\RR^{d}}|f(x)||\sigma(r[\langle\theta, x\rangle+b])-\sigma_{\infty}(\langle\theta, x\rangle+b)| d\rho(x)\\
            &\textstyle\leq  \epsilon_{\eta}(r)\int_{\RR^d}|f(x)|d\rho(x)+c\int_{E^2}|f(x)|d\rho(x)\\
            &\textstyle\leq  \epsilon_{\eta}(r)\int_{\RR^d}|f(x)|d\rho(x)+c\int_{E^2\cap B_d(0,R)}|f(x)|d\rho(x)+c\int_{\RR^d\setminus B_d(0,R)}|f(x)|d\rho(x),
        \end{aligned}
    \end{equation*}
    which yields
    \begin{equation*}
        \textstyle\underset{r\to+\infty}{\mathrm{lim~sup}}\underset{(\theta,b)\in\mathbb{S}^{d}}{\sup}\Delta_r(\theta,b) \leq c\int_{\RR^d\setminus B_d(0,R)}|f(x)|d\rho(x)+c\sup_{(\theta,b)\in\SS^d}\int_{E^2\cap B_d(0,R)}|f(x)|d\rho(x).
    \end{equation*}

    If $|\theta|\leq 1/(2R)$, then since $|\theta|^2+b^2=1$, the set $E^2\cap B_d(0,R)$ is empty as soon as $\eta<(\sqrt{3}-1)/2$. If $|\theta|\geq 1/(2R)$, then $E^2\cap B_d(0,R)$ is contained up to a rotation and translation in $[-a,a]\times B_{d-1}(0,R)$ with $a=4R\eta$. As a consequence, we obtain
    \begin{equation*}
        \textstyle\int_{E^2\cap B_d(0,R)}|f(x)|d\rho(x)\leq \sup\,\{\int_{A}|f(x)|d\rho(x)\,\rvert\, A\subset \RR^{d},~|A|\leq 8R\eta \mathcal{L}^{d-1}(B_{d-1}(0,R))\}.
    \end{equation*}
    The quantity on the right-hand side converges to zero as $\eta\to 0$ since $\rho$ is absolutely continuous with respect to the Lebesgue measure and $f\in L^1(\rho)$. As $\int_{\RR^d\setminus B_d(0,R)}|f(x)|d\rho(x)$ goes to $0$ as $R\to+\infty$ by $f\in L^1(\rho)$, letting first $\eta\to 0$ and then $R\to +\infty$, we obtain the result.
\end{proof}

In \cite[Appendix D.3]{chizat2018global}, the result below is stated for sigmoid activations under a finite moment assumption on $\rho$. However, the integrability of $x\mapsto f(x)\rho(x)(x,1)$ with respect to the surface measure on the hyperplane $\{\langle \theta,x\rangle+b=0\}$ cannot be guaranteed without additional assumptions on $\rho$. For this reason, we introduce a polynomial decay assumption under which the result holds.
\begin{lem}\label{lem_aux_grad_bounded_scalar}
    Assume $\rho\in\mathcal{P}(\RR^d)$ has a continuous density with respect to the Lebesgue measure with $\rho(x)\leq C(1+|x|)^{-p}$ for some $C>0$ and $p>d+1$. If $\sigma\in C^1(\RR)$ with $\sigma'$ integrable and $s\sigma'(s)\to 0$ as $|s|\to +\infty$, then for every bounded and continuous function $f:\RR^{d}\to\RR$, the mapping $g_f:(\theta,b)\mapsto \int_{\RR^{d}}f(x)\sigma(\langle \theta,x\rangle+b)\rho(x)dx$ is of class $C^1$ and $\nabla g_f(\theta,b)=\int_{\RR^{d}}\sigma'(\langle \theta,x\rangle+b)f(x)\rho(x)(x,1)dx$ for every $(\theta,b)\in\RR^{d+1}$. Moreover, the mapping $(\theta,b)\mapsto r\nabla g_f(r(\theta,b))$ converges uniformly on $\mathbb{S}^{d}$ to
    \begin{equation*}
    G_f^{\infty}:(\theta,b)\mapsto\left\{\begin{aligned}
        &\frac{\int_{-\infty}^{+\infty}\sigma'(s)ds}{|\theta|}\int_{\{\langle \theta,x\rangle+b=0\}}f(x)\rho(x)(x,1) d\mathcal{H}^{d-1}(x)&\mathrm{if}~\theta\neq 0,\\
        &0&\mathrm{if}~\theta=0.
    \end{aligned}\right.
    \end{equation*}
    Finally, for every $q\geq 1$ such that $p>d+q$, the quantity
    \[
        \textstyle\sup\{|(\theta,b)|\int_{\RR^d}|\sigma'(\langle\theta,x\rangle+b)||f(x)|(1+|x|)^q\rho(x)dx\,\rvert\,(\theta,b)\in\RR^{d+1},~f\in C^0(\RR^d),~\|f\|_{\infty}\leq \epsilon\}
    \]
    converges to $0$ as $\epsilon\to0$.
\end{lem}
\begin{proof}
First, our assumptions on $\sigma$ ensure that $\sigma'$ is bounded and $\epsilon_{\eta}(r):=\sup_{|s|\geq \eta}r|\sigma'(rs)|\to 0$ as $r\to+\infty$ for every $\eta>0$. Throughout, we use the notation $\psi(x):=f(x)\rho(x)(x, 1)$. For every $x\in\RR^d$, the gradient of $(\theta,b)\mapsto f(x)\sigma(\langle\theta,x\rangle+b)\rho(x)$ is $\sigma'(\langle \theta,x\rangle+b)\psi(x)$, whose norm is bounded by $\|\sigma'\|_{\infty}\|f\|_{\infty}\rho(x)(1+|x|)\lesssim (1+|x|)^{-(p-1)}$. Since $p>d+1$, this last function is integrable, which shows that $g_f$ is of class $C^1$ and that for every $(\theta,b)\in\RR^{d+1}$,
\begin{equation*}
    \nabla g_f(\theta,b)=\int_{\RR^{d}}\sigma'(\langle \theta,x\rangle+b)\psi(x)dx.
\end{equation*}
Let $(\theta,b)\in\mathbb{S}^{d}$ and $r>0$. If $\theta=0$ then $r\nabla g_f(r(\theta,b))=r\sigma'(rb)\int_{\RR^{d}}\psi(x)dx$. If $\theta\neq 0$, we define $u_{\theta}:=\theta/|\theta|$ and $t_{\theta}:=-b/|\theta|$. The co-area formula \cite[Theorem 3.11]{evansMeasureTheoryFine2015} yields that the restriction of $\psi$ to $\{\langle u_\theta,x\rangle=s\}$ is $\mathcal{H}^{d-1}$-integrable for almost every $s\in\RR$ and
\begin{equation*}
    \begin{aligned}
        r\nabla g_f(r(\theta,b))&=\frac{1}{|\theta|}\int_{-\infty}^{+\infty} r\sigma'(rs) H_{u_\theta}(t_{\theta}+s/|\theta|)ds~~~\mathrm{with}~~~ H_{u}(s):=\int_{\{\langle u,x\rangle=s\}}\psi(x)d\mathcal{H}^{d-1}(x).
    \end{aligned}
\end{equation*}

If $\theta_0\in\RR^d\setminus\{0\}$, we can find a continuous map $\theta\mapsto R_{\theta}$ in a neighborhood $\Theta_0\subset\RR^d\setminus\{0\}$ of $\theta_0$ such that $R_\theta$ is an orthogonal matrix with $R_{\theta}e_1=u_{\theta}$ for every $\theta\in\Theta_0$. On $\Theta_0$,
\begin{equation*}
    H_{u_{\theta}}(s)=\int_{\{z_1=0\}}\psi(R_{\theta}(se_1+z))d\mathcal{H}^{d-1}(z).
\end{equation*}
Since
\begin{equation*}
    \begin{aligned}
        |\psi(R_{\theta}(s e_1+z))|&\leq C \|f\|_{\infty}\Big(1+\sqrt{s^2+|z|^2}\Big)^{-(p-1)}\leq C\|f\|_{\infty} (1+|z|)^{-(p-1)}
    \end{aligned}
\end{equation*}
and $p>d$, the dominated convergence theorem shows that $(\theta,b,s)\mapsto H_{u_{\theta}}(t_{\theta}+s/|\theta|)$ is continuous on the set $\{(\theta,b,s)\in\mathbb{S}^d\times \RR\,\rvert\,\theta\neq 0\}$, and hence uniformly continuous on $\{(\theta,b,s)\in\mathbb{S}^d\times [-1,1]\,\rvert\,|\theta|\geq \delta\}$ for every $\delta\in(0,1)$. We therefore obtain that the family $\{H_{u_{\theta}}(t_{\theta}+\cdot/|\theta|)\,\rvert\,(\theta,b)\in\mathbb{S}^d,~|\theta|\geq \delta\}$ is uniformly bounded and equicontinuous at zero. Using $\int_{-\infty}^{+\infty}r\sigma'(rs)ds=\int_{-\infty}^{+\infty}\sigma'(s)ds$ together with
\begin{equation*}
\sup_{r\geq 0}\int_{-\infty}^{+\infty}r|\sigma'(rs)|ds=\int_{-\infty}^{+\infty}|\sigma'(s)|ds<+\infty~~\mathrm{and}~~
\int_{|s|\geq \eta}r|\sigma'(rs)|ds=\int_{|s|\geq r\eta}|\sigma'(s)|ds\underset{r\to+\infty}{\longrightarrow} 0
\end{equation*}
for every $\eta>0$, we can then obtain the desired uniform convergence result on $\{(\theta,b)\in\mathbb{S}^d\,\rvert\,|\theta|\geq \delta\}$ for every $\delta\in(0,1)$.

For every $\theta\neq 0$,
\begin{equation}\label{eq:hyperplane_integral}
    \begin{aligned}
        |H_{u_{\theta}}(s)|&\leq C\|f\|_{\infty}\int_{\{\langle u_{\theta},z\rangle=0\}}\Big[1+\sqrt{|s|^2+|z|^2}\Big]^{-(p-1)}d\mathcal{H}^{d-1}(z)\\
        &\leq C\|f\|_{\infty}\int_{\{\langle u_{\theta},z\rangle=0\}}\big[2^{-1/2}(1+|s|+|z|)\big]^{-(p-1)}d\mathcal{H}^{d-1}(z)\\
        &=C'(1+|s|)^{-(p-d)},
    \end{aligned}
\end{equation}
where $C':=2^{(p-1)/2}C\|f\|_{\infty}\int_{\RR^{d-1}}(1+|z|)^{-(p-1)}d\mathcal{H}^{d-1}(z)$.

Let us fix $\delta\leq 1/2$ and $0<|\theta|\leq \delta$. We have $|b|\geq 1/2$. We define $E_1=\{s\in\RR\,\rvert\,|s-t_{\theta}|\leq |t_{\theta}|/2\}$ and $E_2=\RR\setminus E_1$. It holds
\begin{equation*}
    \begin{aligned}
        |r\nabla g_f(r(\theta,b))|&\leq\frac{1}{|\theta|}\int_{-\infty}^{+\infty}|r\sigma'(rs)H_{u_{\theta}}(t_{\theta}+s/|\theta|)|ds\\
        &= \int_{-\infty}^{+\infty}|r\sigma'(r|\theta|(s-t_{\theta}))H_{u_{\theta}}(s)|ds\\
        &=\int_{E_1}|r\sigma'(r|\theta|(s-t_{\theta}))H_{u_{\theta}}(s)|ds+\int_{E_2}|r\sigma'(r|\theta|(s-t_{\theta}))H_{u_{\theta}}(s)|ds.
    \end{aligned}
\end{equation*}
Since $|s|\geq |t_{\theta}|/2$ on $E_1$, we get
\begin{equation}\label{eq:hyperplane_2}
    \begin{aligned}
        \int_{E_1}r|\sigma'(r|\theta|(s-t_{\theta}))H_{u_{\theta}}(s)|ds&\leq C'\int_{E_1}r|\sigma'(r|\theta|(s-t_{\theta}))|(1+|s|)^{-(p-d)}ds\\
        &\leq \frac{C'}{(1+|t_{\theta}|/2)^{p-d}}\int_{-\infty}^{+\infty}r|\sigma'(r|\theta|(s-t_{\theta}))|ds\\
        &=\frac{C'}{|\theta|(1+|t_{\theta}|/2)^{p-d}}\int_{-\infty}^{+\infty}|\sigma'(s)|ds\\
        &\leq \frac{C'}{|\theta|(1+1/(4|\theta|))^{p-d}}\int_{-\infty}^{+\infty}|\sigma'(s)|ds\\
        &\leq 4^{p-d}C'|\theta|^{p-(d+1)}\int_{-\infty}^{+\infty}|\sigma'(s)|ds.
    \end{aligned}
\end{equation}
We also obtain
\begin{equation*}
    \begin{aligned}
        \int_{E_2}r|\sigma'(r|\theta|(s-t_{\theta}))H_{u_{\theta}}(s)|ds&\leq \epsilon_{1/4}(r)\int_{-\infty}^{+\infty}|H_{u_{\theta}}(s)|ds\\
        &\leq C'' \epsilon_{1/4}(r)
    \end{aligned}
\end{equation*}
where $C'':=C'\int_{-\infty}^{+\infty}(1+|s|)^{-(p-d)}ds$. Combining the previous estimates we obtain
\begin{equation*}
    \textstyle\sup_{(\theta,b)\in\mathbb{S}^{d},~|\theta|\leq \delta}|r\nabla g_f(r(\theta,b))|\leq 4^{p-d}C'\delta^{p-(d+1)}\int_{-\infty}^{+\infty}|\sigma'(s)|ds+C''\epsilon_{1/4}(r).
\end{equation*}
As $G_f^{\infty}(\theta,b)=(\int_{-\infty}^{+\infty}\sigma'(s)ds)H_{u_{\theta}}(t_{\theta})/|\theta|$ for $\theta\neq 0$, we also have $|G_f^{\infty}(\theta,b)|\leq |\int_{-\infty}^{+\infty}\sigma'(s)ds|2^{p-d}C'\delta^{p-(d+1)}$ provided $|\theta|\leq \delta$.
This yields
\begin{equation*}
    \textstyle\limsup_{r\to+\infty}\sup_{(\theta,b)\in\mathbb{S}^{d}}|r\nabla g_f(r(\theta,b))-G_f^{\infty}(\theta,b)|\leq (\int_{-\infty}^{+\infty}|\sigma'(s)|ds)(4^{p-d}+2^{p-d})C'\delta^{p-(d+1)}.
\end{equation*}
As this holds for every $\delta>0$ and $p>d+1$, the first part of the result is proved.

It remains to prove the weighted estimate in the last sentence of the lemma. By homogeneity, it is enough to consider the variable $r(\theta,b)$ with $(\theta,b)\in\SS^d$ and $r\geq0$. Since $\|f\|_\infty\leq\epsilon$ and $\rho(x)\leq C(1+|x|)^{-p}$, we have $|f(x)|(1+|x|)^q\rho(x)\leq C\epsilon(1+|x|)^{-(p-q)}$. Repeating the above estimates with this new integrand instead of $\psi(x)=f(x)\rho(x)(x,1)$ gives a uniform bound
\[
    \sup_{r\geq0,\,(\theta,b)\in\SS^d}
    r\int_{\RR^d}|\sigma'(r(\langle\theta,x\rangle+b))||f(x)|(1+|x|)^q\rho(x)dx
    \lesssim \epsilon.
\]
Indeed, the hyperplane integral \eqref{eq:hyperplane_integral} becomes bounded by $C'\epsilon(1+|t|)^{-(p-q-d+1)}$, and the integral \eqref{eq:hyperplane_2} gives a factor of $4^{p-q-d+1}C\epsilon|\theta|^{p-q-d}$, which leads to a uniform bound precisely when $p>d+q$. This proves the last claim.

\end{proof}

\begin{proof}[Proof of \Cref{prop_glob_conv_sigmoid}]
    As $\sigma$ is differentiable with a bounded Lipschitz differential, \Cref{ass_glob_conv} holds. For every $f\in\mathcal{F} = L^2(\rho_x;\RR^{\dout})$, we have
    \begin{equation*}
        R'(f):x\mapsto \int_{\RR^{\dout}}\nabla_z\ell(f(x),y)d\rho_{y|x}(y).
    \end{equation*}
    If $\ell$ is the square loss, then $\nabla_z \ell(z,y)=z-y$. If $\ell$ is the cross-entropy loss, then, since $d_w=\dout=1$, $\nabla_z \ell(z,y)=e^z/(1+e^z)-y$. In both cases, our assumptions ensure that $f_*:x\mapsto \int_{\RR^{\dout}}y\,d\rho_{y|x}(y)$ has a continuous and bounded representative on the support of $\rho_x$, which can be extended to a bounded continuous function on $\RR^{\din}$. Thus, $f_{\mu}:=R'(\int_{\Omega}\Phi d\mu)$ is bounded and continuous. Since \begin{equation}\label{eq_gmu_sigmoid}
        g_{\mu}(\theta)=\int_{\RR^{\din}}\sigma(\langle\theta,(x,1)\rangle)f_{\mu}(x)d\rho_x(x),
    \end{equation}
    and $p_x(x)\leq C(1+|x|)^{-p}$ with $p>\din+2>\din +1$, we can use \Cref{lem_aux_bounded_scalar} and the first part of \Cref{lem_aux_grad_bounded_scalar}. This shows that condition \ref{conv_c1_gmu_bounded_scalar} in \Cref{prop_esc_reg_bounded_scalar} holds. For \ref{conv_dgmu_dgnu}, since $L=\sup_{y\in\RR^{\dout}}\mathrm{Lip}(\nabla_z\ell (\cdot,y))<+\infty$ (see the proof of \Cref{lem_ass_loss}) and $\Pi$ is a coupling between $\mu$ and $\nu$, then
    \begin{align*}
        \Delta&:=|\theta||\nabla g_{\mu}(\theta)-\nabla g_{\nu}(\theta)|\\
        &\textstyle\leq L\int_{\RR^{\din}}|\theta||\sigma'(\langle\theta,(x,1)\rangle)|(1+|x|)
        |(\int_{\Omega}\Phi d\mu)(x)-(\int_{\Omega}\Phi d\nu)(x)|d\rho_x(x)\\
        &\textstyle\leq L\int_{\Omega\times\Omega}\int_{\RR^{\din}}|\theta||\sigma'(\langle\theta,(x,1)\rangle)|(1+|x|)\\
        &\textstyle\qquad\qquad\qquad \times
        |\sigma(\langle\theta_1,(x,1)\rangle)w_1-\sigma(\langle\theta_2,(x,1)\rangle)w_2|
        d\rho_x(x)d\Pi(w_1,\theta_1,w_2,\theta_2)\\
        &\textstyle\leq L\int_{\Omega\times\Omega}\int_{\RR^{\din}}|\theta||\sigma'(\langle\theta,(x,1)\rangle)|(1+|x|)\\
        &\textstyle\qquad\qquad\qquad \times
        [\|\sigma\|_{\infty}|w_1-w_2|+\|\sigma'\|_{\infty}(1+|x|)|w_2||\theta_1-\theta_2|]
        d\rho_x(x)d\Pi(w_1,\theta_1,w_2,\theta_2)\\
        &\textstyle\leq 2L\big[\int_{\Omega\times\Omega}
        [\|\sigma\|_{\infty}|w_1-w_2|+\|\sigma'\|_{\infty}|w_2||\theta_1-\theta_2|]
        d\Pi(w_1,\theta_1,w_2,\theta_2)\big]\\
        &\textstyle\qquad\qquad\qquad \times
        \big[\int_{\RR^{\din}}|\theta||\sigma'(\langle \theta,(x,1)\rangle)|(1+|x|)^2d\rho_x(x)\big].
    \end{align*}
    Reasoning as in the proof of \Cref{lemma_bound_diff_Phi_mu} and taking $\Pi$ an optimal transport plan for the quadratic cost between $\mu$ and $\nu$, we obtain
    \begin{equation*}
\textstyle\Delta\lesssim W_2(\mu,\nu)\int_{\RR^{\din}}|\theta||\sigma'(\langle\theta,(x,1)\rangle)|(1+|x|)^2d\rho_x(x).
    \end{equation*}
    Taking $f:x\mapsto W_2(\mu,\nu)$, since $p_x(x)\leq C(1+|x|)^{-p}$ with $p>\din+2$, we can use the last part of \Cref{lem_aux_grad_bounded_scalar} with $q=2$ to show that the supremum of $\Delta$ over $\theta\in\RR^{d_\theta}$ converges to zero as $W_2(\mu,\nu)\to 0$, which yields \ref{conv_dgmu_dgnu}. Finally, by assumption, having condition \ref{reg_val_bounded_scalar} or \ref{reg_val_inf_bounded_scalar} allows us to apply \Cref{prop_esc_reg_bounded_scalar}. We therefore have that every $\mu\in\mathcal{P}_2(\Omega)$ that does not minimize $F$ admits an escape region. Since the initial distribution is sub-Gaussian and has full support, \Cref{thm_glob_conv} applies and we obtain the result.
\end{proof}

\subsubsection{Asymptotically positively one-homogeneous activations}
We begin this section with two technical lemmas related to conditions \ref{conv_phi_asymp_hom}-\ref{conv_dphi_asym_hom} in \Cref{prop_esc_reg_asymp_homogeneous}. Then, we proceed with the proof of \Cref{prop_glob_conv_gelu}.
\begin{lem}
    Let $\rho\in\mathcal{P}(\RR^d)$ and $\sigma\in C^0(\RR)$ be such that $\lim_{s\to+\infty}\sigma(s)/s=\alpha$ and $\lim_{s\to-\infty}\sigma(s)/s=\beta$ for some $\alpha,\beta\in \RR$. Then, for every measurable function $f$ such that $[x\mapsto f(x)(1+|x|)]\in L^1(\rho)$, the mapping
    \[
        (\theta,b)\mapsto (1/r)\int_{\RR^{d}}f(x)\sigma(r[\langle\theta,x\rangle+b]) d\rho(x)
    \]
    converges uniformly on $\mathbb{S}^{d}$ to
    \[
        (\theta,b)\mapsto \int_{\RR^{d}}f(x)\sigma_{\infty}(\langle\theta,x\rangle+b) d\rho(x),
    \]
    with $\sigma_{\infty}(s)=\alpha\max(0,s)+\beta\min(0,s)$ when $r\to+\infty$.
\end{lem}
\begin{proof}
    We first note that our assumption on $\sigma$ implies that it has at most linear growth. Our assumption on $f$ therefore ensures that the functions $x\mapsto (1/r)f(x)\sigma(r [\langle \theta,x\rangle+b])$ and $x\mapsto f(x)\sigma_{\infty}(\langle\theta,x\rangle+b)$ belong to $L^1(\rho)$. Defining $\epsilon(r):=\sup_{s\in\RR}|(1/r)\sigma(rs)-\sigma_{\infty}(s)|/(1+|s|)$, we obtain
    \begin{equation*}
        \textstyle|\int_{\RR^d}f(x)[(1/r)\sigma(r[\langle\theta,x\rangle+b])-\sigma_{\infty}(\langle\theta,x\rangle+b)]d\rho(x)|\leq 2\epsilon(r)\int_{\RR^d}|f(x)|(1+|x|)d\rho(x).
    \end{equation*}
    It remains to prove that $\epsilon(r)\to 0$ as $r\to+\infty$. We have $\epsilon(r)=\sup_{s\in\RR}|\sigma(s)-\sigma_{\infty}(s)|/(r+|s|)$. Let $\eta>0$. Then
    \begin{equation*}
        \sup_{|s|\leq \eta}\frac{|\sigma(s)-\sigma_{\infty}(s)|}{r+|s|}\leq \frac{1}{r}\sup_{|s|\leq \eta}|\sigma(s)-\sigma_{\infty}(s)|.
    \end{equation*}
    We also have
    \begin{equation*}
        \sup_{|s|\geq \eta}\frac{|\sigma(s)-\sigma_{\infty}(s)|}{r+|s|}\leq \sup_{|s|\geq \eta}\frac{|\sigma(s)-\sigma_{\infty}(s)|}{|s|}.
    \end{equation*}
    We therefore obtain
    \begin{equation*}
        \limsup_{r\to+\infty}\,\epsilon(r)\leq \sup_{|s|\geq \eta}|\sigma(s)-\sigma_{\infty}(s)|/|s|.
    \end{equation*}
    As this holds for every $\eta>0$ and the right-hand side converges to $0$ as $\eta\to +\infty$ by our assumptions on $\sigma$, we obtain the result.
\end{proof}

\begin{lem}
    Assume that $\rho\in\mathcal{P}(\RR^d)$ is absolutely continuous with respect to the Lebesgue measure and that $\sigma\in C^1(\RR)$ satisfies $\lim_{s\to+\infty}\sigma'(s)=\alpha$ and $\lim_{s\to-\infty}\sigma'(s)=\beta$ for some $\alpha,\beta\in\RR$. Then, for every measurable function $f:\RR^{d}\to\RR$ with $[x\mapsto f(x)(1+|x|)]\in L^1(\rho)$, the mapping
    \[
        g_f:(\theta,b)\mapsto \int_{\RR^{d}}f(x)\sigma(\langle \theta,x\rangle+b)d\rho(x)
    \]
    is of class $C^1$ and $(\theta,b)\mapsto \nabla g_f(r(\theta, b))$ converges uniformly to
    \[
        (\theta,b)\mapsto \int_{\RR^d}f(x)\tau_{\infty}(\langle\theta,x\rangle+b)(x,1)d\rho(x)
    \]
    on $\mathbb{S}^{d}$ when $r\to+\infty$, where $\tau_{\infty}(s)=\alpha\mathbf{1}_{(0,+\infty)}(s)+\beta\mathbf{1}_{(-\infty,0)}(s)$ for $s\neq0$ and $\tau_\infty(0)=0$.
\end{lem}
\begin{proof}
    First, our assumption on $\sigma$ implies that its derivative is bounded and, for every $\eta>0$, $\epsilon_{\eta}(r):=\sup_{|s|\geq \eta}|\sigma'(rs)-\tau_{\infty}(s)|\to 0$ when $r\to\infty$. The gradient of the integrand with respect to $(\theta,b)$ is $f(x)\sigma'(\langle \theta,x\rangle+b)\rho(x)(x,1)$, which is dominated by $\|\sigma'\|_{\infty} |f(x)|(1+|x|)\rho(x)$. This last function is integrable, which shows that $g_f$ is of class $C^1$ and that, for every $(\theta,b)\in\RR^{d+1}$,
    \begin{equation*}
        \nabla g_f(\theta,b)=\int_{\RR^{d}}f(x)\sigma'(\langle\theta,x\rangle+b)\rho(x)(x,1)dx.
    \end{equation*}
    For every $\eta>0$, we have
    \begin{equation*}
         \int_{\{|\langle\theta,x\rangle+b|\geq \eta\}}|f(x)\rho(x)(x,1)||\sigma'(r[\langle\theta,x\rangle+b])-\tau_{\infty}(\langle\theta,x\rangle+b)|dx\leq \epsilon_{\eta}(r)\int_{\RR^{d}}|f(x)\rho(x)(x,1)|dx.
    \end{equation*}
    Let $R\geq 1$ and $(\theta,b)\in\mathbb{S}^d$. If $|\theta|\leq 1/(2R)$, then, since $|\theta|^2+b^2=1$, the set ${\{|\langle\theta,x\rangle+b|\leq \eta\}\cap B_{d}(0,R)}$ is empty as soon as $\eta< (\sqrt{3}-1)/2$. If $|\theta|\geq 1/(2R)$, then $\{|\langle\theta,x\rangle+b|\leq \eta\}\cap B_{d}(0,R)$ is contained, up to a rotation and a translation, in $[-a,a]\times B_{d-1}(0,R)$ with $a=2R\eta$. In this case,
    \begin{equation*}
         \int_{\{|\langle\theta,x\rangle+b|\leq \eta\}\cap B_{d}(0,R)}|f(x)\rho(x)(x,1)||\sigma'(r[\langle\theta,x\rangle+b])-\tau_{\infty}(\langle\theta,x\rangle+b)|dx\leq (\max(|\alpha|,|\beta|)+\|\sigma'\|_{\infty})\Psi(\eta,R)
    \end{equation*}
    with
    \begin{equation*}
        \Psi(\eta,R)\eqdef\textstyle\sup\,\{\int_{A}|f(x)\rho(x)(x,1)|dx\,\rvert\, A\subset\RR^{d},~|A|\leq 4R\eta\mathcal{L}^{d-1}(B_{d-1}(0,R))\}.
    \end{equation*}
    Finally, $\{|\langle\theta,x\rangle+b|\leq \eta\}\setminus B_{d}(0,R)$ is included in $\RR^{d}\setminus B_{d}(0,R)$, and
    \begin{align*}
         &\int_{\RR^{d}\setminus B_{d}(0,R)}|f(x)\rho(x)(x,1)|
         |\sigma'(r[\langle\theta,x\rangle+b])-\tau_{\infty}(\langle\theta,x\rangle+b)|dx\\
         &\qquad\leq (\max(|\alpha|,|\beta|)+\|\sigma'\|_{\infty})
         \int_{\RR^{d}\setminus B_{d}(0,R)}|f(x)\rho(x)(x,1)|dx.
    \end{align*}
    Combining the three terms and using the integrability of $x\mapsto f(x)\rho(x)(x,1)$, we obtain that
    \begin{equation*}
        \Delta:=\limsup_{r\to+\infty}\sup_{(\theta,b)\in\mathbb{S}^{d}}\textstyle|\int_{\RR^{d}}(\sigma'(r[\langle\theta,x\rangle+b])-\tau_{\infty}(\langle\theta,x\rangle+b))f(x)\rho(x)(x,1)dx|
    \end{equation*}
    satisfies $\Delta\leq  (\max(|\alpha|,|\beta|)+\|\sigma'\|_{\infty})(\Psi(\eta,R)+\int_{\RR^{d}\setminus B_{d}(0,R)}|f(x)\rho(x)(x,1)|dx)$ for every $\eta\in (0,(\sqrt{3}-1)/2)$ and $R\geq 1$. As the first term converges to $0$ as $\eta\to 0$ for every $R>0$ and the second term converges to $0$ as $R\to +\infty$, the result follows by first sending $\eta$ to $0$ and then $R$ to $+\infty$.
\end{proof}

\begin{proof}[Proof of \Cref{prop_glob_conv_gelu}]
    As $\sigma \in C^{1,1}(\RR)$ and $\sigma'$ is bounded by our limit assumptions, \Cref{ass_glob_conv} holds. For every $(w,f)\in\RR^{d_w}\times L^2(\rho_x;\RR^{\dout})$, we have
    \begin{equation*}
        \langle \phi(\theta)w,f\rangle_{\mathcal{F}}=\int_{\RR^{\din}} \sigma(\langle \theta,(x, 1)\rangle) \langle f(x),w\rangle d\rho_x(x).
    \end{equation*}
    Since $\rho_x$ has a finite moment of order $2$, we have $(1+|\cdot|)\in L^2(\rho_x)$, so that $[x\mapsto (1+|x|)\langle f(x),w\rangle]\in L^1(\rho_x)$ and the two lemmas above apply.\footnote{Note that when $r\to +\infty$, the assumptions $\sigma'(r)\to a$ and $\sigma'(-r)\to b$ imply $\sigma(r)/r\to a$ and $\sigma(-r)/(-r)\to b$ by l'H\^opital's rule (equivalently, $\sigma(-r)/r\to-b$).} This shows that conditions \ref{conv_phi_asymp_hom} and \ref{conv_dphi_asym_hom} hold. Condition \ref{density_phi_asymp_hom} follows from $a\neq b$ and standard density results for two-layer Leaky ReLU networks. Indeed, according to \cite{hornik1991approximation}, any bounded continuous nonconstant function leads to universal approximation results in $L^p$, and given $\sigma_{\infty}:x\mapsto a\max(0,x)+b\min(0,x)$, the mapping $x\mapsto \sigma_{\infty}(x) - \sigma_{\infty}(x-1)$ is bounded and fits in the framework of Hornik. By assumption, the mapping $h_{\mu}^{\infty}$ has a nonzero regular value, which allows us to apply \Cref{prop_esc_reg_asymp_homogeneous} to obtain that every $\mu\in\mathcal{P}_2(\Omega)$ that does not minimize $F$ admits an escape region. Finally, since the initial distribution is assumed to be sub-Gaussian and to have full support, \Cref{thm_glob_conv} applies: if the Wasserstein gradient flow converges in $W_2$, its limit is a global minimizer of $F$.
\end{proof}

\subsection{Attention layers}\label{appendix_attention}
Throughout this subsection, we denote by $\sigma:z\in\RR^n\mapsto (e^{z_i}/(\sum_{j=1}^n e^{z_j}))_{i=1}^n\in \RR^n$ the softmax function. We start with the proof that, under a finite moment assumption on the distribution of the input contexts $\rho_X$, the mapping $\phi$ is sufficiently regular.
\begin{proof}[Proof of \Cref{lem_reg_transfo}]
    As $\psi(A)(X)$ is a convex combination of the $(x_i)_{1\leq i\leq n}$, using Jensen's inequality we obtain
    \begin{equation*}
        |\psi(A)(X)|^2\leq \sum\limits_{i=1}^n \frac{e^{\langle Ax_n,x_i\rangle}}{\sum\limits_{j=1}^n e^{\langle Ax_n,x_j\rangle}}|x_i|^2\leq \underset{1\leq i\leq n}{\max}~|x_i|^2.
    \end{equation*}
    Therefore, provided $\int_{\RR^{n\times d}}\max_{1\leq i\leq n}|x_i|^2 d\rho_X(X)<+\infty$, we have $\psi(A)\in \calH$ with $\calH\eqdef L^2(\rho_X;\RR^d)$ for every $A\in\RR^{d\times d}$ and $\sup_{A}\|\psi(A)\|_{\calH}<+\infty$.

    We next note that $\psi(A)(X)=X^T\sigma(XAx_n)$, where $\sigma:z\in\RR^n\mapsto (e^{z_i}/(\sum_{j=1}^n e^{z_j}))_{i=1}^n\in \RR^n$ is the softmax function. We claim that $\psi$ is differentiable with differential $d\psi$ given by
    \begin{equation*}
    \begin{aligned}
    \relax[d\psi(A).B](X)&=X^T(d\sigma(XAx_n)\cdot (XBx_n))\\
    &=\sum\limits_{i=1}^n \sigma_i(XAx_n)\bigg(\langle Bx_n,x_i\rangle-\sum\limits_{j=1}^n \sigma_j(XAx_n)\langle Bx_n,x_j\rangle\bigg)x_i.
    \end{aligned}
    \end{equation*}
    We have
    \begin{equation*}
        \begin{aligned}
            \Delta(B)^2&\eqdef \|\psi(A+B)-\psi(A)-d\psi(A)\cdot B\|_{\calH}^2\\
            &=\int_{\RR^{n\times d}}|X^T(\sigma(X(A+B)x_n)-\sigma(XAx_n)-d\sigma(XAx_n)\cdot(XBx_n))|^2 d\rho_X(X)\\
            &\leq \int_{\RR^{n\times d}}\bigg[\underset{1\leq i\leq n}{\mathrm{max}}|x_i|^2\bigg]\Bigg[\sum\limits_{i=1}^n|\sigma_i(X(A+B)x_n)-\sigma_i(XAx_n)-d\sigma_i(XAx_n)\cdot(XBx_n)|\Bigg]^2 d\rho_X(X).
        \end{aligned}
    \end{equation*}
    Using $\sigma_i(z+h)-\sigma_i(z)-d\sigma_i(z)\cdot h=\int_{0}^1(1-t)d^2\sigma_i(z+th)\cdot(h,h)dt$ for every $z,h\in\RR^n$ and \Cref{lemma_bound_d2sigma} below, we obtain
    \begin{equation*}
        \begin{aligned}
            \Delta(B)^2&\leq \int_{\RR^{n\times d}}\bigg[\underset{1\leq i\leq n}{\mathrm{max}}|x_i|^2\bigg]\bigg[\int_{0}^1(1-t)|d^2\sigma(X(A+tB)x_n)\cdot(XBx_n,XBx_n)|_1dt\bigg]^2 d\rho_X(X)\\
            &\leq 9\int_{\RR^{n\times d}}\bigg[\underset{1\leq i\leq n}{\mathrm{max}}|x_i|^2\bigg] |XBx_n|_{\infty}^4 d\rho_X(X)\leq 9 \|B\|^4 \int_{\RR^{n\times d}}\underset{1\leq i\leq n}{\mathrm{max}}|x_i|^{10} d\rho_X(X).
        \end{aligned}
    \end{equation*}
    This shows that $\psi$ is differentiable with differential $d\psi$. A similar computation shows that $d\psi$ is bounded and Lipschitz. Arguing similarly, one can also obtain that $\psi$ is twice differentiable with a bounded Lipschitz second differential provided $\rho_X$ has a finite moment of order $14$.
\end{proof}

\begin{lem}\label{lemma_bound_d2sigma}
    For every $z,h\in\RR^n$, we have $|d\sigma(z)\cdot h|_1\leq 2|h|_{\infty}$ and $|d^2\sigma(z)\cdot(h,h)|_1\leq 6|h|_{\infty}^2$.
\end{lem}
\begin{proof}
    We have $d\sigma(z)\cdot h=\sigma(z)\odot h - \langle \sigma(z),h\rangle \sigma(z)$. Using the fact that the components of $\sigma(z)$ are positive and sum to one, we obtain that
    \begin{equation*}
        \begin{aligned}
        |d\sigma(z)\cdot h|_1&=\sum\limits_{i=1}^n \sigma_i(z)|h_i-\langle \sigma(z),h\rangle|\leq \underset{1\leq i \leq n}{\mathrm{max}}~|h_i-\langle \sigma(z),h\rangle|\leq 2|h|_{\infty}.
        \end{aligned}
    \end{equation*}

    For every $1\leq i\leq n$, we have
    \begin{equation*}
            d^2\sigma_i(z)\cdot(h,h)=\sigma_i(z)\big[(h_i-\langle h,\sigma(z)\rangle)^2-\langle \sigma(z),h\odot h\rangle+\langle \sigma(z),h\rangle^2\big].
    \end{equation*}
    Using again that the components of $\sigma(z)$ are positive and sum to one, we obtain that
    \begin{equation*}
        \begin{aligned}
            |d^2\sigma(z)\cdot(h,h)|_1&\leq \underset{1\leq i\leq n}{\mathrm{max}}~|h_i-\langle \sigma(z),h\rangle|^2+|\langle \sigma(z),h\odot h\rangle|+\langle \sigma(z),h\rangle^2\\
            &= \underset{1\leq i\leq n}{\mathrm{max}}~\langle \sigma(z),h-h_i\mathbf{1}\rangle^2+|\langle \sigma(z),h\odot h\rangle|+\langle \sigma(z),h\rangle^2\\
            &\leq \underset{1\leq i,j\leq n}{\mathrm{max}}~|h_i-h_j|^2+2|h|_{\infty}^2\\
            &\leq 6|h|_{\infty}^2.
        \end{aligned}
    \end{equation*}
\end{proof}

We now turn to the proof of \Cref{prop_conv_attention}. Our strategy is to restrict our study to contexts whose tokens all lie in a given ball by using that $\rho_X$ has finite second-order moments and the rest of the integrand is bounded. Then, given $\eta>0$, we directly bound the integral on the region of the input space where $\Sigma(z)=\{i\}$ for some $i$ and $z_{\max}-z_j=z_i-z_j\geq \eta$ for every $j\neq i$, where $z=(\langle Ax_n,x_i\rangle)_{1\leq i\leq n}$. The remaining region is contained in the region where there exists $i\neq j$ such that $|z_i-z_j|\leq \eta$. This is a region of small width around the zero set of some nontrivial polynomial, and we show that its measure tends to zero when $\eta\to 0$ uniformly in $A\in\mathbb{S}^{d^2-1}$.
\begin{proof}[Proof of \Cref{prop_conv_attention}]
    If $n=1$, then $\psi_r(A)=\psi_\infty(A)$ for every $r>0$ and $A\in\mathbb{S}^{d^2-1}$, so the result is immediate. We therefore assume $n\geq2$.
    Throughout this proof, we set $z=(\langle Ax_n,x_i\rangle)_{1\leq i\leq n}\in\RR^n$. We note that
    \begin{equation*}
        \begin{aligned}
        |\psi_{r}(A)(X)-\psi_{\infty}(A)(X)|&=\Bigg|\sum\limits_{i\notin \Sigma(z)}\sigma_{r,i}(z)\Bigg(x_i-\frac{1}{\#\Sigma(z)}\sum\limits_{j\in \Sigma(z)} x_j\Bigg)\Bigg|\\
        &\leq 2\bigg[\underset{1\leq j\leq n}{\max}|x_j|\bigg]\sum\limits_{i\notin \Sigma(z)}\sigma_{r,i}(z).\\
        \end{aligned}
    \end{equation*}

    Given $R>0$ and $\eta>0$, we split the input space $\RR^{n\times d}$ into three regions
    \begin{equation*}
        \begin{aligned}
            \mathcal{R}_1&=\{X\in\RR^{n\times d}\,\rvert\, \exists i\in\{1,...,n\},~ |x_i|>R\},\\
            \mathcal{R}_2&=\{X\in\RR^{n\times d}\setminus \mathcal{R}_1\,\rvert\, \Sigma(z) \mathrm{~is~a~singleton~and~}z_{\max}-z_i\geq \eta\mathrm{~for~every~}i\notin\Sigma(z)\},\\
            \mathcal{R}_3&=\RR^{n\times d}\setminus (\mathcal{R}_1\cup \mathcal{R}_2).
        \end{aligned}
    \end{equation*}

    \emph{Bound on $\mathcal{R}_1$.} Since $\sum_{i\notin\Sigma(z)}\sigma_{r,i}(z)\leq \sum_{i=1}^n \sigma_{r,i}(z)=1$, we have
    \begin{equation*}
        \int_{\mathcal{R}_1}|\psi_r(A)(X)-\psi_{\infty}(A)(X)|^2 d\rho_X(X)\leq 4\int_{\mathcal{R}_1} \bigg[\max_{1\leq i\leq n}|x_i|^2\bigg]d\rho_X(X).
    \end{equation*}
    By our finite moment assumption on $\rho_X$, this quantity therefore goes to $0$ as $R\to+\infty$.

    \emph{Bound on $\mathcal{R}_2$.} If $\Sigma(z)$ is a singleton and $z_{\max}-z_i\geq \eta$ for every $i\notin\Sigma(z)$, then, for every $i\notin\Sigma(z)$,
    \begin{equation*}
        \sum\limits_{i\notin\Sigma(z)}\sigma_{r,i}(z)=\frac{\sum_{i\notin\Sigma(z)}e^{-r(z_{\max}-z_i)}}{\sum_{j=1}^n e^{-r(z_{\max}-z_j)}}=\frac{1}{1+(\sum_{i\notin\Sigma(z)}e^{-r(z_{\max}-z_i)})^{-1}}\leq \frac{1}{1+e^{r\eta}/(n-1)}.
    \end{equation*}
    As a consequence, we obtain
    \begin{equation*}
        \int_{\mathcal{R}_2}|\psi_r(A)(X)-\psi_{\infty}(A)(X)|^2 d\rho_X(X)\leq \frac{4R^2}{(1+e^{r\eta}/(n-1))^2}.
    \end{equation*}

    \emph{Bound on $\mathcal{R}_3$.} We notice that $\mathcal{R}_3$ is included in
    \[
        \mathcal{R}'_3=\bigcup_{i\neq j}\mathcal{R}'_{3,ij},
        \qquad
        \mathcal{R}'_{3,ij}:=\{X\in \RR^{n\times d}\setminus \mathcal{R}_1\,\rvert\,|z_i-z_j|\leq \eta\}.
    \]
    It is therefore enough to estimate each of the finitely many pairwise regions $\mathcal{R}'_{3,ij}$ and then sum the resulting bounds. As $\|A\|=1$, the set $\{X\,\rvert\,Ax_n=0\}$ has Lebesgue measure zero. If $X\in\mathcal{R}'_{3,ij}$ and $Ax_n\neq0$, then, after a rotation that sends $Ax_n/|Ax_n|$ to a coordinate vector, $x_i-x_j$ belongs to the cylinder
    \[
        B_{d-1}(0,2R)\times[-\eta/|Ax_n|,\eta/|Ax_n|].
    \]
    The radius is $2R$ because $x_i,x_j\in B_d(0,R)$.

    First assume that $i\neq n$ and $j\neq n$. Denoting by $\mathcal{L}^d$ Lebesgue measure on $\RR^d$ and setting $\epsilon=\eta/|Ax_n|$, we obtain
    \[
        \mathcal{L}^{nd}(\mathcal{R}'_{3,ij})
        \leq [\mathcal{L}^d(B_d(0,R))]^{n-3}
        \int_{B_d(0,R)}\mathcal{L}^{2d}(E_{x_n})\,dx_n,
    \]
    where
    \[
        E_{x_n}=\{(y_1,y_2)\in(B_d(0,R))^2\,\rvert\,
        y_1-y_2\in B_{d-1}(0,2R)\times[-\epsilon,\epsilon]\}.
    \]
    This set is contained in the image, under the unit-Jacobian map $(x,z)\mapsto(x,x+z)$, of
    \[
        B_d(0,R)\times\bigl(B_{d-1}(0,2R)\times[-\epsilon,\epsilon]\bigr).
    \]
    Since the largest singular value of $A$ is at least $1/\sqrt d$, rotating the $x_n$ variable along a corresponding right singular vector yields
    \[
        \int_{B_d(0,R)}\mathcal{L}^{2d}(E_{x_n})\,dx_n
        \lesssim R^{2d-1}\int_{B_d(0,R)}\min(R,\eta/|y_1|)\,dy.
    \]
    For fixed $R$, the last integral converges to zero as $\eta\to0$ by dominated convergence, uniformly over $A\in\mathbb S^{d^2-1}$.
    More explicitly, if $u_A$ is a right singular vector associated with the largest singular value of $A$, then $|Ax|\geq d^{-1/2}|x\cdot u_A|$ and, after rotating $u_A$ to the first coordinate direction,
    \[
        \int_{B_d(0,R)}\min\!\left(R,\frac{\eta}{|Ax|}\right)dx
        \leq C_{d,R}\int_0^R\min\!\left(R,\frac{\sqrt d\,\eta}{s}\right)ds
        \leq C'_{d,R}\eta(1+|\log\eta|)
    \]
    for all sufficiently small $\eta$. This makes the uniformity in $A$ explicit.

    If one of the two indices equals $n$, say $i=n$ and $j=1$, the same argument gives
    \[
        \mathcal{L}^{nd}(\mathcal{R}'_{3,n1})
        \leq [\mathcal{L}^d(B_d(0,R))]^{n-2}
        \int_{B_d(0,R)}\mathcal{L}^{d}(E_{x_n})\,dx_n,
    \]
    where
    \[
        E_{x_n}=\{x_1\in B_d(0,R)\,\rvert\,
        x_1-x_n\in B_{d-1}(0,2R)\times[-\epsilon,\epsilon]\}.
    \]
    The same singular-value estimate, with the explicit bound above, shows that this term is also $O_{d,R}(\eta(1+|\log\eta|))$ uniformly in $A$. Summing over the at most $n(n-1)$ ordered pairs proves that $\mathcal{L}^{nd}(\mathcal{R}'_3)\to0$ uniformly in $A$ as $\eta\to0$.

    \emph{Final step.} Take, for instance, $\eta=r^{-1/2}$. The integral over $\mathcal{R}_2$ converges to zero as $r\to+\infty$. On $\mathcal{R}_3$, the integrand is bounded by $4R^2$, and the bounded-density assumption together with the preceding uniform volume estimate shows that the corresponding integral also converges to zero. Thus, the limit is bounded by the integral over $\mathcal{R}_1$. Since this holds for every $R>0$ and the latter integral converges to zero as $R\to+\infty$, the result follows.
\end{proof}

\begin{proof}[Proof of \Cref{prop_glob_conv_attention}]
    By \Cref{cor_ass_glob_conv_attention}, we have that \Cref{ass_glob_conv} holds. By assumption, for every $\mu\in\mathcal{P}_2(\Omega)$ that does not minimize $F$, $g_\mu$ has a nonzero regular value. Applying the analog of \Cref{prop_esc_reg_bounded_scalar} for $d_\theta=1$ proved in \Cref{appendix_bounded_scalar} (see \Cref{rem_bounded_scalar_scalar}), we therefore obtain the existence of an escape region for every such $\mu$. Applying \Cref{thm_glob_conv} therefore shows that, if the Wasserstein gradient flow converges in $W_2$, its limit is a global minimizer.
\end{proof}
\end{document}